\documentclass[11pt,a4paper,reqno]{amsart}
\usepackage{mathrsfs}
\usepackage{bigints}
\usepackage{amsmath,amsfonts,verbatim}
\usepackage{mathtools}
\usepackage{latexsym}
\usepackage{amssymb,leftidx}
\usepackage{extarrows}
\usepackage{overpic}
\usepackage{color}
\usepackage{enumitem}
\usepackage{epsfig}
\usepackage{subfigure}
\usepackage{tikz}
\usepackage{bbm}
\usepackage{tikz, caption}
\usepackage{physics}
\usepackage[font=small,labelfont=bf]{caption}
 \mathtoolsset{showonlyrefs=true}
\usepackage[colorlinks=true,backref=page]{hyperref}
\hypersetup{urlcolor=blue, citecolor=blue}
\usepackage{geometry}
\usepackage{appendix}
\usepackage{amssymb}
\usepackage{mathrsfs}
\usepackage{amscd}
\usepackage{cite}
\newcommand\R{\mathbb{R}}
\newcommand\C{\mathbb{C}}

\usepackage{graphicx}
\usepackage{float}

\numberwithin{equation}{section}
\newtheorem{proposition}{Proposition}[section]

\newtheorem{lemma}{Lemma}[section]
\newtheorem{theorem}{Theorem}[section]
\newtheorem{corollary}{Corollary}[section]
\newtheorem{assumption}{Hypothesis}[section]
\newtheorem{remark}{Remark}[section]

\begin{document}

\title[MAGNETIC UNCERTAINTY IN VARIABLE GEOMETRY]{MAGNETIC UNCERTAINTY IN VARIABLE GEOMETRY}

\author{Luca Fanelli}
\address{Luca Fanelli\newline\indent Ikerbasque, Basque Foundation for Science; Departamento de Mathem\'aticas, Universidad del Pa\'is Vasco/Euskal Herriko Unibertsitatea PV/EHU 48940 Leioa, Spain; BCAM-Basque Center for Applied Mathematics,48009 Bilbao, Spain}
\email{lfanelli@bcmath.org}
\author{Yilin Song}
\address{Yilin Song
	\newline \indent The Graduate School of China Academy of Engineering Physics,
		Beijing 100088,\ P. R. China}
\email{songyilin21@gscaep.ac.cn}
	
\author[Y. Wang]{Ying Wang}
\address{Ying Wang\newline\indent 
  BCAM - Basque Center for Applied Mathematics, 48009, Bilbao, Spain}
\email{ywang@bcamath.org}

\author{Jiqiang Zheng}
\address{Jiqiang Zheng
		\newline \indent Institute of Applied Physics and Computational Mathematics,
		Beijing, 100088, China.
		\newline\indent
		National Key Laboratory of Computational Physics, Beijing 100088, China}
\email{zheng\_jiqiang@iapcm.ac.cn, zhengjiqiang@gmail.com}

\author{Ruihan Zhou}
\address{Ruihan Zhou
	\newline \indent Institute of Applied Physics and Computational Mathematics,
	Beijing, 100088, China.}
\email{19210180085@fudan.edu.cn}

\begin{abstract}
In this paper, we study Hardy-type uncertainty principles and unique continuation properties for linear covariant Schr\"odinger equations with variable coefficients in the presence of bounded electric and magnetic potentials. Under suitable smallness assumptions on the leading coefficients, we prove that any solution exhibiting super-quadratic exponential decay at two distinct times must vanish identically. Under an additional structural assumption on the coefficient matrix $G$, we further establish a Hardy-type result at the quadratic exponential scale. We also obtain an analogous uniqueness result for the heat equation with variable-coefficient magnetic perturbations.

Our results unify and extend previous works in two directions: they recover the constant-coefficient covariant case treated by Barc\'elo--Fanelli--Guti\'errez--Ruiz--Vilela when $G=I$, and the variable-coefficient non-magnetic case considered by Federico--Li--Yu when $A=0$. The proofs combine logarithmic convexity arguments with Carleman estimates adapted to variable-coefficient covariant Schr\"odinger and parabolic flows. Although our approach follows the general strategy introduced by Escauriaza--Kenig--Ponce--Vega, substantial new difficulties arise from the interaction between the variable metric and the magnetic structure, which requires new weight functions and refined commutator estimates.
\end{abstract}

\maketitle

\begin{center}
\begin{minipage}{130mm}
\end{minipage}
\end{center}

\section{Introduction}

In this paper, we investigate uniqueness and uncertainty principles for solutions to covariant Schr\"odinger equations with variable coefficients and electromagnetic potentials. More precisely, we consider the Cauchy problem
\begin{align}\label{eq1}
    \begin{cases}
        i\partial_tu+H_{A,G}u=0,&(t,x)\in\R\times\R^d,\\
        u(x,0)=u_0(x),
    \end{cases}
\end{align}
where $u:I\times\R^d\to\mathbb C$ is the solution and $H_{A,G}$ is the variable-coefficient Schr\"odinger operator with electromagnetic potentials $A(x)\in L^\infty(\mathbb{R}^d;\mathbb{R}^d)\cap C^1_{loc}(\mathbb{R}^{d};\mathbb{R}^{d})$ and $V=V(t,x):\R^{d+1}\to\mathbb C$. More precisely, $H_{A,G}$ is given by
\begin{align*}
    H_{A,G}u=\sum_{j,k=1}^d D_j(g_{jk}D_ku)+V(t,x)u,
\end{align*}
where $D_j:=\partial_j-iA_j$ denotes the covariant derivative. We write $G:=(g_{jk})_{j,k=1}^d$ for the symmetric variable-coefficient matrix. Throughout the paper we assume that $G\in C_b^3(\R^d;\R^{d\times d})$ is real-valued, symmetric, and uniformly elliptic, namely
\begin{align}\label{uniformly elliptic}
    \lambda|\xi|^{2}\leq\sum_{j,k}^{n}g_{jk}\xi_{j}\xi_{k}\leq \Lambda|\xi|^{2}, \qquad \lambda,\Lambda>0.
\end{align}
When $G=Id$, the operator $H_{A,G}$ reduces to the classical electromagnetic Schr\"odinger operator $\Delta_A=(\nabla-iA)^2+V(t,x)$.

Associated with the magnetic potential we use the magnetic field matrix $B=(B_{jk})_{j,k=1}^{d}$, defined by
\[
B_{jk}:=\partial_{j}A_{k}-\partial_{k}A_{j},
\]
so that $[D_j,D_k]=-iB_{jk}$.

Our goal is to understand to what extent Hardy-type uncertainty principles and unique continuation properties, which are well understood in the constant-coefficient setting, remain valid in the presence of both variable geometry and magnetic effects. This problem involves three distinct sources of difficulty: the loss of translation invariance induced by the variable metric \(G\), the nontrivial commutation relations generated by the magnetic potential, and the coupling between the geometry and the magnetic field \(B\). This coupling modifies the associated Hamiltonian dynamics and gives rise to additional mixed terms in the commutator calculations, whose control presents a further analytical challenge.
\

The starting point of our analysis is the classical uncertainty principle due to Hardy \cite{Hardy}, which states that a function and its Fourier transform cannot both decay too rapidly unless the function is trivial.

\begin{theorem}[Hardy's uncertainty principle]

If $f(x)=\mathcal{O}(e^{-\alpha|x|^2})$ and its Fourier transform has Gaussian decay $\widehat{f}(\xi)=\mathcal{O}(e^{-\beta|\xi|^2})$ with $\alpha\beta>\frac{1}{16}$, then $f=0$.
\end{theorem}

For the free Schr\"odinger evolution, one has the representation
\begin{align*}
   u(t)\triangleq e^{it\Delta }f=\frac{1}{(2\pi t)^\frac d2}\int_{\R^d}e^{i\frac{|x-y|^2}{4t}}f(y)\,\dd y=(2\pi t)^{-\frac d2}e^{i\frac{|x|^2}{4t}}\mathcal{F}\big(e^{i\frac{|\cdot|^2}{4t}}f\big)(\frac{x}{2t}).
\end{align*}
Thus, by applying Hardy's theorem, one sees that if $\|e^{\alpha|x|^2}u(0)\|_{L^2}<\infty$, $\|e^{\beta|x|^2}u(T)\|_{L^2}<\infty$, $\alpha\beta>\frac{1}{16T^2}$, then $u:=e^{it\Delta}f\equiv0$. This observation gives a dynamical interpretation of Hardy's principle for dispersive equations.

The first nonlinear dispersive version of the uncertainty principle was proved by Bourgain \cite{Bourgain-IMRN}, who showed that if a solution to the Schr\"odinger equation has compact support in a short time interval $t\in[0,T]$, then $u\equiv0$. This result was later strengthened by Kenig--Ponce--Vega \cite{KPV-CPAM}, who proved uniqueness $u_1\equiv u_2$ if there exists a convex cone $\Gamma$ such that $u_1(x,0)=u_2(x,0)$ and $u_1(x,1)=u_2(x,1)$ for all $x\notin\Gamma+y_0$ and some fixed $y_0\in\R^d$.

Later, in a series of works by Escauriaza--Kenig--Ponce--Vega \cite{EKPV-CPDE,EKPV-JFA,EKPV-2008MRL,EKPV-JEMS,EKPV-Duke}, the following unique continuation result was established for the linear Schr\"odinger equation with electric potential.

\begin{theorem}[Unique continuation property]\label{Th-EKPV}
Let $d\geq3$ and $T>0$ bounded. Assume that $u\in L^\infty([0,T],L^2(\R^d))\cap L^2([0,T],H^1(\R^d))$ is a solution to 
\begin{align*}
    i\partial_tu+\Delta u-V(t,x)u=0.
\end{align*}
Under suitable boundedness assumptions on $V$, if
$$
\big\|e^{\alpha|x|^2}u(x,0)\big\|_{L^2(\R^d)}+\big\|e^{\beta|x|^2}u(x,T)\big\|_{L^2(\R^d)}<\infty,
$$
then $u\equiv0$ provided $\alpha\beta>\frac{1}{16T^2}$.
\end{theorem}

\begin{remark}
In \cite{EKPV-Duke}, the authors obtained an $L^2$ exponential weighted estimate, but did not derive $u\equiv0$ in the endpoint case $\alpha\beta=\frac{1}{16T^2}$. However, in \cite{EKPV-CMP16}, they showed that the heat equation still enjoys unique continuation at the endpoint $\alpha\beta=\frac{1}{16T^2}$.
\end{remark}

Similar results to Theorem \ref{Th-EKPV} for Schr\"odinger equations with gradient perturbations were studied in Dong--Staubach \cite{Dong-Staubach} and Escauriaza--Fanelli--Vega \cite{Indiana}. For the variable-coefficient case $H=\operatorname{div}(G\nabla\cdot)$, Federico--Li--Yu \cite{Yu-CCM} proved that if the solution exhibits cubic exponential decay at times $t=0$ and $t=1$ with large parameter $\sigma$,
$$
\big\|e^{\sigma|x|^3}u(x,0)\big\|_{L^2(\R^d)}+\big\|e^{\sigma|x|^3}u(x,1)\big\|_{L^2(\R^d)}<\infty,
$$
then $u\equiv0$ in $t\in[0,1]$.

For Schr\"odinger equations with electromagnetic potentials, Fanelli--Vega \cite{Fanelli-Vega} established magnetic virial identities and proved Strichartz estimates under suitable decay assumptions on the magnetic field $B$. This was later extended to the magnetic Dirac equation in \cite{BDF-JMPA}. For the virial identity corresponding to variable-coefficient covariant Schr\"odinger flows with nontrapping matrix, we refer to \cite{CD-MathAnn}. Based on magnetic virial identities, Barc\'elo--Fanelli--Guiterrez--Ruiz--Vilela \cite{BFGRV-JFA} and Cassano--Fanelli \cite{Cassano-Fanelli} proved unique continuation results under the condition $\alpha\beta>\frac{1}{16T^2}$ together with suitable smallness assumptions on the magnetic field. Related unique continuation results for other dispersive equations can be found in \cite{BCF-NA,CFL-CPDE,CEKPV-Indiana,KPPV-CPDE,KPV-MRL,KPV-14,KPV-JFA}. For a broader overview of uncertainty principles, we refer to the survey by Bertolin--Mallinikova \cite{Mallinicova}.

The literature on uncertainty principles for Schr\"odinger equations has thus developed in several directions, including the presence of electric potentials, magnetic perturbations, and variable coefficients. However, results that combine variable metrics and magnetic effects are still rather limited. In particular, the methods developed for constant-coefficient magnetic equations and those for variable-coefficient non-magnetic equations cannot be directly combined, due to the interaction between the geometry and the magnetic structure.

The main contribution of this paper is to provide a unified treatment of these two settings. More precisely, we show that Hardy-type uncertainty principles and unique continuation properties continue to hold under suitable smallness assumptions on both the metric and the magnetic field. Our results recover the constant-coefficient covariant case when $G=I$, and the variable-coefficient non-magnetic case when $A=0$, while allowing the interaction of both effects in a perturbative regime.

\subsection{The study of variable-coefficient covariant Schr\"odinger flow}

When $A=0$, equation \eqref{eq1} reduces to the variable-coefficient Schr\"odinger flow. The influence of geometry on the behavior of solutions to linear and nonlinear partial differential equations has been extensively studied. Sharp Strichartz estimates for this model (or, more generally, on manifolds) have been investigated by many authors under various assumptions on the geometry or on the matrix $G$. We refer to \cite{Staffilani-Tataru,Bouclet-Tzvetkov,Burq-Gerard-Tzvetkov,Hassell} for more details. Under nontrapping assumptions and asymptotically flat geometry, the Strichartz estimate is scaling invariant. The local smoothing estimate was established in Doi \cite{Doi} and Bouclet--Burq \cite{BB-Duke}. As an application, local well-posedness is by now standard via Strichartz estimates and Picard iteration. For variable-coefficient covariant Schr\"odinger flows, Strichartz estimates were investigated by Mizutani \cite{Mizutani-1,Mizutani-2}. We also refer to \cite{Cacciafesta-JST,Cacciafesta-SIAM} for uniform resolvent estimates related to this model.
	
\subsection{Main results}

Motivated by these developments, we investigate unique continuation properties for solutions to \eqref{eq1}. Our first result establishes unique continuation under suitable smallness assumptions on the metric and electromagnetic coefficients. Before stating the theorem, we introduce the assumptions on $A$, $V$, and $G$.

\begin{assumption}\label{assm1}
Assume that the electromagnetic potentials $A$ and the matrix $G$ satisfy the following smallness condition:
\begin{align}
\sup_{x\in\mathbb R^d}\langle x\rangle^{3}(|\nabla G|+|\nabla^{2}G|+|\nabla^{3}G|)& \leq\varepsilon_{0}\ll1,\label{wyy1}\\
 \|(Gx)^{\top}B\|_{L^{\infty}}^{2}:=\frac{M_{A}}{4} <\infty. \label{wyy2}
\end{align}

\end{assumption}

We can now state our first theorem.

\begin{theorem}\label{thm1}
Let $u\in L^\infty([0,1], L^{2}(\mathbb{R}^{d}))$ be a solution to \eqref{eq1}
with $V(t,x)=V_1(x)$ real valued satisfying $\|V_1\|_{L_x^\infty}=M_1<\infty$.
Under Hypothesis \ref{assm1} and for any fixed $\vartheta>0$ and $\sigma>0$, if
\begin{equation}
e^{\sigma|x|^{2+\vartheta}}u(0,x),\,e^{\sigma|x|^{2+\vartheta}}u(1,x)\in L_{x}^{2}(\mathbb R^d),
\end{equation}
then $u\equiv0$.
\end{theorem}

Theorem \ref{thm1} shows that super-quadratic exponential decay at two distinct times enforces triviality of solutions under mild smallness assumptions on the coefficients. This can be viewed as a robust extension of known unique continuation results to a setting in which both the geometry and the magnetic field are nontrivial.

To obtain uniqueness under quadratic exponential decay, we impose an additional structural assumption on the matrix $G$. Under this assumption, we prove an analogue of the constant-coefficient results in \cite{BFGRV-JFA, EKPV-JEMS}.

\begin{assumption}\label{assm2}
We assume that the matrix $G$ is of the form
\begin{align*}
    G=I+\widetilde G_1,
\end{align*}
where
\begin{align*}
    \widetilde{G}_1:=
    \begin{pmatrix}
    0&\mathbf{0}^\top \\
\mathbf{0} & \tilde{G}(x')
    \end{pmatrix}.
\end{align*}
Here, $\mathbf{0}$ is the zero vector in $\R^{d-1}$, $x^\prime=(x_{2},\dots,x_{d})$, and
\begin{equation}\label{special metric101}
    \sup_{x^{\prime}\in\mathbb{R}^{d-1}}\langle x^\prime\rangle^{3}(|\nabla\tilde{G}(x^\prime)|+|\nabla^{2}\tilde{G}(x^{\prime})|+|\nabla^{3}\tilde{G}(x^{\prime})|)\leq\varepsilon_{0}.
\end{equation}
In addition, we require that $\tilde{G}(x^\prime)$ is nonnegative for $x^\prime\in\R^{d-1}$.
\end{assumption}
We write the potential $V$ in the form
\[
V(t,x)=V_{1}(x)+V_{2}(t,x),
\]
where $V_{1}$ is real-valued and independent of time, while $V_{2}$ may be complex-valued and depend on both variables; the weighted conditions imposed on $V_2$ are stated in the hypotheses below.

\begin{assumption}\label{assum3}
    Assume that the potential $V:\Bbb R^{d+1}\to\Bbb C$ satisfy  the following conditions
    \begin{align}
\|V_{1}(x)\|_{L^{\infty}}:=M_{1}<\infty,
\end{align}
and
\begin{align}\label{syl}
\sup_{0\leq t\leq1}\|e^{\gamma|x|^2}V_{2}(\cdot,t)\|_{L_{x}^{\infty}}e^{\sup_{0\leq t\leq 1}\|\operatorname{Im}V_{2}(\cdot,t)\|_{L_{x}^{\infty}}}:=M_{2}<\infty.
\end{align}
\end{assumption}

We now state the Hardy-type uncertainty principle for the Schr\"odinger equation and the corresponding uniqueness result for the heat equation with variable-coefficient electromagnetic potentials.

\begin{theorem}\label{schrodinger}
Under the Hypothesis \ref{assm2} on $G$ and \eqref{wyy2} on magnetic field $B$, suppose that $u\in L^{\infty}([0,1],L^{2}(\mathbb{R}^{d}))\cap L^{2}([0,1],H^{1}(\mathbb{R}^{d}))$ satisfies
\begin{equation}
\begin{cases}
i\partial_{t}u+\operatorname{div}_{A}(G\nabla_{A}u)+V(x,t)u=0,&\textit{in}\,\,\mathbb{R}^{d}\times[0,1],\\
u(0)=u_{0}.
\end{cases}
\end{equation}
We assume that there exist a unit vector such that $\textbf{e}_{1}\in\mathbb{S}^{d-1}$ and a small constant $\varepsilon_{1}$ such that
\begin{equation}
\|(G\textbf{e}_{1})^{\top}B\|_{L^{\infty}_{x}}\leq\varepsilon_1, 
\end{equation}
and the potential $V(t,x)$ satisfies Hypothesis \ref{assum3}. If $e^{\gamma|x|^{2}}u_{0}\in L^{2}(\mathbb{R}^{d})$ and $e^{\gamma|x|^{2}}u(1,x)\in L^{2}(\mathbb{R}^{d})$ for $\gamma>\frac12+\frac{N(\varepsilon_{0}+\varepsilon_{1})}{8}$, then $u\equiv0$. Here, $N=\max\{384c_{d},544\}$ and $c_{d}$ is constant depending on $d$.
\end{theorem}

Theorem \ref{schrodinger} establishes a Hardy-type uncertainty principle at the quadratic exponential scale under an additional structural assumption on the metric. This extra condition allows us to overcome certain obstructions in the commutator estimates and to recover a sharp result analogous to the constant-coefficient case.

\begin{remark}
From the Theorem \ref{schrodinger} above, we obtain a quantitative Hardy uncertainty principle compared to Federico--Li--Yu, where only the case of large parameter $\gamma$ was proved when $A=0$.
In particular, when $G=I$, the constants $\varepsilon_{0}$ and $\varepsilon_1$ reduce to $0$, and our result coincides with that of \cite{BFGRV-JFA}. Moreover, when $A=0$, we recover the classical result of \cite{EKPV-JEMS}.
\end{remark}

\begin{theorem}\label{parabolic} Assume Hypothesis \ref{assm2} and \ref{assum3} and  \eqref{wyy2}  hold. 
Suppose that $u\in C([0,1],L^{2}(\mathbb{R}^{d}))\cap L^{2}([0,1],H^{1}(\mathbb{R}^{d}))$ satisfies
\begin{equation}
\begin{cases}
\partial_{t}u=\operatorname{div}_{A}(G\nabla_{A}u)+V(x,t)u,&(x,t)\in\mathbb{R}^{d}\times[0,1],\\
u(0)=u_{0}.
\end{cases}
\end{equation}

Suppose that  there exists a unit vector $\textbf{e}_{1}\in\mathbb{S}^{d-1}$ such that
\begin{equation}
\|(G\textbf{e}_{1})^{\top}B\|_{L^{\infty}_{x}}\leq\varepsilon_1.
\end{equation}
Assume that
$e^{\gamma|x|^{2}}u(x,0)\in L^{2}(\mathbb{R}^{d})$ and $e^{\gamma|x|^{2}}u(x,1)\in L^{2}(\mathbb{R}^{d})$ for $\gamma>\frac{1}{2}+\frac{N\varepsilon_{1}}{8}$. Then $u\equiv0$. Here, $N=\max\{384c_{d},544\}$ and $c_{d}$ is constant depending on $d$.
\end{theorem}
\begin{remark}
The difference between the Schrödinger and heat cases can also be seen at the level of the underlying virial-type structure. In the parabolic setting, the dissipative structure provides sufficient coercivity to absorb the metric-related terms, whereas in the Schrödinger case such coercivity is lacking and $\varepsilon_{0}$ must be incorporated into the weight. In contrast, $\varepsilon_{1}$, associated with the magnetic field, cannot be removed even in the heat case, as the corresponding terms are antisymmetric and cannot be absorbed.
\end{remark}
\begin{remark}
Similarly to \cite{BCF-SIAM}, we can relax the strong degeneracy condition $(G\textbf{e}_1)^{\top}B=0$ to the smallness condition $\|(G\textbf{e}_1)^{\top}B\|_{L^\infty}<\varepsilon_1$ for $\varepsilon_1<1$ in the variable-coefficient case $G\neq I$ when $A$ is independent of time variable. This allows our results to cover the case of nontrivial magnetic fields in two dimensions, whereas the strong degeneracy condition implies that the magnetic field is trivial. Under the strong degeneracy condition in higher dimensions $d\geq3$, one can still allow nontrivial magnetic potentials. Moreover, the uniqueness result holds for $\gamma>\frac{1}{2}$ in the parabolic setting, matching the result in \cite{EKPV-JEMS,BFGRV-JFA}.
\end{remark}

\begin{remark}
We now give an example for a general matrix $G$ and magnetic field $B$ in dimension $d=3$:
\begin{equation}
    G=\Big( 1+\frac{\varepsilon_0}{100\langle x\rangle^3}\Big)I,
\end{equation}
\begin{equation}
    B(x)=\frac{\varepsilon_{1}}{\langle x\rangle^2}(x_{2},-x_{1},0).
\end{equation}
Then $G$ satisfies Hypothesis \ref{assm1}, and
\begin{equation}
\Big(1+\frac{\varepsilon_0}{100\langle x\rangle^3}\Big)x^\top B=0.
\end{equation}
Next, for the special matrix $G$, let
\begin{equation}
    G= I+\begin{pmatrix}
    0&\mathbf{0}^\top \\
\mathbf{0} &\frac{\varepsilon_0}{100\langle x^{\prime}\rangle^{3}}I_{2}
    \end{pmatrix}.
\end{equation}
Then $G$ satisfies Hypothesis \ref{assm2}, and
\begin{equation}
    \Big|(G\textbf{e}_{1})^{\top}B\Big|=\Big|\textbf{e}_{1}^{\top} B\Big|=\frac{\varepsilon_{1}\big|x_{2}\big|}{\langle x\rangle^{2}}<\varepsilon_{1}.
\end{equation}
\end{remark}

\subsection{Roadmap and novelties}
In this subsection we give the main ideas of the proofs and point out the novelties. Each item below corresponds to one step, following the order in which the ingredients are used. 

\begin{itemize}[
    itemsep=0.35em,
    parsep=0.3em,
    topsep=0.5em
]
\item \textbf{The virial identity.} Our starting point is the weighted virial-type identity for the covariant Ginzburg--Landau flow in Lemma \ref{technical-lemma}. In the constant-coefficient and non-magnetic settings, the conjugated operator is split into its symmetric and antisymmetric parts $\mathcal{S}$ and $\mathcal{A}$, and the commutator $[\mathcal{S},\mathcal{A}]$ is computed directly; with a variable metric, this computation becomes intractable. Our method is to treat the covariant derivative $D_{j}=\partial_{j}-iA_{j}$ as a single operator. Commutators of covariant derivatives then generate the magnetic field through the identity $[D_{j},D_{k}]=-iB_{jk}$, and the virial identity is obtained by integration by parts rather than by expanding $[\mathcal{S},\mathcal{A}]$ directly. This departs from \cite{Yu-CCM}, where $[\mathcal{S},\mathcal{A}]$ is computed directly for $\operatorname{div}(G\nabla\cdot)$.\

\item \textbf{Logarithmic convexity.} Substituting the virial identity into the abstract scheme of \cite{EKPV-JEMS} yields logarithmic convexity of the weighted energy. We establish two versions: one for a general uniformly elliptic metric (\S\ref{subsec:logconvex}), and a sharper one for the matrices satisfying Hypothesis \ref{assm2} (\S\ref{subsec:logconvex-special}). In the general case, the same argument also yields the super-quadratic weighted estimate of Corollary \ref{cor-log}. Since these estimates require $H^{1}$ regularity whereas the data are only assumed to lie in $L^{2}$, we first lift the solution to $H^{1}$ by the Ginzburg--Landau parabolic regularization of \S\ref{subsec:parabolic-reg}, recovering the Gaussian weights estimates at $t=0$ and $t=1$ from the decay assumptions, and then let the regularization parameter tend to zero. This provides the upper bound in the comparison carried out below.

\item \textbf{Carleman estimate for a general metric.} For a general metric, we prove the Carleman estimate of Proposition \ref{Carleman estimate} in Section~\ref{sec:proof-thm1}. Following the approch of \cite{EKPV-CPDE}, one would natually consider the space-time mixed weight $\mu|\frac{x}{R}+\varphi(t)\mathbf{e}_{1}|^{2}$; with a variable $G$ this weight generates commutator terms that cannot be absorbed by the coercive part of the estimate. We replace it by the space-time separated weight
\[
\Phi(x,t)=\mu\left|\frac{x}{R}\right|^{4/3}+\mu^{\beta_{\ell}}\varphi(t),
\qquad
\beta_{\ell}=1+\tfrac{2}{3}2^{-\ell},
\]
in which the exponent of $\mu$ is tied to the exponent $p_{\ell}=4/(3-\beta_{\ell})=6/(3-2^{-\ell})$ of $R$ in the resulting exponential rate by a scaling relation adapted to the super-quadratic decay $e^{\sigma|x|^{2+\vartheta}}$, which is a novelty of the paper.

All three uniqueness theorems are proved by contradiction, and the mechanism is similar in each case. Assume that a nontrivial solution decays at the two times $t=0$ and $t=1$ faster than a certain threshold, and consider the energy of $u$ on a unit annulus of large radius,
\[
\delta(R):=\int_{1/8}^{7/8}\int_{B_{R}\backslash B_{R-1}}\big(|u|^{2}+|\nabla_{A}u|^{2}\big)\,\dd x\,\dd t .
\]
Two estimates are obtained for this single quantity, and they are incompatible. The first is an upper bound coming from logarithmic convexity: the weighted energy at an intermediate time is controlled by its values at the two endpoints, so the decay assumed at $t=0$ and $t=1$ propagates to the interior of the time interval, giving $\delta(R)\lesssim e^{-\sigma R^{2+\vartheta}}$. The second is a lower bound coming from the Carleman estimate above: if $u$ does not vanish identically, then its energy cannot be exponentially small on a large annulus, and Theorem \ref{mass lower} gives $\delta(R)\gtrsim e^{-CR^{\alpha}}$, where the exponent $\alpha$ is determined by the Carleman weight and the constants may depend on \(u\) but are independent of \(R\). For sufficiently large \(R\), the two bounds are incompatible whenever \(\alpha<2+\vartheta\), so the problem is reduced to constructing a weight for which the exponent $\alpha$ is as small as possible.

For the space-time separated weight $\Phi$ above, one has $\alpha=\frac{6}{3-2^{-\ell}}$, which is smaller than $2+\vartheta$ for sufficiently large $\ell$; this proves Theorem \ref{thm1}. The quadratic case therefore requires a different weight function, designed so that its large parameter remains comparable with the decay rate, as we explained below.

\item \textbf{Carleman estimate for a special metric.} The required mixed weight is admissible only for the structured matrices of Hypothesis \ref{assm2}. For this class, we prove a refined Carleman estimate in Section~\ref{sec:proof-schrodinger} using the auxiliary function
\begin{equation}
\psi(x,t):=\mu|x+Rt(1-t)\textbf{e}_{1}|^{2}-\frac{(1+\varepsilon+N(\varepsilon_{0}+\varepsilon_{1})\mu^{2})R^{2}t(1-t)}{16\mu},
\end{equation}
where $\varepsilon_{0}$ measures the size of the metric perturbation in Hypothesis \ref{assm2}, $\varepsilon_{1}$ bounds the transverse component of the magnetic field, $\|(G\textbf{e}_{1})^{\top}B\|_{L^{\infty}_{x}}\leq\varepsilon_{1}$, $0<\varepsilon\ll1$ is a fixed margin, and $N=\max\{384c_{d},544\}$, with $c_{d}$ depending only on the dimension. Two features distinguish this auxiliary function from the one used in \cite{EKPV-JEMS}, which is adapted to the constant metric only. First, it includes the corrections terms $N\varepsilon_{0}\mu^{2}$ and $N\varepsilon_{1}\mu^{2}$, which absorb the terms generated  by the derivatives of $G$ and by the magnetic field, respectively; this is the third novelty of the paper. Second, only smallness of the transverse component of the magnetic field is required, in place of the vanishing condition $(G\textbf{e}_{1})^{\top}B=0$ imposed in \cite{BFGRV-JFA}; this is the fourth novelty, and it allows nontrivial magnetic fields. This yields the uniqueness result of Theorem \ref{schrodinger} under the condition
\[
\gamma>\frac{1}{2}+\frac{N(\varepsilon_{0}+\varepsilon_{1})}{8}.
\]
When $G=I$, so that $\varepsilon_{0}=0$, our result reduces to that of \cite{BFGRV-JFA}; when $A=0$, so that $\varepsilon_{1}=0$, it reduces to that of \cite{EKPV-JEMS}.
\item \textbf{The parabolic case.} The same scheme, applied to the heat equation in Section~\ref{sec:proof-parabolic}, yields the threshold $\gamma>\frac{1}{2}+\frac{N\varepsilon_{1}}{8}$ of Theorem \ref{parabolic}. This threshold does not involve $\varepsilon_{0}$ is absent: the dissipative structure provides enough coercivity to absorb the terms generated by the geometry structure,  derivatives of $G$, so that no geometric correction is needed in the weight. The magnetic correction, however, cannot be removed with this argument, since it arises from magnetic terms in the commutator that are not controlled by the available coercive term. Thus, the dissipative structure of the parabolic equation leads to an improved threshold, shows the intrinsic connection between the threshold and the structure of the equation.
\item \textbf{Organization of the paper.} Section~\ref{sec:prelim} contains the preliminaries: the virial identity and the logarithmic convexity estimates for general and structured metrics (\S\ref{subsec:logconvex}--\S\ref{subsec:logconvex-special}), together with the parabolic regularization (\S\ref{subsec:parabolic-reg}). Section~\ref{sec:proof-thm1} establishes the Carleman estimate for a general metric and proves Theorem \ref{thm1}; Section~\ref{sec:proof-schrodinger} establishes the refined Carleman estimate for the structured class and Theorem \ref{schrodinger}; Section~\ref{sec:proof-parabolic} treats the parabolic case and proves Theorem \ref{parabolic}.
\end{itemize}

To summarize, the paper provides what is, to our knowledge, the first unified treatment of Hardy-type uncertainty principles and unique continuation for variable-coefficient covariant Schr\"odinger and parabolic equations in the presence of both a nontrivial metric and a nontrivial magnetic field. The proof is not a straightforward combination of existing arguments: the interaction between the geometry and the magnetic field produces new commutator terms and requires auxiliary functions and Carleman estimates adapted to the covariant setting.

\textbf{Notations. }\,
We write $X \lesssim Y$ or $Y \gtrsim X$ to mean that $X \leq CY$ for some absolute constant $C>0$. We denote by $O(Y)$ any quantity $X$ such that $|X|\lesssim Y$. The notation $X\sim Y$ means that $X\lesssim Y\lesssim X$. The quantity $o(1)$ denotes a term converging to zero. Finally, we use the Japanese bracket notation $\langle x\rangle=\sqrt{1+|x|^2}$.	
	
\section{Preliminaries}\label{sec:prelim}

In this section, we introduce the notation and basic identities that will be used throughout the paper, and we establish the logarithmic convexity properties for variable-coefficient covariant Schr\"odinger flows. These estimates will later provide the weighted upper bounds needed in the proof of the Carleman inequalities and the uniqueness results.

Let $A:=(A_1(x),\dots, A_d(x)):\R^{d}\to\R^d$ be the magnetic potential. The associated magnetic field is defined by
\begin{align*}
    B(x):=\nabla A(x)-(\nabla A)^{\top}(x):=\big[\partial_jA_k-\partial_kA_j\big]_{j,k=1}^d.
\end{align*}
Thus, $B$ can be viewed as the antisymmetric gradient of $A$. In the variable-coefficient setting.
Moreover, we introduce the vector field $\Psi(x)=(Gx)^{\top}B$, whose components are given by
\begin{align*}
    \Psi_{jk}(x):=g_{jk}x_jB_{jk}(x).
\end{align*}

\subsection{Logarithmic convexity}\label{subsec:logconvex}

In this subsection, we establish the logarithmic convexity estimates that will later be combined with the Carleman lower bounds. As in related works, we first work with a parabolic regularization, which allows us to justify the weighted commutator computations at the $H^1$ level.

We begin with the following Ginzburg--Landau-type equation:
\begin{equation}\label{equation of u}
    \partial_{t}u=(a+ib)\big(\operatorname{div}_{A}(G\nabla_{A}u)+V(x,t)u+F(x,t)\big).
\end{equation}

Our first lemma provides the weighted virial-type identity that underlies the logarithmic convexity argument.

For generality, we provide the virial identity for Schr\"odinger operator with time dependent electromagnetic field.
\begin{lemma}\label{technical-lemma}
Let $u$ be a solution to \eqref{equation of u}, where $a\in\R_+$, $b\in\R$, $A(t,x):\R^{d+1}\to\R^d$, $V,F:\R^{d+1}\to\C$, and $G:\R^{d}\to\R^d$.   
Let $v=e^{\varphi}u$. Then $v$ solves
\begin{equation}\label{equation of v}
    \partial_{t}v=(\mathcal{S}+\mathcal{A})v+(a+ib)(V(x,t)v+e^{\varphi}F),
\end{equation}
where
\begin{align}\label{symmetry part}
\mathcal{S}=a(G\nabla\varphi\cdot\nabla\varphi+\operatorname{div}_{A}\big(G\nabla_{A}\cdot\big))+ib(-\operatorname{div}(G\nabla\varphi)-2G\nabla\varphi\cdot\nabla_{A})+\partial_{t}\varphi
\end{align}
and
\begin{align}\label{antisymmetry part}
\mathcal{A}=ib(G\nabla\varphi\cdot\nabla\varphi+\operatorname{div}_{A}\big(G\nabla_{A}\cdot\big))+a(-\operatorname{div}(G\nabla\varphi)-2G\nabla\varphi\cdot\nabla_{A}).
\end{align}
Moreover,
\begin{equation}
\partial_{t}\mathcal{S}=2a\Big(\operatorname{Im}(G\partial_{t}A\nabla_{A}\cdot)+G\nabla\varphi\cdot\nabla\partial_{t}\varphi\Big)+2b\Big(\operatorname{Im}(G\nabla\partial_{t}\varphi\cdot\nabla_{A})-G\nabla\varphi\cdot\partial_{t}A\Big)+\partial_{tt}\varphi\label{St}
\end{equation}
and
\begin{align*}
     &\int_{\mathbb R^d}[\mathcal{S},\mathcal{A}]f\bar{f}\,\dd x\\
=
    &(a^2+b^2)\int_{\mathbb R^d}2G\nabla\varphi\cdot\nabla(G\nabla\varphi\cdot\nabla\varphi)|f|^{2}\,dx+\int_{\mathbb R^d}2b\operatorname{Im}(G\partial_{t}\nabla\varphi\cdot\nabla_{A}f)\bar{f}\,dx\\
    &+\int_{\mathbb R^d}2aG\nabla\varphi\cdot\nabla\partial_{t}\varphi|f|^{2}\,dx
    +(a^{2}+b^{2})\Bigg\{-\int_{\mathbb R^d}\mathcal{Q}^{2}\varphi|f|^{2}\,dx\\
&-4\operatorname{Im}\int_{\mathbb R^d}g_{jk}g_{lm}\partial_{l}\varphi(\partial_{k}A_{m}-\partial_{m}A_{k})f\overline{D_jf}\\
&+4\int_{\mathbb R^d} g_{jk}g_{lm}\partial_{k}\partial_{l}\varphi D_mf\overline{D_jf}
+\int_{\mathbb R^d}2g_{jk}\partial_{k}g_{lm}\partial_{l}\varphi D_mf\overline{D_jf}\\
&-2g_{lm}\partial_{l}\varphi\partial_{m}g_{jk}D_kf\overline{D_jf}
+2\partial_{j}g_{lm}\partial_{l}\varphi g_{jk}D_kf\overline{D_m f}\Bigg\}.
\end{align*}
\end{lemma}

\begin{proof}
We first prove \eqref{St}. A direct computation gives
\begin{align*}
    \partial_{t}\mathcal{S}
    =&\, a\Big(2G\nabla\varphi\cdot\nabla\partial_{t}\varphi-i\partial_{t}A\,G\nabla_{A}\cdot-\nabla_{A}\cdot(G\,i\partial_{t}A\cdot)\Big)\\
    &+ib\Big(-\operatorname{div}(G\nabla\partial_{t}\varphi)-2G\nabla\partial_{t}\varphi\cdot\nabla_{A}-2G\nabla\varphi\cdot(-i\partial_{t}A)\Big)+\partial_{t}^2\varphi.
\end{align*}
Let 
\[
D^{\star}=\nabla_{A}\cdot(G\partial_{t}A\cdot).
\]
Then $D=-G\partial_{t}A\,\nabla_{A}\cdot$, and therefore
\begin{align*}
&a\Big(2G\nabla\varphi\cdot\nabla\partial_{t}\varphi-i\partial_{t}A\,G\nabla_{A}\cdot-(\nabla-iA)\cdot(G\,i\partial_{t}A\cdot)\Big)\\
=&\,a\Big(iD-iD^{\star}+2G\nabla\varphi\cdot\nabla\partial_{t}\varphi\Big)\\
=&\,2a\Big(\frac{-D+D^{\star}}{2i}+G\nabla\varphi\cdot\nabla\partial_{t}\varphi\Big)\\
=&\,2a\Big(-\operatorname{Im}(D)+G\nabla\varphi\cdot\nabla\partial_{t}\varphi\Big).
\end{align*}
Similarly, let
\[
D_{1}=G\nabla\partial_{t}\varphi\cdot\nabla_{A}.
\]
Its adjoint is
\[
D_{1}^{\star}=-\operatorname{div}(G\nabla\partial_{t}\varphi)-G\nabla\partial_{t}\varphi\cdot\nabla_{A}.
\]
Hence,
\begin{align*}
    &-ib\Big(\operatorname{div}(G\partial_{t}\varphi)+2G\nabla\partial_{t}\varphi\cdot\nabla_{A}-2iG\nabla\varphi\cdot\partial_{t}A\Big)\\
=&\,-ib(D_{1}-D_{1}^{\star}-2iG\nabla\varphi\cdot\partial_{t}A)\\
=&\,2b\Big(\operatorname{Im}(D_1)-G\nabla\varphi\cdot\partial_{t}A\Big).
\end{align*}
Therefore,
\begin{equation}\label{S time derivative}
\begin{aligned}
\partial_{t}\mathcal{S}
=
2a\Big(\operatorname{Im}(G\partial_{t}A\nabla_{A}\cdot)+G\nabla\varphi\cdot\nabla\partial_{t}\varphi\Big)
+2b\Big(\operatorname{Im}(G\nabla\partial_{t}\varphi\cdot\nabla_{A})-G\nabla\varphi\cdot\partial_{t}A\Big)
+\partial_{t}^2\varphi.
\end{aligned}
\end{equation}

It remains to compute the commutator term $[\mathcal{S},\mathcal{A}]$. From \eqref{symmetry part} and \eqref{antisymmetry part}, we have
\begin{align*}
    \int_{\mathbb R^d}[\mathcal{S},\mathcal{A}]f\bar{f}\,\dd x
    =&\,(a^{2}+b^{2})\int_{\mathbb{R}^{d}}[G\nabla\varphi\cdot\nabla\varphi+\operatorname{div}_{A}(G\nabla_{A}\cdot),-\operatorname{div}(G\nabla\varphi)-2G\nabla\varphi\cdot\nabla_{A}]f\bar{f}\,\dd x\\
    &+\int_{\mathbb{R}^{d}}[\partial_{t}\varphi,\mathcal{A}]f\bar{f}\,\dd x.
\end{align*}
For convenience, we decompose this into three terms:
\begin{align*}
    I_1&=(a^2+b^2)\int_{\R^d}[G\nabla\varphi\cdot\nabla\varphi,-\operatorname{div}(G\nabla\varphi)-2G\nabla\varphi\cdot\nabla_A]f\bar f\,dx,\\
    I_2&=(a^2+b^2)\int_{\R^d}[\operatorname{div}_A(G\nabla_A),-\operatorname{div}(G\nabla\varphi)-2G\nabla\varphi\cdot\nabla_A]f\bar f\,dx,\\
    I_3&=\int_{\R^d}[\partial_t\varphi,\mathcal{A}]f\bar f\,dx.
\end{align*}
The terms $I_1$ and $I_3$ are straightforward. Indeed,
\begin{align*}
 I_1
=&\,(a^2+b^2)\int_{\R^d}[G\nabla\varphi\cdot\nabla\varphi,-\operatorname{div}(G\nabla\varphi)-2G\nabla\varphi\cdot\nabla_{A}]f\bar f\,dx\\
=&\,(a^2+b^2)\int_{\R^d}[G\nabla\varphi\cdot\nabla\varphi,-2G\nabla\varphi\cdot\nabla_{A}]f\bar{f}\,dx\\
=&\,(a^2+b^2)\int_{\R^d}2G\nabla\varphi\cdot\nabla\big(G\nabla\varphi\cdot\nabla\varphi\big)|f|^2\,dx.
\end{align*}
Likewise,
\begin{align*}
  I_3
  =&\, \int_{\R^d}[\partial_{t}\varphi,ib\operatorname{div}_{A}(G\nabla_{A}\cdot)-2aG\nabla\varphi\cdot\nabla_{A}]f\bar f\,dx\\
    =&\,\int_{\R^d}\Big(-ib\big(\operatorname{div}(G\nabla\partial_{t}\varphi)+2G\nabla\partial_{t}\varphi\cdot\nabla_{A}\big)+2aG\nabla\varphi\cdot\nabla\partial_{t}\varphi\Big)|f|^2\,dx\\
    =&\,\int_{\R^d}\Big(2b\operatorname{Im}(G\nabla\partial_{t}\varphi\cdot\nabla_{A})+2aG\nabla\varphi\cdot\nabla\partial_{t}\varphi\Big)|f|^2\,dx.
\end{align*}

It therefore remains to compute $I_2$. For simplicity, let

\begin{equation}\label{Q}
\mathcal{Q}\varphi=\partial_{j}(g_{jk}\partial_{k}\varphi).
\end{equation}
Then
\begin{align*}
&[\operatorname{div}_A(G\nabla_A\cdot),-\mathcal{Q}\varphi-2G\nabla\varphi\cdot\nabla_A]f\\
=&
[\operatorname{div}_{A}(G\nabla_{A}\cdot),-\mathcal{Q}\varphi]f-[\operatorname{div}_{A}(G\nabla_{A}\cdot),2G\nabla\varphi\cdot\nabla_A]f
\stackrel{\triangle}{=}\mathcal{B}_2^1+\mathcal{B}_2^2.
\end{align*}
By definition,
\begin{align}
  \mathcal{B}_2^1=& -\operatorname{div}_{A}\big(G\nabla_{A}(\mathcal{Q}\varphi f)\big)+\mathcal{Q}\varphi\operatorname{div}_{A}(G\nabla_{A}f).\label{difficult1}
\end{align}
For the first term in $\mathcal{B}_2^1$, we obtain
\begin{align*}
    -\operatorname{div}_A\big(G\nabla_A(\mathcal{Q}\varphi f)\big)
    =&\,-D_{j}\big(g_{jk}D_{k}(\mathcal{Q}\varphi f)\big)\\
    =&\,-\partial_{j}(g_{jk}\partial_{k}(\mathcal{Q}\varphi))f-2g_{jk}\partial_{k}(\mathcal{Q}\varphi)D_{j}f-\mathcal{Q}\varphi\operatorname{div}_{A}(G\nabla_{A}f),
\end{align*}
where repeated indices are summed over. Hence
\begin{align*}
  \mathcal{B}_{2}^1
  =&\,[\operatorname{div}_{A}(G\nabla_{A}\cdot),-\mathcal{Q}\varphi]f\\
  =&\,-\partial_{j}(g_{jk}\partial_{k}(\mathcal{Q}\varphi))f-2g_{jk}\partial_{k}(\mathcal{Q}\varphi)D_jf\\
    =&\,-\mathcal{Q}^{2}\varphi f-2g_{jk}\partial_{k}(\mathcal{Q}\varphi)D_jf.
\end{align*}
Multiplying by $(a^2+b^2)\bar f$ and integrating in $x$, we get
\begin{align}
(a^2+b^2)\int_{\R^d}\mathcal{B}_2^1\bar{f}\,dx=-(a^2+b^2)\int_{\R^d}\big(\mathcal{Q}^2\varphi|f|^2+2g_{jk}\partial_{k}(\mathcal{Q}\varphi)D_jf\bar f\big)\,dx.\label{red-part1-2}
\end{align}

We now turn to $\mathcal{B}_2^2$. In local coordinates,
\begin{align*}
    \mathcal{B}_{2}^2
    &=-D_j(g_{jk}D_k(2g_{lm}\partial_{l}\varphi D_mf))+2g_{lm}\partial_{l}\varphi D_m\Big(D_j\big(g_{jk}D_kf\big)\Big)\\
    &\stackrel{\triangle}{=}\mathcal{B}_{2}^{2,1}+\mathcal{B}_{2}^{2,2}.
\end{align*}
Using Leibniz's rule for $\nabla_A$,
\begin{align*}
    \nabla_A(fh)=f\nabla h+h\nabla_Af,
\end{align*}
we obtain
\begin{align*}
    \mathcal{B}_2^{2,1}
    &=-2D_j(g_{jk}\partial_{k}g_{lm}\partial_{l}\varphi D_{m}f+g_{jk}g_{lm}\partial_{k}\partial_{l}\varphi D_{m}f+g_{jk}g_{lm}\partial_{l}\varphi D_kD_mf)\\
    &=\operatorname{I+II+III}.
\end{align*}
For $\mathcal{B}_2^{2,2}$, we use the commutation identity
\begin{equation*}
D_{m}D_{j}-D_{j}D_{m}=-(i\partial_{m}A_{j}-i\partial_{j}A_{m}),
\end{equation*}
which yields
\begin{align}
\mathcal{B}_2^{2,2}
=&\,2g_{lm}\partial_{l}\varphi D_{j}D_{m}(g_{jk}D_{k}f)-2ig_{lm}\partial_{l}\varphi(\partial_{m}A_{j}-\partial_{j}A_{m})g_{jk}D_kf.\label{B22-1}
\end{align}
Next,
\begin{align*}
    &2g_{lm}\partial_{l}\varphi D_jD_m(g_{jk}D_kf)\\
=&\,2D_j\Big(g_{lm}\partial_{l}\varphi D_m(g_{jk}D_kf)\Big)-2\partial_{j}(g_{lm}\partial_{l}\varphi)D_m(g_{jk}D_kf)\\
=&\,2D_j\Big(g_{lm}\partial_{l}\varphi\partial_{m}g_{jk}D_kf+g_{lm}\partial_{l}\varphi g_{jk}D_mD_kf \Big)-2D_m\Big(\partial_{j}(g_{lm}\partial_{l}\varphi)(g_{jk}D_kf)\Big)\\
&+2\partial_{m}\partial_{j}(g_{lm}\partial_{l}\varphi)g_{jk}D_kf\\
=&\,2D_j\Big(g_{lm}\partial_{l}\varphi\partial_{m}g_{jk}D_kf+g_{lm}\partial_{l}\varphi g_{jk}D_mD_kf \Big)-2D_m\Big(\partial_{j}g_{lm}\partial_{l}\varphi g_{jk}D_kf+g_{lm}\partial_{j}\partial_{l}\varphi g_{jk}D_kf \Big)\\
&+2\partial_{m}\partial_{j}(g_{lm}\partial_{l}\varphi)g_{jk}D_kf.
\end{align*}
Combining this with \eqref{B22-1}, we obtain
\begin{align*}
\mathcal{B}_2^{2,2}
=&\,2D_j\Big(g_{lm}\partial_{l}\varphi\partial_{m}g_{jk}D_kf+g_{lm}\partial_{l}\varphi g_{jk}D_mD_kf \Big)-2D_m\Big(\partial_{j}g_{lm}\partial_{l}\varphi g_{jk}D_kf+g_{lm}\partial_{j}\partial_{l}\varphi g_{jk}D_kf \Big)\\
&+2\partial_{m}\partial_{j}(g_{lm}\partial_{l}\varphi)g_{jk}D_kf-2ig_{lm}\partial_{l}\varphi(\partial_{m}A_{j}-\partial_{j}A_{m})g_{jk}D_kf\\
&\stackrel{\triangle}{=}\operatorname{IV+V+VI+VII+VIII+IX}.
\end{align*}
Multiplying by $\bar f$ and integrating by parts, we obtain
\begin{align}
    &\int_{\mathbb R^d} (\operatorname{III+V+IX})\bar{f}\,\dd x\nonumber\\
=&\int_{\mathbb{R}^{d}}\Big[\Big(2g_{jk}g_{lm}\partial_{l}\varphi D_kD_mf-2g_{jk}g_{lm}\partial_{l}\varphi D_mD_kf\Big)\overline{D_jf}-2ig_{lm}\partial_{l}\varphi(\partial_{m}A_{j}-\partial_{j}A_{m})g_{jk}D_kf\bar{f}\Big]\,dx\nonumber\\
     =&\int_{\mathbb{R}^{d}}\Big[2ig_{jk}g_{lm}\partial_{l}\varphi(\partial_{k}A_{m}-\partial_{m}A_{k})f\overline{D_jf}-2ig_{lm}\partial_{l}\varphi(\partial_{m}A_{j}-\partial_{j}A_{m})g_{jk}D_kf\bar{f}\Big]\,dx\nonumber\\
     =&-\int_{\mathbb{R}^{d}}4\operatorname{Im}\Big[g_{jk}g_{lm}\partial_{l}\varphi(\partial_{k}A_{m}-\partial_{m}A_{k})f\overline{D_jf}\Big]\,dx.\label{357}
\end{align}

For the terms $\operatorname{II}$ and $\operatorname{VII}$, symmetry gives
\begin{equation}
\begin{aligned}
    \int_{\mathbb{R}^{d}} (\operatorname{II+VII})\bar{f}\,dx
    =&\int_{\mathbb R^d} 2g_{jk}g_{lm}\partial_{k}\partial_{l}\varphi D_mf\overline{D_jf}+2g_{lm}\partial_{j}\partial_{l}\varphi g_{jk}D_kf\overline{D_mf}\\
 =&\int_{\mathbb{R}^{d}} \big(2g_{jk}g_{lm}\partial_{k}\partial_{l}\varphi D_mf\overline{D_jf}+2g_{lm}\partial_{k}\partial_{l}\varphi g_{kj}D_jf\overline{D_mf}\big)\,dx\\
    =&4\int_{\mathbb R^d} g_{jk}g_{lm}\partial_{k}\partial_{l}\varphi D_mf\overline{D_jf}\,dx.\label{27}
\end{aligned}
\end{equation}

The integral involving $\operatorname{VIII}$ cancels with the second term in $\mathcal{B}_2^1$, so these contributions vanish. It remains to consider $\operatorname{I}$, $\operatorname{IV}$, and $\operatorname{VI}$. A direct computation yields
\begin{equation}
\begin{aligned}
    \int_{\mathbb{R}^{d}} (\operatorname{I+IV+VI})\bar{f}\,dx
=&\int_{\mathbb{R}^{d}}\big(2g_{jk}\partial_{k}g_{lm}\partial_{l}\varphi
    D_mf\overline{D_jf}-2g_{lm}\partial_{l}\varphi\partial_{m}g_{jk}D_kf\overline{D_jf}\\
    &\qquad\qquad\qquad+2\partial_{j}g_{lm}\partial_{l}\varphi g_{jk}D_kf\overline{D_mf}\big)\,dx.\label{146}
\end{aligned}
\end{equation}

Finally, combining \eqref{red-part1-2}, \eqref{146}, \eqref{27}, and \eqref{357}, we obtain

\begin{align}
    &\int_{\mathbb{R}^{d}}
    [\mathcal{S},\mathcal{A}]f\bar{f}\,\dd x
    \nonumber\\
    ={}&
    (a^2+b^2)\int_{\mathbb{R}^{d}}
    2G\nabla\varphi\cdot
    \nabla(G\nabla\varphi\cdot\nabla\varphi)|f|^{2}\,dx
    +\int_{\mathbb{R}^{d}}
    2b\operatorname{Im}
    (G\partial_{t}\nabla\varphi\cdot\nabla_{A}f)\bar{f}\,dx
    \nonumber\\
    &+\int_{\mathbb{R}^{d}}
    2aG\nabla\varphi\cdot\nabla\partial_{t}\varphi
    |f|^{2}\,dx
    \nonumber\\
    &+(a^{2}+b^{2})\Bigg\{
    \int_{\mathbb{R}^{d}}\Big(
    -\mathcal{Q}^{2}\varphi|f|^{2}
    +4\operatorname{Im}g_{jk}g_{lm}\partial_{l}\varphi
    (\partial_{k}A_{m}-\partial_{m}A_{k})
    f\overline{D_jf}
    +4g_{jk}g_{lm}\partial_{k}\partial_{l}\varphi
    D_mf\overline{D_jf}
    \Big)\,dx
    \nonumber\\
    &+\int_{\mathbb{R}^{d}}\Big(
    2g_{jk}\partial_{k}g_{lm}\partial_{l}\varphi
    D_mf\overline{D_jf}
    -2g_{lm}\partial_{l}\varphi\partial_{m}g_{jk}
    D_kf\overline{D_jf}
    +2\partial_{j}g_{lm}\partial_{l}\varphi g_{jk}
    D_kf\overline{D_mf}
    \Big)\,dx\Bigg\}.
    \nonumber\\
    &\label{formula of commutator}
\end{align}

\end{proof}

The previous computation yields the Gaussian inhomogeneous estimate that will be used repeatedly in the sequel.

\begin{lemma}\label{sub-Gaussian}
Suppose that $u\in L^{\infty}([0,1],L^{2}(\mathbb R^d))\cap L^{2}([0,1],H^{1}(\mathbb R^d))$ is a solution of 
\begin{equation}
\partial_{t}u=(a+ib)\Big(\operatorname{div}_{A}\big(G\nabla_{A}u\big)+V(x,t)u+F(x,t)\Big), \qquad \mathbb R^d\times[0,1],
\end{equation}
with $a>0$ and $b\in\mathbb{R}$. Assume moreover that $V(x,t)=V_{1}(x)+V_{2}(x,t)$, where $V_{1}:\mathbb R^d\to\mathbb{R}$ and $V_{2}(x,t):\mathbb R^{d+1}\to\mathbb{C}$ satisfy
\begin{align}
\|V_{1}(x)\|_{L^{\infty}}:=M_{1}<\infty,
\end{align}
and
\begin{align}
\sup_{0\leq t\leq1}\|e^{{\gamma|\cdot|^{2}}}V_{2}(\cdot,t)\|_{L_{x}^{\infty}}e^{\sup_{0\leq t\leq 1}\|\operatorname{Im}V_{2}(\cdot,t)\|_{L_{x}^{\infty}}}:=M_{2}<\infty.
\end{align}
Then the following inhomogeneous estimate holds:
\begin{align}
        &\|e^{\frac{\gamma a}{a+4\gamma\Lambda(a^{2}+b^{2})t}|x|^{2}}u(t)\|_{L^{2}}\leq {e^{\int_{0}^{t}\|a(\operatorname{Re}V)^{+}(s)-b\operatorname{Im}V(s)\|_{L_{x}^{\infty}}\,\dd s}}\|e^{\gamma|x|^{2}}u_{0}\|_{L_{x}^{2}}\nonumber\\
&\qquad+\sqrt{a^{2}+b^{2}}\int_{0}^{t}{e^{\int_{s}^{t}\|a(\operatorname{Re}V)^{+}(\tau)-b\operatorname{Im}V(\tau)\|_{L_{x}^{\infty}}\,\dd \tau}}\|e^{\frac{\gamma a}{a+4\gamma\Lambda(a^{2}+b^{2})s}|x|^{2}}F(s,\cdot)\|_{L^{2}_{x}}\,\dd s,
\end{align}
which in particular implies
\begin{equation}
\begin{aligned}\label{inhomogeneous type}
\|e^{\frac{\gamma a}{a+4\gamma\Lambda(a^{2}+b^{2})t}|x|^{2}}u(t)\|_{L^{2}}
\leq
e^{M_{T}}\Big(\|e^{\gamma|\cdot|^{2}}u(\cdot,0)\|_{L^{2}}+\sqrt{a^{2}+b^{2}}\|e^{\frac{\gamma a}{a+4\gamma\Lambda(a^{2}+b^{2})T}}F(x,t)\|_{L^{1}_{T}L_{x}^{\infty}}\Big).
\end{aligned}
\end{equation}
Here, $\Lambda$ is the constant in \eqref{uniformly elliptic}.
\end{lemma}

\begin{proof}
Let $v=e^{\varphi}u$, where $\varphi$ will be chosen below. Multiplying \eqref{equation of v} by $\bar v$ and integrating in $x$, we obtain
\begin{align}
\frac{1}{2}\frac{\dd}{\dd t}\|v\|_{L^{2}}^{2}
=&\operatorname{Re}\int_{\mathbb{R}^{d}}\mathcal{S}v\bar{v}\,\dd x+\operatorname{Re}\Big\{(a+ib)\int_{\mathbb{R}^{d}}(|v|^{2}V+e^{\varphi}F\bar{v})\,\dd x\Big\}\nonumber\\
=&-a\int_{\mathbb{R}^{d}}g_{jk}D_jv\overline{D_kv}\,\dd x+a\int_{\mathbb{R}^{d}}G\nabla\varphi\cdot\nabla\varphi|v|^{2}\,\dd x+\int_{\mathbb{R}^{d}}\partial_{t}\varphi|v|^{2}\,\dd x\nonumber\\
&+2b\operatorname{Im}\int_{\mathbb{R}^{d}}G\nabla\varphi\cdot\nabla_{A}v\bar{v}\,\dd x+\operatorname{Re}\Big\{(a+ib)\int_{\mathbb{R}^{d}}(|v|^{2}V+e^{\varphi}F\bar{v})\,\dd x\Big\}.\label{v integral with symmetry}
\end{align}

We first estimate the potential and forcing terms. By taking real parts and using H\"older's inequality,
\begin{align*}
\operatorname{Re}\Bigg[(a+ib)\int_{\mathbb{R}^{d}}|v|^{2}V\,\dd x\Bigg]
\leq
\Big\|a(\operatorname{Re}V)^{+}-b\operatorname{Im}V\Big\|_{L_{x}^{\infty}}\|v\|_{L^{2}_{x}}^{2}.
\end{align*}
Similarly,
\begin{align*}
\operatorname{Re}\Bigg[(a+ib)\int_{\mathbb{R}^{d}}e^{\varphi}F\bar{v}\,\dd x\Bigg]
\leq
\sqrt{a^{2}+b^{2}}\|e^{\varphi}F\|_{L^{2}_{x}}\|v\|_{L^{2}_{x}}.
\end{align*}

Next, by the Cauchy--Schwarz inequality and the ellipticity of $G$,
\begin{align*}
2b\operatorname{Im}\int_{\mathbb{R}^{d}}\bar{v}G\nabla\varphi\cdot\nabla_{A}v\,\dd x
\leq
a\int_{\mathbb{R}^{d}}G\nabla_{A}v\cdot\nabla_{A}\bar{v}\,\dd x+\frac{b^{2}}{a}\int_{\mathbb{R}^{d}}G\nabla\varphi\cdot\nabla\varphi|v|^{2}\,\dd x.
\end{align*}
Recalling the expression of $\mathcal S$, this implies
\begin{align*}
\operatorname{Re}\int_{\mathbb{R}^{d}}\mathcal{S}v\bar{v}\,\dd x
\leq
\int_{\mathbb{R}^{d}}\Big\{\big(a+\frac{b^{2}}{a}\big)G\nabla\varphi\cdot\nabla\varphi+\partial_{t}\varphi\Big\}|v|^{2}\,\dd x
\leq
\int_{\mathbb{R}^{d}}\Big\{\big(a+\frac{b^{2}}{a}\big)\Lambda|\nabla\varphi|^{2}+\partial_{t}\varphi\Big\}|v|^{2}\,\dd x.
\end{align*}

Now choose
\begin{align*}
    \varphi(x,t)=\frac{\gamma a}{a+4\gamma\Lambda(a^{2}+b^{2})t}\Big((\min\{|x|,R\})^{2}*\eta_{\varepsilon}\Big),
\end{align*}
where $R>\varepsilon>0$ and $\eta_\varepsilon\in C^\infty_c(B_{\varepsilon})$ is a mollifier. 
A direct computation gives
\begin{align*}
    \partial_{t}\varphi=-\Lambda\Big(a+\frac{b^{2}}{a}\Big)|\nabla\varphi|^{2}.
\end{align*}
Therefore,
\begin{equation*}
\operatorname{Re}\int_{\mathbb R^d}\mathcal{S}v\bar{v}\,\dd x\leq0.
\end{equation*}
Using this estimate in \eqref{v integral with symmetry}, we find
\begin{align*}
    \frac{1}{2}\frac{\dd}{\dd t}\|v\|_{L^{2}}^{2}
    \leq
    \Big\|a(\operatorname{Re}V)^+-b\operatorname{Im}V\Big\|_{L_{x}^{\infty}}\|v\|_{L^{2}_{x}}^{2}
    +\sqrt{a^{2}+b^{2}}\|e^{\varphi}F\|_{L_{x}^{2}}\|v\|_{L_{x}^{2}}.
\end{align*}
Hence,
\begin{align*}
\frac{\dd}{\dd t}\|v\|_{L^{2}}
\leq
\Big\|a(\operatorname{Re}V)^+-b\operatorname{Im}V\Big\|_{L_{x}^{\infty}}\|v\|_{L_{x}^{2}}+\sqrt{a^{2}+b^{2}}\|e^{\varphi}F\|_{L_{x}^{2}}.
\end{align*}
An application of Gronwall's inequality gives
\begin{align}
    \|v\|_{L^{2}}
    \leq
    e^{\int_{0}^{t}\|a(\operatorname{Re}V)^{+}(s)-b\operatorname{Im}V(s)\|_{L_{x}^{\infty}}\,\dd s}\|v_{0}\|_{L_{x}^{2}}
    +\sqrt{a^2+b^2}\int_{0}^{t}e^{\int_{s}^{t}\|a(\operatorname{Re}V)^{+}(\tau)-b\operatorname{Im}V(\tau)\|_{L_{x}^{\infty}}\,\dd\tau}\|e^{\varphi}F\|_{L_{x}^{2}}\,\dd s.
\end{align}

This concludes the proof.
\end{proof}

We now recall the abstract logarithmic convexity lemma from Escauriaza--Kenig--Ponce--Vega \cite{EKPV-JEMS}, which will be applied to our weighted covariant flow.

\begin{lemma}\label{abstract lemma1}
Let $\mathcal{S}$ be a symmetric operator and $\mathcal{A}$ a skew-symmetric operator, both with coefficients depending on $x$ and $t$. Assume that $f(x,t)$ is smooth and $W$ is a positive function. Define
\begin{align*}
    H(t)=\int_{\mathbb R^d}|f|^{2}\,\dd x,
\end{align*}
and
\begin{align}
    D(t)=(\mathcal{S}f,f),\qquad \partial_{t}\mathcal{S}=\partial_{t}\mathcal{S},\qquad N(t)=\frac{D(t)}{H(t)}.
\end{align}
Then
\begin{align}
    \partial_{t}^{2}H(t)
    =&\,2\partial_{t}\operatorname{Re}(\partial_{t}f-\mathcal{S}f-\mathcal{A}f,f)+2(\partial_{t}\mathcal{S}f+[\mathcal{S},\mathcal{A}]f,f)\nonumber\\
    &+\|\partial_{t}f-\mathcal{A}f+\mathcal{S}f\|^{2}-\|\partial_{t}f-\mathcal{A}f-\mathcal{S}f\|^{2},\label{second derivative of H}
\end{align}
and
\begin{align}
    \partial_{t}N(t)\geq(\partial_{t}\mathcal{S}f+[\mathcal{S},\mathcal{A}]f,f)/H-\|\partial_{t}f-\mathcal{A}f-\mathcal{S}f \|^{2}/2H.
\end{align}
Assume moreover that
\begin{align}
    \Big|\partial_{t}f-\big(\mathcal{S}+\mathcal{A}\big)f\Big|\leq M_{1}|f|+W, \qquad (t,x)\in[0,1]\times\mathbb R^d,\quad M_{1}\geq0,
\end{align}
and
\begin{align}\label{lower bound}
    \partial_{t}\mathcal{S}+[\mathcal{S},\mathcal{A}]\geq-M_{0},\qquad M_0\geq0.
\end{align}
Let
\begin{align}
    M_{2}:=\sup_{t\in[0,1]}\frac{\big\|W(t)\big\|_{L^{2}}}{\|f(t)\|_{L^{2}}}.
\end{align}
Then $\psi(t):=\log H(t)$ is convex on $[0,1]$. In particular, if
\begin{align}\label{boundedness of H}
    H(0)<\infty\Rightarrow H(t)<\infty,\qquad \forall t\in[0,1],
\end{align}
then there exists a universal constant  $N\geq0$ such that
\begin{align}
    H(t)\leq e^{N(M_{0}+M_{1}+M_{2}+M_{1}^{2}+M_{2}^{2})}H(0)^{1-t}H(1)^{t},\qquad \forall t\in[0,1].
\end{align}
\end{lemma}

We are now in a position to prove the logarithmic convexity estimate for the parabolic regularization of the variable-coefficient covariant flow.

\begin{lemma}\label{log convex lemma}
Suppose that $u\in L^{\infty}([0,1],L^{2}(\mathbb R^d))\cap L^{2}([0,1],H^{1}(\mathbb R^d))$ is a solution to 
\begin{equation}
\partial_{t}u=(a+ib)\Big(\operatorname{div}_{A}(G\nabla_{A}u)+V(x,t)u+F(x,t)\Big)
\end{equation}
in $\mathbb R^d\times[0,1]$, where $a>0$ and $b\in\mathbb{R}$. The magnetic potential $A$ satisfies Hypothesis \ref{assm1}.  
Let $\gamma>0$, and assume that
\begin{equation}\label{FFFFF}
\sup_{t\in[0,1]}\|V\|_{L^{\infty}}:=M_{1}<\infty,\qquad
\sup_{t\in[0,1]}\frac{\|e^{\gamma|\cdot|^{2}}F(\cdot,t)\|_{L^{2}}}{\|u(\cdot,t)\|_{L^{2}}}:=M_{2}<\infty.
\end{equation}
For the magnetic field $B=DA-(DA)^{\top}$, assume that
\begin{align}
      M_{G,A}:= &(a^{2}+b^{2})\Big[2\gamma\sup_{x\in\mathbb R^d}|x||\nabla G|\|\nabla^{2}G\|_{L^{\infty}}+2\gamma\sup_{x\in\mathbb R^d}|x||\nabla^{3}G|\nonumber\\
&\qquad\qquad
+4\gamma\|\nabla G\|^{2}_{L^{\infty}}+4\gamma\Lambda\|\nabla^{2}G\|_{L^{\infty}}+4\gamma\|(Gx)^{\top}B\|^{2}_{L^\infty}\Big]<\infty.\label{MGA}
\end{align}
If
\begin{align}
\|e^{\gamma|\cdot|^{2}}u(\cdot,0)\|_{L^{2}}+\|e^{\gamma|\cdot|^{2}}u(\cdot,1)\|_{L^{2}}<\infty,
\end{align}
then
\[
H(t)=\|e^{\gamma|\cdot|^{2}}u(t,\cdot)\|^{2}_{L^{2}}
\]
is finite and ``logarithmically convex" on $[0,1]$. Moreover, there exists a  universal  constant $N$ such that
\begin{equation}
    H(t)\leq e^{N[M_{G,A}+\sqrt{a^{2}+b^{2}}(M_{1}+M_{2})+(a^{2}+b^{2})(M_{1}^{2}+M_{2}^{2})]}H(0)^{1-t}H(1)^{t}. \label{log convex}
\end{equation}
In addition, the following space-time estimate holds:
\begin{align}
   & \mathcal{C}_{0}\|\sqrt{t(1-t)}e^{\gamma|x|^{2}}\nabla_{A}u\|^{2}_{L^{2}(\mathbb R^d\times[0,1])}
+\mathcal{C}_{2}\|\sqrt{t(1-t)}|x|e^{\gamma|x|^{2}}u\|^{2}_{L^{2}(\mathbb R^d\times[0,1])}\nonumber\\
\leq&\mathcal{C}_{1}\sup_{t\in[0,1]}\|e^{\gamma|x|^{2}}u(t)\|^{2}_{L^2(\R^d)}
+\frac{7}{6}(a^{2}+b^{2})\sup_{t\in[0,1]}\|e^{\gamma|x|^{2}}F\|^{2}_{L^{2}(\mathbb{R}^{d})},\label{t and 1-t}
\end{align}
where
\[
\mathcal{C}_{0}=\frac{\lambda^2}{2}(3\gamma a^{2}+4\gamma b^{2}-12(a^{2}+b^{2})c_{d}\gamma\Lambda\varepsilon_{0}),
\]
\[
\mathcal{C}_{1}=M_{G,A}+\frac{7}{6}(a^{2}+b^{2})M_{1}^{2}+3,
\]
\[
\mathcal{C}_{2}=4\lambda^2\gamma^3(5a^{2}+4b^{2}+12(a^{2}+b^{2})c_{d}\Lambda\varepsilon_{0})),
\]
and $c_d$ a constant depend on dimension $d$. 
\end{lemma}

\begin{proof}
Let $v=e^{\varphi(t,x)}u$ with $\varphi(t,x)=\gamma|x|^{2}$. From \eqref{equation of v}, we have
\begin{equation}\label{formula of sym anti}
    |\partial_{t}v-(\mathcal{S}+\mathcal{A})v|\leq\sqrt{a^{2}+b^{2}}\Big(M_{1}v+e^{\varphi}|F|\Big).
\end{equation}
We now apply Lemma \ref{abstract lemma1}. Let
\[
W=\sqrt{a^{2}+b^{2}}e^{\varphi}F.
\]
To use the abstract lemma, it remains to verify \eqref{lower bound}. By Lemma \ref{technical-lemma} and the choice $\varphi=\gamma|x|^{2}$ (hence $\partial_t\varphi=0$), we get
\begin{align}
    &\int_{\mathbb R^d}\bar{v}\Big(\partial_{t}\mathcal{S}+[\mathcal{S},\mathcal{A}] \Big)v\,\dd x\nonumber\\
=&\int_{\mathbb{R}^{d}}\big(2a\operatorname{Im}(G\partial_{t}A\nabla_{A}v)\bar{v}-4b\gamma Gx\cdot\partial_{t}A|v|^{2}\big)\,\dd x\nonumber\\
&+(a^{2}+b^{2})\Bigg\{16\gamma^{3}\int_{\mathbb R^d}G x\cdot\nabla(Gx\cdot x)|v|^{2}\,\dd x
-\int_{\mathbb R^d}\mathcal{Q}^{2}(\gamma|x|^{2})|v|^{2}\,\dd x\nonumber\\
&\qquad\qquad\qquad
-8\gamma\operatorname{Im}\int_{\mathbb{R}^{d}}g_{jk}g_{lm}x_{l}B_{mk}v\overline{D_jv}\,\dd x+8\gamma\int_{\mathbb{R}^{d}}g_{jk}g_{lm}\delta_{kl}D_mv\overline{D_jv}\,\dd x\nonumber\\
&\qquad\qquad\qquad
+\gamma\int_{\mathbb R^d}\Big[4g_{jk}\partial_{k}g_{lm}x_{l}D_mv\overline{D_jv}-4g_{lm}x_{l}\partial_{m}g_{jk}D_kv\overline{D_jv}+4\partial_{j}g_{lm}x_{l}g_{jk}D_kv\overline{D_mv}\Big]\,\dd x\Bigg\}\nonumber\\
=:&\sum_{j=1}^9 Z_j.\label{symmetry part and time derivative}
\end{align}
We estimate these terms one by one. Since $A$ is time-independent
\begin{align}\label{term1}
  Z_1=0,\,\, Z_2=0.
\end{align}


For $Z_3$ and $Z_4$, integration by parts and ellipticity of $G\geq\lambda$ gives
\begin{align}\label{term4}
Z_3
=&16\gamma^{3}(a^{2}+b^{2})\int_{\mathbb R^d}\Big(2|Gx|^{2}|v|^{2}+g_{kj}x_{j}\partial_{k}g_{lm}x_{l}x_{m}|v|^{2}\Big)\,\dd x\nonumber\\
\geq&
32\gamma^{3}(a^{2}+b^{2})\int_{\mathbb R^d}|Gx|^{2}|v|^{2}\,\dd x
-16\gamma^{3}(a^{2}+b^{2})c_{d}\Lambda\varepsilon_{0}\int_{\mathbb{R}^{d}}|x|^{2}|v|^{2}\,\dd x\\ \geq&(32\lambda^2-16c_{d}\Lambda\varepsilon_{0})\gamma^3(a^2+b^2)\int_{\mathbb{R}^d}|x|^{2}|v|^{2}\,\dd{x},
\end{align}
and
\begin{align}
Z_4
=&-2\gamma(a^{2}+b^{2})\int_{\mathbb R^d}\Big(\partial_{l}g_{lm}\partial_{m}\partial_{k}g_{kj}x_{j}+g_{lm}\partial_{l}\partial_{m}\partial_{k}g_{kj}x_{j}+\partial_{l}g_{lm}\partial_{k}g_{kj}\delta_{jm}\nonumber\\
&\qquad\qquad\qquad
+g_{lm}\partial_{m}\partial_{k}g_{kj}\delta_{jl}+g_{lm}\partial_{l}\partial_{k}g_{kj}\delta_{jm}
+\partial_{l}g_{lm}\partial_{m}g_{kk}+g_{lm}\partial_{l}\partial_{m}g_{kk}\Big)|v|^{2}\,\dd x\nonumber\\
\geq&
-2(a^{2}+b^{2})\Big(\sup_{x\in\mathbb R^d}|x||\nabla G|\|\nabla^{2}G\|_{L^{\infty}}+\sup_{x\in\mathbb R^d}|x||\nabla^{3}G|+2\|\nabla G\|^{2}_{L^{\infty}}+2\Lambda\|\nabla^{2}G\|_{L^{\infty}}\Big)\gamma\int_{\mathbb R^d}|v|^{2}\,\dd x.\label{term5}
\end{align}

Next,
\begin{align}
Z_5
=&-8\gamma(a^{2}+b^{2})\operatorname{Im}\int_{\mathbb R^d}G\overline{\nabla_{A}v}\cdot\,(Gx)^{\top}B v\,\dd x\nonumber\\
\geq&
-4(a^{2}+b^{2})\gamma\|(Gx)^{\top}B\|^{2}_{L^{\infty}}\int_{\mathbb R^d}|v|^{2}\,\dd x
-4(a^{2}+b^{2})\gamma\int_{\mathbb R^d}|G\nabla_{A}v|^{2}\,\dd x.\label{term6}
\end{align}

The term $Z_6$ is positive:
\begin{align}
    Z_6
    =&8\gamma(a^{2}+b^{2})\int_{\mathbb R^d}g_{jk}g_{mk}D_{m}v\overline{D_{j}v}\,\dd x
    =8\gamma(a^{2}+b^{2})\int_{\mathbb R^d}|G\nabla_{A}v|^{2}\,\dd x.\label{term7}
\end{align}

Finally,
\begin{equation}
\begin{aligned}
    &\gamma(a^{2}+b^{2})\int_{\mathbb R^d}4g_{jk}\partial_{k}g_{lm}x_{l}D_{m}v\overline{D_{j}v}-4g_{lm}x_{l}\partial_{m}g_{jk}D_kv\overline{D_{j}v}+4\partial_{j}g_{lm}x_{l}g_{jk}D_kv\overline{D_{m}v}\\
    \geq& -12(a^{2}+b^{2})c_{d}\gamma\sup_{x\in\mathbb R^d}\big(|x||\nabla G|\big)\int_{\mathbb R^d}|G\nabla_{A}v|^{2}\,\dd x.\label{term8}
\end{aligned}
\end{equation}
Combining \eqref{term1}--\eqref{term8}, we obtain
\begin{align}
&\int_{\mathbb R^d}\bar{v}\Big(\partial_{t}\mathcal{S}+[\mathcal{S},\mathcal{A}]\Big)v\,\dd x\nonumber\\
\geq &
\Big(-\gamma\,a^{2}+4\gamma(a^{2}+b^{2})-12\gamma(a^{2}+b^{2})c_{d}\Lambda\sup_{x\in\mathbb R^d}\big(|x||\nabla G|\big)\Big)\int_{\mathbb R^d}|G\nabla_{A}v|^{2}\,\dd x\nonumber\\
&+(32\lambda^2-16c_d\Lambda\varepsilon_0)(a^{2}+b^{2})\gamma^{3}\int_{\mathbb R^d}|Gx|^{2}|v|^{2}\,\dd x\nonumber\\
&-(a^{2}+b^{2})\Big(2\gamma\sup_{x\in\mathbb R^d}|x||\nabla G|\|\nabla^{2}G\|_{L^{\infty}}
+2\gamma\sup_{x\in\mathbb R^d}|x||\nabla^{3}G|\nonumber\\
&\qquad\qquad
+4\gamma\|\nabla G\|^{2}_{L^{\infty}}+4\gamma\Lambda\|\nabla^{2}G\|_{L^{\infty}}+4\gamma\|(Gx)^{\top}B\|^{2}_{L^\infty}\Big)\int_{\mathbb R^d}|v|^{2}\,\dd x.\label{complex formula}
\end{align}

By the assumption on $G$, we have
\begin{equation}
    32\lambda^2-16c_d\Lambda\varepsilon_{0}>16\lambda^2>0,
\end{equation}
\begin{equation}
    -\gamma a^{2}+4\gamma(a^{2}+b^{2})-12\gamma(a^{2}+b^{2})c_{d}\Lambda\sup_{x\in\mathbb{R}^{d}}(|x||\nabla G|)\geq 2\gamma(a^{2}+b^{2}).
\end{equation}
Hence,
\begin{align}
    \partial_{t}\mathcal{S}+[\mathcal{S},\mathcal{A}]\geq- {M}_{G,A},
\end{align}
with ${M}_{G,A}$ as in \eqref{MGA}. Together with the bounds $\sqrt{a^{2}+b^{2}}M_{1}$ and $\sqrt{a^{2}+b^{2}}M_{2}$ coming from \eqref{formula of sym anti}, Lemma \ref{abstract lemma1} gives the logarithmic convexity estimate \eqref{log convex}.

However, the above calculation is formally. To make the proof rigorously,  we need 
\begin{equation}\label{claim-uniform}
    \|e^{\gamma|x|^{2}}u(t)\|_{L^{2}(\mathbb{R}^{d})}\leq C,\,\forall t\in[0,1].
\end{equation} 
To verify this, we will use the approximation argument.  First, as a consequence of Lemma \ref{sub-Gaussian}, we obtain a sub-Gaussian decay $$\|e^{\gamma|x|^{2-2\delta}}u(t,x)\|_{L^2(\Bbb R^d)}\leq C.$$ However, to take the limit as $\delta\to0$, we need to prove the estimate uniformly in $\delta$. 
Replacing $\varphi$ by $\varphi_{\delta}=|x|^{2-2\delta}$ with $\delta\in(0,1)$. Repeating the proof in calculation the lower bound of $\partial_{t}\mathcal{S}+[\mathcal{S},\mathcal{A}]$, one checks that the previous calculations remain valid. Then passing to the limit as $\delta\to0$. We refer to \cite{Cazenave} for this standard argument.

We now prove \eqref{t and 1-t}. Using \eqref{second derivative of H}, we have
\begin{align}
    \partial_{t}^{2}H(t)
    =&2\partial_{t}\operatorname{Re}(\partial_{t}v-\mathcal{S}v-\mathcal{A}v,v)+2(\partial_{t}\mathcal{S}v+[\mathcal{S},\mathcal{A}]v,v)\nonumber\\
    &+\|\partial_{t}v-\mathcal{A}v+\mathcal{S}v\|_{L^2}^{2}-\|\partial_{t}v-\mathcal{A}v-\mathcal{S}v\|_{L^2}^{2}\nonumber\\
    \geq&
    2\partial_{t}\operatorname{Re}(\partial_{t}v-\mathcal{S}v-\mathcal{A}v,v)+2(\partial_{t}\mathcal{S}v+[\mathcal{S},\mathcal{A}]v,v)-\|\partial_{t}v-\mathcal{A}v-\mathcal{S}v\|_{L^2}^{2}\nonumber\\
    =:&\mathfrak{R}_1+\mathfrak{R}_2+\mathfrak{R}_3.
\end{align}
Multiplying by $t(1-t)$ and integrating by parts in time yields
\begin{align}
    \int_{0}^{1}t(1-t)\frac{d^{2}}{dt^{2}}H(t)\,\dd t
    =H(1)+H(0)-2\int_{0}^{1}H(t)\,\dd t
    \leq 2\sup_{0\leq t\leq1}\|v(\cdot,t)\|_{L^{2}}^{2}.
\end{align}

For $\mathfrak{R}_{1}$, integration by parts gives
\begin{align}
    \mathfrak{R}_{1}
    =&2\int_{0}^{1}\int_{\mathbb R^d}t(1-t)\partial_{t}\operatorname{Re}(\partial_{t}v-\mathcal{S}v-\mathcal{A}v)\bar{v}\,\dd x\dd t\nonumber\\
    =&-2\int_{0}^{1}\int_{\mathbb R^d}(1-2t)\operatorname{Re}(\partial_{t}v-\mathcal{S}v-\mathcal{A}v)\bar{v}\,\dd x\dd t\nonumber\\
    \geq&
    -\Big((a^{2}+b^{2})M_{1}^{2}+1\Big)\sup_{0\leq t\leq1}\|v(\cdot,t)\|_{L^{2}}^{2}
    -(a^{2}+b^{2})\sup_{0\leq t\leq1}\|e^{\gamma|\cdot|^{2}}F\|_{L^2}^{2}.
\end{align}

For $\mathfrak{R}_{2}$,
\begin{align}
   \mathfrak{R}_{2}
   =&2\int_{0}^{1}\int_{\mathbb R^d}t(1-t)\bar{v}(\partial_{t}\mathcal{S}+[\mathcal{S},\mathcal{A}])v\,\dd x\dd t\nonumber\\
   \geq&
   -M_{G,A}\sup_{0\leq t\leq1}\|v(t,\cdot)\|^{2}_{L^{2}}
   +(3\gamma a^{2}+4\gamma b^{2}-12(a^{2}+b^{2})c_{d}\gamma\Lambda\varepsilon_{0})\int_{0}^{1}\int_{\mathbb R^d}t(1-t)|G\nabla_{A}v|^{2}\,\dd x\dd t\nonumber\\
   &+32\gamma^{3}(a^{2}+b^{2})\int_{0}^{1}\int_{\mathbb R^d}t(1-t)|Gx|^{2}|v|^{2}\,\dd x\dd t.
\end{align}

For $\mathfrak{R}_{3}$, using \eqref{formula of sym anti},
\begin{align}
    \mathfrak{R}_{3}
    =&-\int_{0}^{1}t(1-t)\|\partial_{t}v-\mathcal{S}v-\mathcal{A}v\|_{L^{2}}^{2}\,\dd t\nonumber\\
    \geq&
    -\frac{1}{6}(a^{2}+b^{2})M_{1}^{2}\sup_{0\leq t\leq1}\|v(t,\cdot)\|^{2}_{L^{2}}
    -\frac{1}{6}(a^{2}+b^{2})\sup_{0\leq t\leq1}\|e^{\gamma|\cdot|^{2}}F(\cdot,t)\|^{2}_{L^{2}}.
\end{align}
Combining the upper and lower bounds for $\partial_t^2 H$, we obtain
\begin{align}\label{last step for t and 1-t estimate}
    &(3\gamma a^{2}+4\gamma b^{2}-12(a^{2}+b^{2})c_{d}\gamma\Lambda\varepsilon_{0})\|\sqrt{t(1-t)}G\nabla_{A}v\|^{2}_{L^{2}([0,1]\times\mathbb R^d)}\nonumber\\
&+ 32\gamma^{3}(a^{2}+b^{2})\|\sqrt{t(1-t)}Gxv\|^{2}_{L^{2}([0,1]\times\mathbb R^d)}\nonumber\\
\leq&
M_{3}\sup_{0\leq t\leq1}\|v(\cdot,t)\|^{2}_{L^{2}}+M_{4}\sup_{0\leq t\leq1}\|e^{\gamma|\cdot|^{2}}F\|^{2}_{L^{2}},
\end{align}
where
\[
M_{3}=M_{G,A}+\frac{7}{6}(a^{2}+b^{2})M_{1}^{2}+3,
\qquad
M_{4}=\frac{7}{6}(a^{2}+b^{2}).
\]
Since $v=e^{\gamma|x|^{2}}u$, we have
\begin{equation}
\nabla_{A}u=e^{-\gamma|x|^{2}}\nabla_{A}v-2\gamma xe^{-\gamma|x|^{2}}v,
\end{equation}
and therefore
\begin{equation}\label{upper bound for nabla u}
  \|\sqrt{t(1-t)}e^{\gamma|x|^2}G\nabla_{A} u\|_{L^{2}}^{2}
  \leq
  2\|\sqrt{t(1-t)}G\nabla_{A} v\|^{2}_{L_{t,x}^{2}}+2\|2\sqrt{t(1-t)}\gamma Gx v\|_{L_{t,x}^{2}}^2.
\end{equation}
Combining \eqref{upper bound for nabla u} with \eqref{last step for t and 1-t estimate} yields \eqref{t and 1-t}.
\end{proof}

As a consequence, we obtain the following version with super-quadratic exponential weights.

\begin{corollary}\label{cor-log} Let $V=V_1(x)$.
Under the assumptions of Lemma \ref{log convex lemma}, let $\eta\in(1,2)$ and $\varrho>0$. If
\begin{equation}
e^{\varrho|x|^{2\eta}}u(0,x),\qquad e^{\varrho|x|^{2\eta}}u(1,x)\in L^{2}(\mathbb R^d),
\end{equation}
then for every $t\in(0,1)$, $0<\tilde{\varrho}<\varrho$,
\begin{equation}
e^{\tilde{\varrho}|x|^{2\eta}}u(t,x)\in L^{2}(\mathbb R^d),
\end{equation}
and
\begin{equation}
\sqrt{t(1-t)}\, e^{\tilde{\varrho} |x|^{2\eta}} \nabla_A u,
\qquad
\sqrt{t(1-t)}\, e^{\tilde{\varrho} |x|^{2\eta}} x u(t)
\in L^2\bigl([0,1], L^2(\mathbb{R}^d)\bigr).
\end{equation}
Moreover,
\begin{equation}
\int_{\mathbb{R}^d} e^{2\tilde{\varrho} |x|^{2\eta}} |u(t,x)|^2 \, \dd x
\le C
\left(
\int_{\mathbb{R}^d} e^{2\varrho |x|^{2\eta}} |u(0,x)|^2 \, \dd x
\right)^{1-t}
\left(
\int_{\mathbb{R}^d} e^{2\varrho |x|^{2\eta}} |u(1,x)|^2 \, \dd x
\right)^t
\end{equation}
and
\begin{equation}
\Bigl\|
\sqrt{t(1-t)}\, e^{\tilde{\varrho} |x|^{2\eta}} \nabla_A u
\Bigr\|_{L^2_{t,x}}^2
+
\Bigl\|
\sqrt{t(1-t)}\, e^{\tilde{\varrho}|x|^{2\eta}} x u
\Bigr\|_{L^2_{t,x}}^2
\le C
\Bigl(
\|e^{\varrho |x|^{2\eta}} u(0,x)\|_{L^2_x}^2
+
\|e^{\varrho |x|^{2\eta}} u(1,x)\|_{L^2_x}^2
\Bigr).
\end{equation}
\end{corollary}

\begin{proof} Since $V(x,t)=V_{1}(x)$ and $V_2$ vanishes, we take  $F=V_2 u$. Then by \eqref{FFFFF}, $M_2=0$ in \eqref{log convex}. Since \(M_{G,A}=O(\gamma)\) and the abstract logarithmic convexity constant is independent of \(\gamma\), the estimate \eqref{log convex} gives an exponential factor \(e^{C_0\gamma}\) with \(C_0\) independent of \(\gamma\). The remaining \(\gamma\)-dependent constants contribute only to the prefactor, which we denote by \(C\). We now remove this dependence by
integrating against a suitable weight in $\gamma$. Multiply
\eqref{log convex} by
$
e^{-\frac{(2\gamma)^q}{qa_0^q}} (2\gamma)^{\frac{q-2}{2}}$ with $q=\frac{\eta}{\eta-1}$
and integrate with respect to $\gamma$ over $(\gamma_0,+\infty)$,
for some fixed $\gamma_0>0$. H\"older's inequality gives
\begin{align}
    &\int_{\mathbb{R}^d} \int_{\gamma_0}^{\infty}
    e^{2\gamma |x|^{2} - \frac{(2\gamma)^q}{qa_0^q}}
    (2\gamma)^{\frac{q-2}{2}}
    \, \dd\gamma \,
    |u(x,t)|^2 \, \dd x
    \notag\\
    &\le
    C
    \left(
        \int_{\mathbb{R}^d} \int_{\gamma_0}^{\infty}
        e^{2\gamma |x|^2+C_{0}\gamma - \frac{(2\gamma)^q}{qa_0^q}}
        (2\gamma)^{\frac{q-2}{2}}
        \, \dd\gamma \,
        |u(x,0)|^2 \, \dd x
    \right)^{1-t}
    \notag\\
    &\qquad \times
    \left(
        \int_{\mathbb{R}^d} \int_{\gamma_0}^{\infty}
        e^{2\gamma |x|^2+C_{0}\gamma - \frac{(2\gamma)^q}{qa_0^q}}
        (2\gamma)^{\frac{q-2}{2}}
        \, \dd\gamma \,
        |u(x,1)|^2 \, \dd x
    \right)^t.
    \label{high order exponential decay}
\end{align}
 For any \(a_0\) such that
\[
    \frac{a_0^\eta}{\eta}
    >
    \varrho_0
    :=
    \frac{1}{\eta}
    \left(
        4\gamma_0 \frac{2}{q-2}
    \right)^\eta,
\]
we have
\begin{equation}\label{approx}
    \int_{\gamma_0}^{\infty}
    e^{2\gamma |x|^2 - \frac{(2\gamma)^q}{qa_0^q}}
    (2\gamma)^{\frac{q-2}{2}}
    \, \dd\gamma
    \approx_{a_0,\eta}
    e^{\frac{a_0^\eta |x|^{2\eta}}{\eta}}.
\end{equation}
Moreover, since the additional factor \(e^{C_0\gamma}\) only changes
the prefactor and contributes a lower-order factor at the critical
point, we have
\begin{equation}\label{approx1}
    \int_{\gamma_0}^{\infty}
    e^{2\gamma |x|^2 +C_0\gamma- \frac{(2\gamma)^q}{qa_0^q}}
    (2\gamma)^{\frac{q-2}{2}}
    \, \dd\gamma
    \approx_{a_0,\eta}
    e^{\frac{a_0^\eta |x|^{2\eta}}{\eta}+O(|x|^{2\eta-2})}.
\end{equation}
Taking $a_0=(2\eta\tilde{\varrho})^{\frac{1}{\eta}}$ and 
substituting \eqref{approx} and \eqref{approx1} into
\eqref{high order exponential decay}, and absorbing the lower-order
factor \(e^{O(|x|^{2\eta-2})}\) into \(e^{\delta|x|^{2\eta}}\) with
\(\delta>0\) small enough so that \(2\tilde{\varrho}+\delta<2\varrho\), we obtain
\begin{equation}
    \int_{\mathbb{R}^d} e^{2\tilde{\varrho} |x|^{2\eta}} |u(t,x)|^2 \, dx
\le C
\left(
\int_{\mathbb{R}^d} e^{2\varrho |x|^{2\eta}} |u(0,x)|^2 \, dx
\right)^{1-t}
\left(
\int_{\mathbb{R}^d} e^{2\varrho |x|^{2\eta}} |u(1,x)|^2 \, dx
\right)^t,
\end{equation}
where \(C=C(\varrho,\tilde{\varrho},\eta)\). Uniformity as $\tilde{\varrho}\uparrow\varrho$ is unnecessary.


For the space-time gradient estimate, we apply the same integration procedure to \eqref{t and 1-t}. The factors \(\gamma^k\) appearing on the left-hand side produce, after integration in \(\gamma\), it gives polynomial factors \(|x|^{2k/(q-1)}\)
\begin{equation}
    \int_{\gamma_0}^{\infty} \gamma^k e^{2\gamma |x|^2 - \frac{(2\gamma)^q}{qa_0^q}} (2\gamma)^{\frac{q-2}{2}} \, d\gamma \le C_k a_0^{\frac{q}{2}+\eta k}\langle x\rangle^{\frac{2k}{q-1}} e^{\frac{a^\eta |x|^{2\eta}}{\eta}}.
\end{equation} Since \(\tilde{\varrho}<\varrho\), these polynomial factors can be absorbed by choosing \(\delta>0\) small enough so that \(\tilde{\varrho}+\delta<\varrho\).
 We also obtain
\begin{equation}
\Bigl\|
\sqrt{t(1-t)}\, e^{\tilde{\varrho} |x|^{2\eta}} \nabla_A u
\Bigr\|_{L^2_{t,x}}^2
+
\Bigl\|
\sqrt{t(1-t)}\, e^{\tilde{\varrho} |x|^{2\eta}} x u
\Bigr\|_{L^2_{t,x}}^2
\le C
\Bigl(
\|e^{\varrho |x|^{2\eta}} u(0,x)\|_{L^2_x}^2
+
\|e^{\varrho |x|^{2\eta}} u(1,x)\|_{L^2_x}^2
\Bigr).
\end{equation}
This completes the proof of Corollary \ref{cor-log}.
\end{proof}

\subsection{Logarithmic convexity for special metric}\label{subsec:logconvex-special}

We now consider the special metric case. Under the additional structure of the metric, one can improve the uncertainty principle to the quadratic exponential regime.

We begin with the corresponding logarithmic convexity estimate.

\begin{lemma}\label{Log-convex-special}
Suppose that $u\in L^{\infty}([0,1],L^{2}(\mathbb{R}^{d}))\cap L^{2}([0,1],H^{1}(\mathbb{R}^{d}))$ is a solution to 
\begin{equation}
\partial_{t}u=(a+ib)\Big(\operatorname{div}_{A}(G\nabla_{A}u)+V(x,t)u+F(x,t)\Big)
\end{equation}
in $\mathbb{R}^{d}\times[0,1]$, where $a>0$ and $b\in\mathbb{R}$. Assume that $G$ and 
$B$ satisfies \eqref{uniformly elliptic}, Hypothesis \ref{assm2} and \eqref{wyy2} respectively.

Let $\gamma>0$, and assume that
\begin{equation*}
\sup_{t\in[0,1]}\|V\|_{L^{\infty}}:=M_{1}<\infty,\qquad
\sup_{t\in[0,1]}\frac{\|e^{\gamma|\cdot|^{2}}F(\cdot,t)\|_{L^{2}}}{\|u(\cdot,t)\|_{L^{2}}}:=M_{2}<\infty.
\end{equation*}
Assume moreover that
\begin{align}
         &(a^{2}+b^{2})c_{d}\Big(16\gamma^{3}\Lambda\sup_{x'\in\mathbb{R}^{d-1}}(|x'|^{3}|\nabla\tilde{G}(x')|)+2\sup_{x'\in\mathbb{R}^{d-1}}|x'||\nabla \tilde{G}(x')|\|\nabla^{2}G\|_{L^{\infty}}\nonumber\\
&\qquad\qquad
+2\sup_{x'\in\mathbb{R}^{d-1}}|x'||\nabla^{3}\tilde{G}(x')|
    +4\|\nabla G\|^{2}_{L^{\infty}}+4\Lambda\|\nabla^{2}G\|_{L^{\infty}}+4\gamma\|(Gx)^{\top}B\|^{2}_{L^\infty}\Big):=\widetilde{M}_{G,A}<\infty.
\end{align}
If
\begin{align}
\|e^{\gamma|\cdot|^{2}}u(\cdot,0)\|_{L^{2}}+\|e^{\gamma|\cdot|^{2}}u(\cdot,1)\|_{L^{2}}<\infty,
\end{align}
and
\[
H(t)=\|e^{\gamma|\cdot|^{2}}u(t,\cdot)\|_{L^{2}},
\]
then $H(t)$ is finite and logarithmically convex on $[0,1]$. Moreover, there exists a universal  constant $N$ such that
\begin{equation}
    H(t)\leq e^{N[\widetilde{M}_{G,A}+\sqrt{a^{2}+b^{2}}(M_{1}+M_{2})+(a^{2}+b^{2})(M_{1}^{2}+M_{2}^{2})]}\|e^{\gamma|\cdot|^{2}}u(\cdot,0)\|_{L^{2}}^{1-t}\|e^{\gamma|\cdot|^{2}}u(\cdot,1)\|_{L^{2}}^{t}.\label{log convex-special}
\end{equation}
\end{lemma}

\begin{proof}
Let $v=e^{\gamma|x|^{2}}u$. From \eqref{equation of v}, we have
\begin{equation*}
    |\partial_{t}v-(\mathcal{S}+\mathcal{A})v|\leq\sqrt{a^{2}+b^{2}}\Big(M_{1}v+e^{\varphi}|F|\Big).
\end{equation*}
As in the proof of Lemma \ref{log convex lemma}, we apply Lemma \ref{abstract lemma1} with
\[
W=\sqrt{a^{2}+b^{2}}e^{\varphi}F.
\]
A direct computation gives
\begin{align}
    &\int_{\mathbb{R}^{d}}\bar{v}\Big(\partial_{t}\mathcal{S}+[\mathcal{S},\mathcal{A}] \Big)v\,\dd x\nonumber\\
    =&\int_{\mathbb{R}^{d}}2a\operatorname{Im}(G\partial_{t}A\nabla_{A}v)\bar{v}-4b\gamma Gx\cdot\partial_{t}A|v|^{2}\,\dd x\nonumber\\
    &+(a^{2}+b^{2})\Bigg\{16\gamma^{3}\int_{\mathbb{R}^{d}}G x\cdot\nabla(Gx\cdot x)|v|^{2}\,\dd x-\int_{\mathbb{R}^{d}}\mathcal{Q}^{2}(\gamma|x|^{2})|v|^{2}\,\dd x\nonumber\\
    &\qquad\qquad\qquad
    -8\gamma\operatorname{Im}\int_{\mathbb{R}^{d}}g_{jk}g_{lm}x_{l}B_{mk}v\overline{D_{j}v}\,\dd x+8\gamma\int_{\mathbb{R}^{d}}g_{jk}g_{lm}\delta_{kl}D_{m}v\overline{D_{j}v}\,\dd x\nonumber\\
    &\qquad\qquad\qquad
    +\sum_{j,k,l,m\geq2}\gamma\int_{\mathbb{R}^{d}}\Big(4g_{jk}\partial_{k}g_{lm}x_{l}D_{m}v\overline{D_{j}v}-4g_{lm}x_{l}\partial_{m}g_{jk}D_{k}v\overline{D_{j}v}\\&\qquad\qquad\qquad+4\partial_{j}g_{lm}x_{l}g_{jk}D_{k}v\overline{D_{m}v}\Big)\,\dd x\Bigg\}\nonumber\\
    =&\sum_{j=1}^{9}Z_{j}.
\end{align}

Since $A$ is time-independent, we have
\begin{equation}
    Z_1=Z_2=0.
\end{equation}

For $Z_{3}$ and $Z_{4}$, Leibniz's rule gives
\begin{align}
Z_{3}
=&16(a^{2}+b^{2})\gamma^{3}\int_{\mathbb{R}^{d}}Gx\cdot\nabla(Gx\cdot x)|v|^{2}\,\dd x\\
=&16(a^{2}+b^{2})\gamma^{3}\int_{\mathbb{R}^{d}}\Big(|Gx|^{2}|v|^{2}+\sum_{k,j,l,m\geq2}g_{kj}x_{j}\partial_{k}g_{lm}x_{l}x_{m}|v|^{2}\Big)\,\dd x\nonumber\\
\geq&
16(a^{2}+b^{2})\gamma^{3}\int_{\mathbb{R}^{d}}|Gx|^{2}|v|^{2}\,\dd x-16(a^{2}+b^{2})c_{d}\gamma^{3}\sup_{x^{\prime}\in\mathbb{R}^{d-1}}(|x'|^{3}|\nabla\tilde{G}(x')|)\int_{\mathbb{R}^{d}}\Lambda|v|^{2}\,\dd x,
\end{align}
and
\begin{align}
Z_{4}
=&-(a^{2}+b^{2})\int_{\mathbb{R}^{d}}\mathcal{Q}^{2}(\gamma|x|^{2})|v|^{2}\,\dd x\nonumber\\
=&-2(a^{2}+b^{2})\gamma\int_{\mathbb{R}^{d}}\Big(\sum_{j,k,l,m\geq2}\Big[\partial_{l}g_{lm}\partial_{m}\partial_{k}g_{kj}x_{j}+g_{lm}\partial_{l}\partial_{m}\partial_{k}g_{kj}x_{j}\Big]+\partial_{l}g_{lm}\partial_{k}g_{kj}\delta_{jm}\nonumber\\
&\qquad\qquad\qquad
+g_{lm}\partial_{m}\partial_{k}g_{kj}\delta_{jl}
    +g_{lm}\partial_{l}\partial_{k}g_{kj}\delta_{jm}
    +\partial_{l}g_{lm}\partial_{m}g_{kk}+g_{lm}\partial_{l}\partial_{m}g_{kk}\Big)|v|^{2}\,\dd x\nonumber\\
\geq&
-2c_{d}(a^{2}+b^{2})\gamma\Big(\sup_{x'\in\mathbb{R}^{d-1}}|x'||\nabla \tilde{G}(x')|\|\nabla^{2}G\|_{L^{\infty}}+\sup_{x'\in\mathbb{R}^{d-1}}|x'||\nabla^{3}\tilde{G}(x')|+2\|\nabla G\|^{2}_{L^{\infty}}\nonumber\\
&\qquad\qquad\qquad
+2\Lambda\|\nabla^{2}G\|_{L^{\infty}}\Big)\int_{\mathbb{R}^{d}}|v|^{2}\,\dd x.
\end{align}

For $Z_{5}$, Cauchy--Schwarz yields
\begin{align}
    Z_{5}
    =&8(a^{2}+b^{2})\gamma\operatorname{Im}\int_{\mathbb{R}^{d}}G\overline{\nabla_{A}v}\cdot\,(Gx)^{\top}B\, v\nonumber\\
    \geq&
    -4(a^{2}+b^{2})\gamma\|(Gx)^{\top}B\|^{2}_{L^{\infty}}\int_{\mathbb{R}^{d}}|v|^{2}\,\dd x
    -4(a^{2}+b^{2})\gamma\int_{\mathbb{R}^{d}}|G\nabla_{A}v|^{2}\,\dd x.
\end{align}
Moreover,
\begin{align}
    Z_{6}
    =&8(a^{2}+b^{2})\gamma\int_{\mathbb{R}^{d}}g_{jk}g_{mk}D_{m}v\overline{D_{j}v}\,\dd x
    =8(a^{2}+b^{2})\gamma\int_{\mathbb{R}^{d}}|G\nabla_{A}v|^{2}\,\dd x.
\end{align}

Finally,
\begin{align}
  Z_{7}+Z_{8}+Z_{9}
  \geq
  -12\gamma(a^{2}+b^{2})c_{d}\sup_{x^{\prime}\in\mathbb{R}^{d-1}}\big(|x'||\nabla \tilde{G}(x')|\big)\int_{\mathbb{R}^{d}}|G\nabla_{A}v|^{2}\,\dd x.
\end{align}
Combining all terms, we obtain
\begin{align}
&\int_{\mathbb{R}^{d}}\bar{v}\Big(\partial_{t}\mathcal{S}+[\mathcal{S},\mathcal{A}]\Big)v\,\dd x\nonumber\\
\geq &
\Big(-\gamma\,a^{2}+4\gamma(a^{2}+b^{2})-12\gamma(a^{2}+b^{2})c_{d}\sup_{x'\in\mathbb{R}^{d-1}}\big(|x'||\nabla \tilde{G}(x')|\big)\Big)\int_{\mathbb{R}^{d}}|G\nabla_{A}v|^{2}\,\dd x\nonumber\\
&-(a^{2}+b^{2})c_{d}\Big(16\gamma^{3}\Lambda\sup_{x'\in\mathbb{R}^{d-1}}(|x'|^{3}|\nabla\tilde{G}(x')|)+2\gamma\sup_{x'\in\mathbb{R}^{d-1}}|x'||\nabla \tilde{G}(x')|\|\nabla^{2}G\|_{L^{\infty}}\nonumber\\
&\qquad
+2\gamma\sup_{x'\in\mathbb{R}^{d-1}}|x'||\nabla^{3}\tilde{G}(x')|
    +4\gamma\|\nabla G\|^{2}_{L^{\infty}}+4\gamma\Lambda\|\nabla^{2}G\|_{L^{\infty}}+4\gamma\|(Gx)^{\top}B\|^{2}_{L^\infty}\Big)\int_{\mathbb{R}^{d}}|v|^{2}\,\dd x.
\end{align}

By the assumption on the metric,
\begin{align}
    -\gamma a^{2}+4\gamma(a^{2}+b^{2})-12\gamma(a^{2}+b^{2})c_{d}\sup_{x'\in\mathbb{R}^{d-1}}(|x'||\nabla\tilde{G}(x')|)\geq 2\gamma(a^{2}+b^{2}),
\end{align}
and therefore
\begin{align}
    \partial_{t}\mathcal{S}+[\mathcal{S},\mathcal{A}]\geq- \widetilde{M}_{G,A}.
\end{align}
The conclusion now follows exactly as in Lemma \ref{log convex lemma}.
\end{proof}	

\subsection{Parabolic regularization}\label{subsec:parabolic-reg}

In this subsection, we introduce a parabolic regularization that guarantees Gaussian decay at intermediate times from weighted assumptions at $t=0$ and $t=1$.
Recall that $A=A(x)$ and $G=G(x)$ are independent of time, and write $V(t,x)=V_1(x)+V_2(t,x)$ with $V_1$ real-valued and $M_1=\|V_1\|_{L^\infty}<\infty$. We set
\[
\mathcal H_{G,A} := \operatorname{div}_{A}(G\nabla_{A}\cdot),
\qquad
H:= \mathcal H_{G,A}+V_1(x),
\]
so that $\mathcal H_{G,A}$ and $H$ are independent of time. The closed, lower semibounded form
\[
\mathfrak h(f,g)=\int_{\mathbb R^d}G\nabla_Af\cdot\overline{\nabla_Ag}\,\dd x-\int_{\mathbb R^d}V_1f\bar g\,\dd x,
\qquad D(\mathfrak h)=H^1(\mathbb R^d),
\]
defines the self-adjoint realization of $-H$. The form domain is $H^1$ since $A$ is bounded and $G$ is uniformly elliptic. In particular $H\le M_1I$.  The spectral theorem gives
\begin{equation}\label{reg semigroup bounds}
\|e^{sH}\|_{L^2\to L^2}\le e^{sM_1}\quad(s\ge0),
\qquad
\|e^{itH}f\|_{L^2}=\|f\|_{L^2};
\end{equation}
see for instance \cite{P}.

Following \cite[Section 3.2, equation (3.5)]{BFGRV-JFA}, we write the equation in the equivalent form
\begin{equation}\label{reg original forcing}
\partial_tu=i\bigl(Hu+F\bigr),
\qquad
F(t,x):=V_2(t,x)u(t,x),
\end{equation}
that is, we put the time-dependent part of the potential into the inhomogeneous term. This decomposition allows $V_2$ to depend on both $t$ and $x$, while keeping $H$ time-independent and self-adjoint. Hence
its solution can be written in Duhamel form as
\begin{equation}
u(t)=e^{itH}u_0+i\int_0^t e^{i(t-s)H}\big(V_2(s)u(s)\big)\,\dd s.
\end{equation}

For $0<\varepsilon<1$ we define the regularized solution by applying the forward heat flow of $H$ to the exact solution,
\begin{equation}\label{reg definition}
u_{\varepsilon}(t):=e^{\varepsilon tH}u(t),
\qquad
F_{\varepsilon}(t):=\frac{i}{\varepsilon+i}\,e^{\varepsilon tH}\bigl(V_2(t)u(t)\bigr),
\qquad 0\le t\le1 .
\end{equation}
Since all semigroups in this formula are functions of the same operator $H$, the Duhamel formula above gives
\begin{equation}\label{para duhamel}
u_{\varepsilon}(t)
=
e^{(i+\varepsilon)tH}u_0
+
(\varepsilon+i)\int_0^t
e^{(i+\varepsilon)(t-s)H}
F_{\varepsilon}(s)\,ds,
\end{equation}
i.e. $u_\varepsilon$ solves
\begin{equation}
\partial_t u_{\varepsilon}
=
(\varepsilon+i)\bigl(\mathcal H_{G,A}u_{\varepsilon}+V_1u_{\varepsilon}+F_{\varepsilon}\bigr),
\qquad
u_\varepsilon(0)=u_0,
\end{equation}
so that \eqref{equation of u} holds with $a=\varepsilon$, $b=1$, $V=V_1$ and $F=F_\varepsilon$.
The dissipative energy estimate for this equation gives
\begin{equation}
u_{\varepsilon}
\in
L^\infty([0,1],L^2(\mathbb R^d))
\cap
L^2([0,1],H^1(\mathbb R^d)),
\qquad \varepsilon>0 \hbox{ fixed},
\end{equation}
since the dissipative term $-\varepsilon\int_{\mathbb{R}^d} G\nabla_Au_\varepsilon\cdot\overline{\nabla_Au_\varepsilon}\,\dd{x}$ is non-positive.
In particular, the endpoint identities are exact:
\begin{equation}\label{reg endpoints}
u_{\varepsilon}(0)=u(0),
\qquad
u_{\varepsilon}(1)=e^{\varepsilon H}u(1).
\end{equation}
 
We shall use the following heat estimate: for $s\ge0$ and $\gamma>0$,
\begin{equation}\label{reg heat gaussian}
\bigl\|e^{\frac{\gamma|x|^{2}}{1+4\Lambda\gamma s}}e^{sH}f\bigr\|_{L^2}
\le e^{sM_1}\bigl\|e^{\gamma|x|^{2}}f\bigr\|_{L^2}.
\end{equation}
Indeed, let $w(s)=e^{sH}f$, $a(s)=\frac{\gamma}{1+4\Lambda\gamma s}$ and $v_R=e^{a(s)\rho_R}w$ with $\rho_R(x)=\frac{|x|^{2}}{1+|x|^{2}/R^{2}}$. Then $\rho_R\uparrow|x|^{2}$ and $|\nabla\rho_R|^{2}\le4\rho_R$, while $a'+4\Lambda a^{2}=0$; hence the weighted energy identity gives
\begin{equation}
\frac{1}{2}\frac{\dd}{\dd s}\|v_R\|_{L^2}^{2}
\le M_1\|v_R\|_{L^2}^{2}
+\int_{\mathbb{R}^d}\bigl(a'(s)\rho_R+a(s)^{2}G\nabla\rho_R\cdot\nabla\rho_R\bigr)|v_R|^{2}\,\dd x
\le M_1\|v_R\|_{L^2}^{2}.
\end{equation}
The weights are bounded for fixed $R$, so this identity follows by form approximation; Gronwall's inequality and Fatou's lemma, applied as $R\to\infty$, give \eqref{reg heat gaussian}.
Assume now that $e^{\gamma|x|^{2}}u(0)$ and $e^{\gamma|x|^{2}}u(1)$ belong to $L^{2}$, and set
\begin{equation}\label{reg gamma}
\gamma_\varepsilon:=\frac{\gamma}{1+4\Lambda\gamma\varepsilon}.
\end{equation}
By \eqref{reg endpoints} and $\gamma_\varepsilon\le\gamma$ we have
\begin{equation}\label{reg endpoint zero}
\bigl\|e^{\gamma_\varepsilon|x|^{2}}u_\varepsilon(0)\bigr\|_{L^2}
\le\bigl\|e^{\gamma|x|^{2}}u(0)\bigr\|_{L^2}.
\end{equation}
For the endpoint $t=1$ we apply \eqref{reg heat gaussian} with $f=u(1)$ and $s=\varepsilon$:
\[
\Bigl\|e^{\frac{\gamma|x|^{2}}{1+4\Lambda\gamma\varepsilon}}e^{\varepsilon H}u(1)\Bigr\|_{L^2}
\le e^{\varepsilon M_1}\bigl\|e^{\gamma|x|^{2}}u(1)\bigr\|_{L^2}.
\]
Here the weight on the left is $e^{\gamma_\varepsilon|x|^{2}}$, by \eqref{reg gamma}, and the function is $u_\varepsilon(1)$, by \eqref{reg endpoints}. Hence
\begin{equation}\label{reg endpoint bounds}
\bigl\|e^{\gamma_\varepsilon|x|^{2}}u_\varepsilon(1)\bigr\|_{L^2}
\le e^{\varepsilon M_1}\bigl\|e^{\gamma|x|^{2}}u(1)\bigr\|_{L^2}.
\end{equation}
Furthermore, \eqref{reg semigroup bounds} gives
\begin{equation}
\|u_\varepsilon(t)\|_{L^2}\le e^{\varepsilon M_1}\|u(t)\|_{L^2},
\qquad 0\le t\le1 .
\end{equation}
Finally, by \eqref{reg definition} and $\left|\frac{i}{\varepsilon+i}\right|=\frac{1}{\sqrt{1+\varepsilon^{2}}}$,
\[
\bigl\|e^{\gamma_\varepsilon|x|^{2}}F_\varepsilon(t)\bigr\|_{L^2}
=\frac{1}{\sqrt{1+\varepsilon^{2}}}
\bigl\|e^{\gamma_\varepsilon|x|^{2}}e^{\varepsilon tH}\bigl(V_2(t)u(t)\bigr)\bigr\|_{L^2}.
\]
 Using  Hypothesis \ref{assum3} and $t\le1$, we obtain the weighted bound for the regularized source,
\begin{equation}\label{reg forcing bound}
\bigl\|e^{\gamma_\varepsilon|x|^{2}}F_\varepsilon(t)\bigr\|_{L^2}
\le\frac{e^{\varepsilon tM_1}}{\sqrt{1+\varepsilon^{2}}}
\bigl\|e^{\gamma|x|^2}V_2(t)u(t)\bigr\|_{L^2}
\le e^{\varepsilon M_1}M_2\|u(t)\|_{L^2}.
\end{equation}

To applying the logarithmic convexity lemma, we need the ratio $\tfrac{\|e^{\gamma_\varepsilon|\cdot|^{2}}F_\varepsilon\|_{L^2}}{\|u_\varepsilon\|_{L^2}}$ is bounded uniformly in $\varepsilon$ in place of $M_2$. We first control the denominator. Put $K=\int_0^1\|\operatorname{Im}V_2(t)\|_{L^\infty}\,\dd t$. The $L^2$ energy identity for \eqref{reg original forcing} gives
\begin{equation}
e^{-K}\|u_0\|_{L^2}\le\|u(t)\|_{L^2}\le e^{K}\|u_0\|_{L^2},
\qquad 0\le t\le1 .
\end{equation}
If $u_0=0$ then $u\equiv0$ and there is nothing to prove; otherwise, since $e^{sH}$ is strongly continuous and $\{u(t):0\le t\le1\}$ is compact in $L^2$, we have $\sup_{0\le t\le1}\|u_\varepsilon(t)-u(t)\|_{L^2}\to0$ as $\varepsilon\downarrow0$. To see the uniformity, approximate the compact set by a finite $L^2$ net and use \eqref{reg semigroup bounds} on the approximation errors. Hence, for all sufficiently small $\varepsilon>0$,
\[
\inf_{0\le t\le1}\|u_\varepsilon(t)\|_{L^2}\ge\tfrac12e^{-K}\|u_0\|_{L^2},
\]
and \eqref{reg forcing bound} yields the uniform bound
\begin{equation}\label{reg forcing ratio}
\operatorname*{ess\,sup}_{0\le t\le1}
\frac{\|e^{\gamma_\varepsilon|\cdot|^{2}}F_\varepsilon(t)\|_{L^2}}{\|u_\varepsilon(t)\|_{L^2}}
\le 2e^{M_1+2K}M_2 .
\end{equation}

We may therefore apply the Gaussian weighted interpolation and gradient estimates of Lemma \ref{log convex lemma} under Hypothesis \ref{assm1}, or of Lemma \ref{Log-convex-special} under Hypothesis \ref{assm2}, to \eqref{reg definition}, with
\[
a=\varepsilon,\qquad b=1,\qquad V=V_1,\qquad F=F_\varepsilon,\qquad \gamma=\gamma_\varepsilon .
\]
%
This completes the proof of logarithmic convexity.
%

\section{Proof of Theorem \ref{thm1}}\label{sec:proof-thm1}

In this section, we prove that if a solution exhibits super-quadratic exponential decay at the two times $t=0$ and $t=1$, then it must vanish identically. The main ingredient is the Carleman estimate stated below, which provides the lower bound mechanism needed in the contradiction argument.

\begin{proposition}[Carleman estimate]\label{Carleman estimate}
Suppose that $d\geq 2$, $G$ satisfies \eqref{uniformly elliptic} and \eqref{wyy1}, and the magnetic field $B$ satisfies \eqref{wyy2}.
Fix $r_0>0$, $\ell\in\mathbb N$, $\ell\geq1$, and a real-valued
$\varphi\in C^\infty([0,1])$. Set
\begin{equation}\label{new weight parameters}
\begin{aligned}
	&a_\ell=2^{-\ell},\qquad
	\beta_\ell=1+\frac23a_\ell,\qquad
	p_\ell=\frac{6}{3-a_\ell}=\frac{4}{3-\beta_\ell},\qquad\\ &
	\Phi(x,t)=\mu\left|\frac{x}{R}\right|^{4/3}
	+\mu^{\beta_\ell}\varphi(t).
    \end{aligned}
\end{equation}
There is a dimensional constant $C_d\geq1$ such that, if
\begin{equation}\label{radial smallness}
	\varepsilon_0\leq\min\left\{1,\frac{\lambda^2}{9C_d\Lambda}\right\},
\end{equation}
the following assertion holds. Define
\begin{align}
	\mathfrak A_1(R)=R^{4/3}
	\left(\frac{81C_Q}{32\lambda^2}\right)^{1/2},
	\mathfrak A_2(R)=R^{4/3}
	\left(\frac{81C_B}{32\lambda^2}\right)^{1/2},
	\mathfrak A_3(R)=R^{p_\ell}
	\left(\frac{81\|\varphi^{\prime\prime}\|_{L^\infty}}{32\lambda^2}
	\right)^{1/(3-\beta_\ell)},\label{new carleman thresholds}
\end{align}
where \begin{equation}
		\Xi=\Lambda^2+\Lambda\varepsilon_0+\varepsilon_0^2,\qquad
		C_Q=C_d\Xi r_0^{-8/3},\qquad
		C_B=\frac{16\Lambda^2}{\lambda^2}M^2r_0^{-2/3}.
	\end{equation}
Then for $R\geq1$ and
$\mu\geq\max\{1,\mathfrak A_1(R),\mathfrak A_2(R),\mathfrak A_3(R)\}$,
one has
\begin{align}\label{carleman-lemma}
	&\mathfrak c_1\frac{\mu}{R^{4/3}}
	\int_0^1\int_{\mathbb R^d}|x|^{-2/3}|\nabla_Af|^2\,\dd x\dd t
	+\mathfrak c_2\frac{\mu^3}{R^4}
	\int_0^1\int_{\mathbb R^d}|f|^2\,\dd x\dd t\nonumber\\
	&\hspace{3em}\leq
	\int_0^1\int_{\mathbb R^d}
	\left|e^\Phi(\partial_t-i\operatorname{div}_{A}(G\nabla_{A}\cdot))(e^{-\Phi}f)\right|^2\,\dd x\dd t,
\end{align}
for every $f\in C_c^\infty((0,1)\times
(\mathbb R^d\setminus\overline{B_{r_0}}))$, where $
\mathfrak c_1=\frac{4\lambda^2}{9},\quad
\mathfrak c_2=\frac{32\lambda^2}{81}$.

\end{proposition}

\begin{proof}
Write
\begin{equation}
f=e^\Phi h,\qquad
\Phi=\frac{\mu}{R^{4/3}}|x|^{4/3}+\mu^{\beta_\ell}\varphi(t).
\end{equation}
Since the test function $f$ are supported away from the origin, all derivatives
of $|x|^{4/3}$ used here are smooth on $\mathbb{R}^d\backslash\overline{B_{r_0}}$. Direct differentiation gives
\begin{align}\label{radial derivatives}
	\partial_j|x|^{4/3}&=\frac43|x|^{-2/3}x_j,\qquad
	\partial_j\partial_k|x|^{4/3}=\frac43|x|^{-2/3}\delta_{jk}
	-\frac89|x|^{-8/3}x_jx_k.
	\end{align}
	Then we have
	\begin{align}\label{radial derivatives2}
	\nabla^2|x|^{4/3}&\geq\frac49|x|^{-2/3}I,\qquad
	\bigl|\nabla^m|x|^{4/3}\bigr|\leq C_d |x|^{4/3-m}\quad(1\leq m\leq4).
\end{align}
By conjugating, we have
\begin{align*}
    &e^{\Phi}(\partial_{t}-i\operatorname{div}_{A}(G\nabla_{A}\cdot))h\\
=& e^{\Phi}(\partial_{t}-i\operatorname{div}_{A}(G\nabla_{A}\cdot))\big(e^{-\Phi}f\big)\\
=&(\partial_{t}-\mathcal{S}-\mathcal{A})f.
\end{align*}
We see that $\mathcal{S}$ and $\mathcal{A}$ are given by \eqref{symmetry part} and \eqref{antisymmetry part}, respectively, with $a=0$ and $b=1$ and $\varphi:=\Phi$ here. Here the symmetric and skew-symmetric parts are
\begin{align}
\mathcal S=-i\frac{\mu}{R^{4/3}}\Bigl(\mathcal Q|x|^{4/3}+2G\nabla|x|^{4/3}\cdot\nabla_A\Bigr)
+\mu^{\beta_\ell}\varphi^\prime(t),
\end{align}
\begin{align}
\mathcal A=i\Bigl(\operatorname{div}_{A}(G\nabla_{A}\cdot)+\frac{\mu^2}{R^{8/3}}G\nabla|x|^{4/3}\cdot\nabla|x|^{4/3}\Bigr).
\end{align}
Taking the $L^2$ inner product, we obtain
\begin{align}
    \big\|(\partial_{t}-\mathcal{S}-\mathcal{A})f\big\|^{2}_{L^{2}(\mathbb R^{d+1})}
    =&\|(\partial_{t}-\mathcal{A})f\|^{2}_{L^{2}(\mathbb R^{d+1})}+\|\mathcal{S}f\|^{2}_{L^{2}(\mathbb R^{d+1})}\nonumber\\
    &-2\operatorname{Re}\int_{\mathbb{R}^{d+1}}\mathcal{S}f\overline{(\partial_{t}-\mathcal{A})f}\,\dd x\dd t\\
    \geq&\int_{\mathbb{R}^{d+1}}\big(\partial_{t}\mathcal{S}+[\mathcal{S},\mathcal{A}]\big)f\bar{f}\,\dd x\dd t,
\end{align}

We now compute the last term using \eqref{S time derivative} and \eqref{formula of commutator} with $a=0$, $b=1$. Since $A$ and
$G$ are independent of $t$, and $\nabla\Phi_t=0$, one has
$\mathcal S_t=\mu^{\beta_\ell}\varphi^{\prime\prime}(t)$. 
Recall that $\mathcal{Q}=\operatorname{div}(G\nabla\cdot)$. It holds that
\begin{align*}
    &\int_{\mathbb{R}^{d+1}}(\partial_{t}\mathcal{S}+[\mathcal{S},\mathcal{A}])f\bar{f}\,\dd x\dd t=\sum_j^8 K_j,
\end{align*}
where
\begin{align*}
	K_1&=2\frac{\mu^3}{R^4}\int_{\mathbb{R}^{d+1}} G\nabla|x|^{4/3}\cdot
	\nabla\bigl(G\nabla|x|^{4/3}\cdot\nabla|x|^{4/3}\bigr)|f|^2\,\dd x\dd t,\\
	K_2&=-\frac{\mu}{R^{4/3}}\int_{\mathbb{R}^{d+1}}\mathcal Q^2|x|^{4/3}\,|f|^2\,\dd x\dd t,\\
	K_3&=4\frac{\mu}{R^{4/3}}\operatorname{Im}\int_{\mathbb{R}^{d+1}}
	g_{jk}g_{lm}\partial_l|x|^{4/3}(\partial_kA_m-\partial_mA_k)
	f\overline{D_jf}\,\dd x\dd t,\\
	K_4&=4\frac{\mu}{R^{4/3}}\operatorname{Re}\int_{\mathbb{R}^{d+1}}
	g_{jk}g_{lm}\partial_k\partial_l|x|^{4/3}D_mf\overline{D_jf}\,\dd x\dd t,\\
	K_5&=2\frac{\mu}{R^{4/3}}\operatorname{Re}\int_{\mathbb{R}^{d+1}}
	g_{jk}\partial_kg_{lm}\partial_l|x|^{4/3}D_mf\overline{D_jf}\,\dd x\dd t,\\
	K_6&=-2\frac{\mu}{R^{4/3}}\operatorname{Re}\int_{\mathbb{R}^{d+1}}
	g_{lm}\partial_l|x|^{4/3}\partial_mg_{jk}D_kf\overline{D_jf}\,\dd x\dd t,\\
	K_7&=2\frac{\mu}{R^{4/3}}\operatorname{Re}\int_{\mathbb{R}^{d+1}}
	\partial_jg_{lm}\partial_l|x|^{4/3}g_{jk}D_kf\overline{D_mf}\,\dd x\dd t,\\
	K_8&=\mu^{\beta_\ell}\int_{\mathbb{R}^{d+1}}\varphi^{\prime\prime}(t)|f|^2\,\dd x\dd t.
\end{align*}

We estimate these terms one by one. First, according to 
\begin{align*}
	G\nabla|x|^{4/3}\cdot\nabla|x|^{4/3}
	&=\frac{16}{9}|x|^{-4/3}x^\top Gx,
	\end{align*}
	we get
	\begin{align*}
	K_1&=\frac{\mu^3}{R^4}\int_{\mathbb{R}^{d+1}}\left(
	\frac{256}{27}|x|^{-2}|Gx|^2
	-\frac{512}{81}|x|^{-4}(x^\top Gx)^2
	+\frac{128}{27}|x|^{-2}g_{aj}x_j
	(\partial_ag_{lm})x_lx_m\right)|f|^2\,\dd x\dd t.
\end{align*}
Using $(x^\top Gx)^2\leq |x|^2|Gx|^2$, \eqref{uniformly elliptic}, and \eqref{wyy1},
it holds that
\begin{align}\label{radial K1}
	K_1&\geq\frac{\mu^3}{R^4}\left(\frac{256}{81}\lambda^2
	-\frac{128}{27}C_d\Lambda\sup_x |x||\nabla G(x)|\right)
	\int_{\mathbb{R}^{d+1}}|f|^2\,\dd x\dd t\nonumber\\
	&\geq\frac{128}{81}\lambda^2\frac{\mu^3}{R^4}\int_{\mathbb{R}^{d+1}}|f|^2\,\dd x\dd t,
\end{align}
after the choice \eqref{radial smallness}.

For $K_2$, we expand the fourth-order scalar expression explicitly:
\begin{align*}
	\mathcal Q^2|x|^{4/3}={}&
	(\partial_i g_{ij})(\partial_j\partial_k g_{kl})\partial_l|x|^{4/3}
	+(\partial_i g_{ij})(\partial_k g_{kl})\partial_j\partial_l|x|^{4/3}
	+(\partial_i g_{ij})(\partial_j g_{kl})\partial_k\partial_l|x|^{4/3}\\
	&+(\partial_i g_{ij})g_{kl}\partial_j\partial_k\partial_l|x|^{4/3}
	+g_{ij}(\partial_i\partial_j\partial_k g_{kl})\partial_l|x|^{4/3}
	+g_{ij}(\partial_j\partial_k g_{kl})\partial_i\partial_l|x|^{4/3}\\
	&+g_{ij}(\partial_i\partial_k g_{kl})\partial_j\partial_l|x|^{4/3}
	+g_{ij}(\partial_k g_{kl})\partial_i\partial_j\partial_l|x|^{4/3}
	+g_{ij}(\partial_i\partial_j g_{kl})\partial_k\partial_l|x|^{4/3}\\
	&+g_{ij}(\partial_j g_{kl})\partial_i\partial_k\partial_l|x|^{4/3}
	+g_{ij}(\partial_i g_{kl})\partial_j\partial_k\partial_l|x|^{4/3}
	+g_{ij}g_{kl}\partial_i\partial_j\partial_k\partial_l|x|^{4/3}.
\end{align*}
Consequently,
\begin{align*}
	\bigl|\mathcal Q^2|x|^{4/3}\bigr|\leq C_d\big(&
	\Lambda^2\bigl|\nabla^4|x|^{4/3}\bigr|
	+\Lambda|\nabla G|\bigl|\nabla^3|x|^{4/3}\bigr|
	+(\Lambda|\nabla^2G|+|\nabla G|^2)\bigl|\nabla^2|x|^{4/3}\bigr|\\
	&+(\Lambda|\nabla^3G|+|\nabla G||\nabla^2G|)\bigl|\nabla|x|^{4/3}\bigr|\big).
\end{align*}
Hypothesis \eqref{wyy1} implies
$|x|^j|\nabla^jG|\leq\varepsilon_0$ for $j=1,2,3$.
It follows from \eqref{radial derivatives} and \eqref{radial derivatives2} that
\[
\bigl|\mathcal Q^2|x|^{4/3}\bigr|\leq C_d\Xi |x|^{-8/3},
\qquad
K_2\geq-C_Q\frac{\mu}{R^{4/3}}\int_{\mathbb{R}^{d+1}}|f|^2\,\dd x\dd t.
\]

For $K_3$, the radial form of $\nabla|x|^{4/3}$ gives
\[
\bigl|(G\nabla|x|^{4/3})^\top B\bigr|
=\frac43|x|^{-2/3}|(Gx)^\top B|
\leq\frac43M|x|^{-2/3}.
\]
Young's inequality therefore yields
\begin{align}\label{K3}
	|K_3|&\leq\frac{16\Lambda M\mu}{3R^{4/3}}
	\int_{\mathbb{R}^{d+1}} |x|^{-2/3}|f||\nabla_Af|\,\dd x\dd t\nonumber\\
	&\leq\frac{4\lambda^2\mu}{9R^{4/3}}
	\int_{\mathbb{R}^{d+1}} |x|^{-2/3}|\nabla_Af|^2\,\dd x\dd t
	+\frac{16\Lambda^2M^2\mu}{\lambda^2R^{4/3}}
	\int_{\mathbb{R}^{d+1}} |x|^{-2/3}|f|^2\,\dd x\dd t\nonumber\\
	&\leq\frac{4\lambda^2\mu}{9R^{4/3}}
	\int_{\mathbb{R}^{d+1}} |x|^{-2/3}|\nabla_Af|^2\,\dd x\dd t
	+C_B\frac{\mu}{R^{4/3}}\int_{\mathbb{R}^{d+1}}|f|^2\,\dd x\dd t.
\end{align}
In particular, the coefficient of the magnetic error involves $M^2$.

For $K_4$, the Hessian calculation gives
\begin{align*}
	K_4&=\frac{16\mu}{3R^{4/3}}\int_{\mathbb{R}^{d+1}} |x|^{-2/3}|G\nabla_Af|^2\,\dd x\dd t
	-\frac{32\mu}{9R^{4/3}}\int_{\mathbb{R}^{d+1}} |x|^{-8/3}
	|(Gx)\cdot\nabla_Af|^2\,\dd x\dd t\\
	&\geq\frac{16\lambda^2\mu}{9R^{4/3}}
	\int_{\mathbb{R}^{d+1}} |x|^{-2/3}|\nabla_Af|^2\,\dd x\dd t.
\end{align*}

For $K_5$, $K_6$, $K_7$ a direct computation gives
\begin{align}\label{K57}
	K_5+K_6+K_7
	&\geq-8C_d\Lambda\frac{\mu}{R^{4/3}}\sup_x(|x||\nabla G|)
	\int_{\mathbb{R}^{d+1}} |x|^{-2/3}|\nabla_Af|^2\,\dd x\dd t\nonumber\\
	&\geq-\frac{8\lambda^2\mu}{9R^{4/3}}
	\int_{\mathbb{R}^{d+1}} |x|^{-2/3}|\nabla_Af|^2\,\dd x\dd t.
\end{align}
Combining this with \eqref{K3}, we find
\begin{equation}\label{zrh}
	K_3+K_4+K_5+K_6+K_7
	\geq\frac{4\lambda^2\mu}{9R^{4/3}}\int_{\mathbb{R}^{d+1}} |x|^{-2/3}|\nabla_Af|^2\,\dd x\dd t
	-C_B\frac{\mu}{R^{4/3}}\int_{\mathbb{R}^{d+1}}|f|^2\,\dd x\dd t.
\end{equation}
Finally, since $\varphi''$ has no prescribed sign, we have
\[
K_8\geq-\mu^{\beta_\ell}\|\varphi^{\prime\prime}\|_{L^\infty}\int_{\mathbb{R}^{d+1}}|f|^2\,\dd x\dd t.
\]
Collecting all eight estimates for $K_j$ gives
\begin{align*}
	\|e^\Phi(\partial_t-i\operatorname{div}_{A}(G\nabla_{A}\cdot))h\|_2^2\geq{}&
	\frac{4\lambda^2\mu}{9R^{4/3}}\int_{\mathbb{R}^{d+1}} |x|^{-2/3}|\nabla_Af|^2\,\dd x\dd t\\
	&+\left(\frac{128\lambda^2}{81}\frac{\mu^3}{R^4}
	-C_Q\frac{\mu}{R^{4/3}}-C_B\frac{\mu}{R^{4/3}}
	-\mu^{\beta_\ell}\|\varphi''\|_{L^\infty}\right)\int_{\mathbb{R}^{d+1}}|f|^2\,\dd x\dd t.
\end{align*}
The three conditions $\mu\geq\mathfrak A_j(R)$ imply, respectively,
\[
C_Q\frac{\mu}{R^{4/3}}\leq\frac{32\lambda^2}{81}\frac{\mu^3}{R^4},\qquad
C_B\frac{\mu}{R^{4/3}}\leq\frac{32\lambda^2}{81}\frac{\mu^3}{R^4},\qquad
\mu^{\beta_\ell}\|\varphi^{\prime\prime}\|_{L^\infty}
\leq\frac{32\lambda^2}{81}\frac{\mu^3}{R^4}.
\]

For the last condition, we used
\[
\mu^{3-\beta_\ell}\geq
\frac{81\|\varphi''\|_{L^\infty}}{32\lambda^2}R^4,
\qquad \frac{4}{3-\beta_\ell}=p_\ell.
\]

Therefore, we obtain the desired estimate
\begin{equation}\label{carleman estimate formula1}
	\|e^\Phi(\partial_t-i\operatorname{div}_{A}(G\nabla_{A}\cdot))h\|_2^2\geq
	\mathfrak c_1\frac{\mu}{R^{4/3}}\int_{\mathbb{R}^{d+1}} |x|^{-2/3}|\nabla_Af|^2\,\dd x\dd t
	+\mathfrak c_2\frac{\mu^3}{R^4}\int_{\mathbb{R}^{d+1}}|f|^2\,\dd x\dd t,
\end{equation}

The integrated identity can first be established for smooth approximations
of $A$. For $A\in C^1_{\rm loc}$, pass to the limit in $C^1$ on a
neighborhood of the compact support of $f$, and then apply the bounds for
the original magnetic field. The approximations need not satisfy a
uniform global magnetic bound. No derivatives of $A$ beyond those
contained in $B$ enter the estimates.

which completes the proof. 
\end{proof}

Following the strategy of \cite{EKPV-CPDE}, we now prove a lower bound for nontrivial solutions.

\begin{theorem}\label{mass lower}Assume the coefficient hypotheses and smallness condition of Proposition
	\ref{Carleman estimate}. Fix $\ell\geq1$, 
and let
\[
u\in L^{\infty}([0,1],L^{2}(\mathbb R^d))\cap L_{\rm loc}^{2}((0,1),H^{1}(\mathbb R^d))
\]
be a solution of
\begin{equation*}
\partial_{t}u=i\operatorname{div}_{A}\big(G\nabla_{A}u\big)+Vu
\end{equation*}
with $V(t,x)\in L^{\infty}(\R^{d+1})$.

Assume, in addition, that $E_{1}$, $E_{2}$, $R_{0}$, and $\varepsilon$ are positive constants such that
\begin{align}
\int_{\frac{1}{8}}^{\frac{7}{8}}\int_{\mathbb R^d}\Big(|u|^{2}+|\nabla_{A} u|^{2}\Big)\,\dd x\dd t\leq E_{1}^{2}<\infty
\end{align}
and
\begin{align}\label{annular mass hypothesis}
\int_{\frac{1}{4}}^{\frac{3}{4}}\int_{B_{R_{0}}\backslash B_{2\varepsilon}}|u|^{2}\,\dd x\dd t\geq E_{2}^{2}.
\end{align}
Then there exist positive constants
\[
R_{1}=R_{1}(d,\lambda,\Lambda,\|(Gx)^{\top}B\|_{L^\infty},\varepsilon,\|V\|_{L^{\infty}},E_{1},E_{2},R_{0},\ell),
\]
\[
C_{0}=C_{0}(\lambda,\varepsilon),
\qquad
C=C(d,\lambda,\Lambda,\|(Gx)^{\top}B\|,\varepsilon,\varepsilon_0,\|V\|_{L^{\infty}},E_{2},\ell),
\]
such that
\begin{align}\label{goal estimate}
\delta(R)=\int_{\frac{1}{8}}^{\frac{7}{8}}\int_{B_{R}\backslash B_{R-1}}|u|^{2}+|\nabla_{A} u|^{2}\,\dd x\dd t\geq Ce^{-C_{0}R^{\frac{6}{3-2^{-\ell}}}}
\end{align}
for all $R\ge R_{1}$.
\end{theorem}

\begin{proof}
	We first choose two independent time functions
	\[
	\varphi\in C_c^\infty((3/16,13/16)),\qquad
	\eta\in C_c^\infty((1/8,7/8)),
	\]
	such that
	\[
	0\leq\varphi\leq3,\quad
	\varphi=3\ \hbox{on }[1/4,3/4],\qquad
	0\leq\eta\leq1,\quad
	\eta=1\ \hbox{on }[3/16,13/16].
	\]
	In particular,
	\begin{equation}\label{time cutoff separation}
		\varphi=0\quad\hbox{on }\operatorname{supp}\eta'.
	\end{equation}
Next, define $\theta\in C^{\infty}(\mathbb R^d)$ by
\begin{equation}
    \theta(r)=\begin{cases}
        0,& r\leq1,\\
        1,& r\geq2.
    \end{cases}
\end{equation}
For any $\varepsilon>0$, let
\[
\theta_{\varepsilon}(x):=\theta\Big(\frac{|x|}{\varepsilon}\Big).
\]
Let $\kappa_{R}(x)$ be a smooth bump function such that
\[
\kappa_R =1 \quad \text{on } B_{R-1},
\qquad
\kappa_{R}=0 \quad \text{on } B_{R}^{c}.
\]
Moreover,
\begin{equation}
    |\varphi^\prime|+|\theta^\prime|+\sup_{j=1,2}|\nabla^{j}\kappa_{R}|\leq C,
    \qquad
    |\theta^\prime_{\varepsilon}|\leq \frac{C}{\varepsilon}.
\end{equation}
For $R>\max\{2,R_0+1,2\varepsilon+1\}$, set
\begin{equation}\label{new cutoff h}
	h(x,t)=\eta(t)\theta_\varepsilon(x)\kappa_R(x)u(x,t),\qquad
	f=e^\Phi h,
\end{equation}
where $\Phi$ is given by \eqref{new weight parameters}.
The basic support properties are now
\begin{equation}\label{suppg1}
	\operatorname{supp}h\subset
	\operatorname{supp}\eta\times
	(\overline{B_R}\setminus B_\varepsilon),
	\qquad h(\cdot,t)=0\quad
	\hbox{for }t\in[0,1/8]\cup[7/8,1],
\end{equation}
and
\begin{equation}\label{h equal u}
	h=u\quad\hbox{on }[1/4,3/4]\times
	(B_{R_0}\setminus B_{2\varepsilon}).
\end{equation}
Thus the vanishing at the time endpoints holds for every spatial point.

We now apply the Carleman estimate from Proposition \ref{Carleman estimate} to
\[
f=e^{\Phi}h.
\]
We compute the cutoff errors before applying the estimate. For a smooth
scalar function $b$, the covariant product rule gives
\[
(\partial_t-i\operatorname{div}_{A}(G\nabla_{A}\cdot))(bu)=b(\partial_t-i\operatorname{div}_{A}(G\nabla_{A}\cdot))u+b_tu
-i(\mathcal Qb)u-2iG\nabla b\cdot\nabla_Au.
\]
The supports of $\nabla\theta_\varepsilon$ and $\nabla\kappa_R$ are
disjoint, so no mixed spatial cutoff terms remain. Since $(\partial_t-i\operatorname{div}_{A}(G\nabla_{A}\cdot))u=Vu$,
we obtain
\begin{align}\label{aha}
	&(\partial_t-i\operatorname{div}_{A}(G\nabla_{A}\cdot))h=Vh+J_1+J_2+J_3+J_4+J_5,\nonumber\\&
	J_1=\eta'\theta_\varepsilon\kappa_Ru,
	\\& J_2=-i\eta(\mathcal Q\theta_\varepsilon)u,
	\quad J_3=-2i\eta G\nabla\theta_\varepsilon\cdot\nabla_Au,\\&
	J_4=-i\eta(\mathcal Q\kappa_R)u,\quad
	J_5=-2i\eta G\nabla\kappa_R\cdot\nabla_Au.
\end{align}
Here $\kappa_R=1$ on the supports of $J_2,J_3$, and
$\theta_\varepsilon=1$ on those of $J_4,J_5$.
In particular, $h\in L^2_tH^1_x$ and $(\partial_t-i\operatorname{div}_{A}(G\nabla_{A}\cdot))h\in L^2_{t,x}$.

Write
\[
K_R:=\frac{\mu^3}{R^4},\qquad
W_R:=6\mu^{\beta_\ell}
+2\mu\left(\frac{2\varepsilon}{R}\right)^{4/3},
\qquad
N:=\int_0^1\int_{\mathbb R^d}e^{2\Phi}|h|^2\,\dd x\dd t.
\]
On $[1/4,3/4]\times(B_{R_0}\setminus B_{2\varepsilon})$ we have
$\varphi=3$. On $|x|\ge2\varepsilon$, we have
$2\Phi\ge6\mu^{\beta_\ell}+2\mu(2\varepsilon/R)^{4/3}=W_R$. By
\eqref{h equal u} and \eqref{annular mass hypothesis} this gives
$N\ge e^{W_R}E_2^2$, and therefore
\begin{equation}\label{new lhs lower bound}
	\operatorname{LHS}\text{ of }\eqref{carleman-lemma}
	\geq\mathfrak c_2K_RN
	\geq\frac{\mathfrak c_2}{2}K_RN
	+\frac{\mathfrak c_2}{2}K_Re^{W_R}E_2^2.
\end{equation}

We next estimate the right-hand side, keeping the time, inner, and outer
cutoff errors separate. On $\operatorname{supp}J_1$, one has
$\varphi=0$ and $|x|\leq R$. Therefore
\begin{align}\label{new time error}
	\int_0^1\int_{\mathbb{R}^{d}} e^{2\Phi}|J_1|^2\,\dd x\dd t
	&\leq C_1e^{2\mu}
	\int_{1/8}^{7/8}\int_{\mathbb{R}^{d}}|u|^2\,\dd x\dd t
	\leq C_1e^{2\mu}E_1^2.
\end{align}
For the inner errors, expand
\[
\mathcal Q\theta_\varepsilon
=(\partial_jg_{jk})\partial_k\theta_\varepsilon
+g_{jk}\partial_{jk}\theta_\varepsilon.
\]
Thus $|J_2+J_3|\leq C_\varepsilon(|u|+|\nabla_Au|)$ on their support,
which is contained in $(1/8,7/8)\times
(B_{2\varepsilon}\setminus B_\varepsilon)$. Since $\varphi\leq3$,
\begin{align}\label{new inner error}
	\int_0^1\int_{\mathbb{R}^{d}} e^{2\Phi}|J_2+J_3|^2\,\dd x\dd t
	&\leq C_2e^{W_R}\int_{1/8}^{7/8}
	\int_{B_{2\varepsilon}\setminus B_\varepsilon}
	(|u|^2+|\nabla_Au|^2)\,\dd x\dd t\nonumber\\
	&\leq C_2e^{W_R}E_1^2.
\end{align}
Here $C_2$ depends on $d,\Lambda,\|\nabla G\|_\infty$ and
$\varepsilon$, but not on $R,\mu,\ell$.

For the outer errors, the uniform derivative bounds on $\kappa_R$ give
\begin{align}\label{new outer error}
	\int_0^1\int_{\mathbb{R}^{d}} e^{2\Phi}|J_4+J_5|^2\,\dd x\dd t
	&\leq C_3e^{2\mu+6\mu^{\beta_\ell}}
	\int_{1/8}^{7/8}\int_{B_R\setminus B_{R-1}}
	(|u|^2+|\nabla_Au|^2)\,\dd x\dd t\nonumber\\
	&=C_3e^{2\mu+6\mu^{\beta_\ell}}\delta(R).
\end{align}
The constant $C_3$ depends only on $d,\Lambda,\|\nabla G\|_\infty$
and the fixed cutoff profile. Finally,
\[
\int_0^1\int_{\mathbb{R}^{d}} e^{2\Phi}|Vh|^2\,\dd x\dd t
\leq\|V\|_{L^\infty}^2N.
\]

Using $|z_1+z_2+z_3+z_4|^2\leq4\sum_{j=1}^4|z_j|^2$ with the four
groups $Vh,J_1,J_2+J_3,J_4+J_5$, we conclude that
\begin{align}\label{control0}
	\frac{\mathfrak c_2}{2}K_RN
	+\frac{\mathfrak c_2}{2}K_Re^{W_R}E_2^2
	\leq{}&4\|V\|_{L^\infty}^2N+4C_1e^{2\mu}E_1^2
	+4C_2e^{W_R}E_1^2\nonumber\\
	&+4C_3e^{2\mu+6\mu^{\beta_\ell}}\delta(R).
\end{align}
Since $\mu\geq1$ and $\beta_\ell>1$, we have
$2\mu\leq6\mu^{\beta_\ell}\leq W_R$. Hence it is enough to impose
\begin{align}\label{control1}
	K_R&\geq\frac{16}{\mathfrak c_2}\|V\|_{L^\infty}^2,\\
	K_RE_2^2&\geq
	\frac{16}{\mathfrak c_2}(C_1+C_2)E_1^2.
	\label{control2}
\end{align}
Under these conditions, \eqref{control0} becomes
\begin{equation}\label{final mu R}
	\frac{\mathfrak c_2}{4}K_RN
	+\frac{\mathfrak c_2}{4}K_Re^{W_R}E_2^2
	\leq4C_3e^{2\mu+6\mu^{\beta_\ell}}\delta(R).
\end{equation}
It remains to choose $\mu,R$ so that all the parameter conditions hold
simultaneously. Set
\begin{equation}\label{selection of mu and R}
	H_0=\frac{81\|\varphi''\|_{L^\infty}}{32\lambda^2},\qquad
	\mu_* =\max\{1,H_0^{3/5}\},\qquad
	\mu=\mu_*R^{p_\ell}.
\end{equation}
Since $5/3\leq3-\beta_\ell<2$, this choice gives
$\mu_*^{3-\beta_\ell}\geq H_0$, and therefore
$\mu\geq\mathfrak A_3(R)$ for all $R\geq1$.
Moreover,
\[
\frac{\mu}{R^{4/3}}=\mu_*R^{p_\ell-4/3}\longrightarrow\infty,
\qquad
K_R=\mu_*^3R^{3p_\ell-4}\longrightarrow\infty.
\]
Thus the conditions involving $\mathfrak A_1,\mathfrak A_2$,
\eqref{control1}, and \eqref{control2} all hold for sufficiently large
$R$. More explicitly, evaluate $C_Q,C_B$ at $r_0=\varepsilon/2$ and put
\begin{align*}
	D_* =\max\bigg\{&1,
	\left(\frac{81C_Q}{32\lambda^2}\right)^{1/2},
	\left(\frac{81C_B}{32\lambda^2}\right)^{1/2},
	\left(\frac{16\|V\|_{L^\infty}^2}{\mathfrak c_2}\right)^{1/3},\\
	&\left(\frac{16(C_1+C_2)E_1^2}
	{\mathfrak c_2E_2^2}\right)^{1/3}\bigg\}.
\end{align*}
One may take
\[
R_1=2+\max\left\{R_0,2\varepsilon,
\left(\frac{D_*}{\mu_*}\right)^{1/(p_\ell-4/3)}\right\}.
\]
Then all the above conditions hold for $R\geq R_1$, including $K_R\geq1$.

Finally, cancelling the common time weight in \eqref{final mu R} yields
\begin{align*}
	\delta(R)&\geq\frac{\mathfrak c_2}{16C_3}K_RE_2^2
	\exp\left(-2\mu
	+2\mu\left(\frac{2\varepsilon}{R}\right)^{4/3}\right)\\
	&\geq\frac{\mathfrak c_2E_2^2}{16C_3}e^{-2\mu}
	=\frac{\mathfrak c_2E_2^2}{16C_3}
	e^{-2\mu_*R^{6/(3-2^{-\ell})}}.
\end{align*}
This proves \eqref{goal estimate} with
$C=\mathfrak c_2E_2^2/(16C_3)$ and $C_0=2\mu_*$.
Since the time cutoff $\varphi$ was fixed,
$C_0$ depends only on $\lambda$, as asserted.

\end{proof}

\begin{proof}[Proof of Theorem \ref{thm1}]
We argue by contradiction. Assume that there exist $R_{0}>0$ sufficiently large and $\varepsilon>0$ sufficiently small such that
\begin{equation}
\int_{1/4}^{3/4}\int_{B_{R_{0}}\backslash B_{2\varepsilon}}|u|^{2}\,\dd x\dd t\in(0,\infty).
\end{equation}
Moreover, by logarithmic convexity, we know that
\begin{equation}
\int_{\frac{1}{8}}^{\frac{7}{8}}\int_{\mathbb R^d}(|u|^{2}+|\nabla_{A}u|^{2})\,\dd x\dd t<\infty.
\end{equation}
Applying Theorem \ref{mass lower}, we obtain the lower bound
\begin{equation}\label{contra part}
\int_{\frac{1}{8}}^{\frac{7}{8}}\int_{B_{R}\backslash B_{R-1}}\Big(|u|^{2}+|\nabla_{A}u|^{2}\Big)\,\dd x\dd t\geq C_{1}e^{-C_{0}R^{\frac{6}{3-2^{-\ell}}}}.
\end{equation}
On the other hand, Corollary \ref{cor-log} implies that
\begin{equation}
\int_{0}^{1}\int_{\mathbb R^d}e^{\sigma|x|^{2+\vartheta}}\Big(|u|^{2}+t(1-t)|\nabla_{A}u|^{2}\Big)\,\dd x\dd t<\infty,
\end{equation}
hence
\begin{equation}\label{reduction last}
\lim_{R\to\infty}e^{\sigma R^{2+\vartheta}}\int_{\frac{1}{8}}^{\frac{7}{8}}\int_{B_{R}\backslash B_{R-1}}\big(|u|^{2}+|\nabla_{A}u|^{2}\big)\,\dd x\dd t=0.
\end{equation}
Since
\[
\frac{6}{3-2^{-\ell}}<2+\vartheta
\]
for $\ell$ sufficiently large, the lower bound in \eqref{contra part} is incompatible with \eqref{reduction last} for large $R$. This contradiction proves that $u\equiv0$.
\end{proof}
\begin{remark}
	The separate temporal cutoff makes $h$ vanish near both time endpoints
	for every $x$. The parameters $\beta_\ell$, $p_\ell$, and the lower
	bound on $\mu$ are specified in Proposition \ref{Carleman estimate}.
	The argument uses a finite $\ell$ chosen after $p$; no assertion that
	the cutoff fails in the limit $\ell\to\infty$ is needed.
\end{remark}

\section{Proof of Theorem \ref{schrodinger}}\label{sec:proof-schrodinger}

In this section, we prove the Hardy-type uniqueness result under the additional structural assumption on the metric $G$. More precisely, under Hypothesis \ref{assm2}, we establish the uniqueness statement in the spirit of \cite{BFGRV-JFA}.

We begin with the corresponding Carleman estimate.

\begin{proposition}[Carleman estimate]
Assume that $d\geq2$, the metric $G$ satisfies Hypothesis \ref{assm2}, the magnetic field $B$ satisfies
 \begin{equation}
     \|(G\textbf{e}_1)^{\top}B\|_{L^{\infty}_{t,x}}\leq\varepsilon_1,\, (Gx)^{\top}B\in L_{x}^{\infty}.
 \end{equation}


Then, for any $\varepsilon>0$, any $\mu>0$, any
\[
h=h(x,t)\in C_{0}^{\infty}(\mathbb{R}_{x}^{d}\times(0,1)_{t}),
\]
and any $R>\mathcal{C}_{3}\mu\varepsilon^{-\frac{1}{2}}$, the following inequality holds:
\begin{align}
    &\frac{R}{4}\sqrt{\frac{\varepsilon}{\mu}}\Big\|e^{\mu|x+Rt(1-t)\textbf{e}_{1}|^{2}-\frac{(1+\varepsilon+N\mu^{2}(\varepsilon_{0}+\varepsilon_{1}))R^{2}t(1-t)}{16\mu}}h(x,t)\Big\|_{L^{2}(\mathbb{R}^{d+1})}\nonumber\\
    \leq&
    \Big\|e^{\mu|x+Rt(1-t)\textbf{e}_{1}|^{2}-\frac{(1+\varepsilon+N\mu^{2}(\varepsilon_{0}+\varepsilon_{1}))R^{2}t(1-t)}{16\mu}}(\partial_{t}-i\operatorname{div}_{A}(G\nabla_{A}\cdot))h\Big\|_{L^{2}(\mathbb{R}^{d+1})}. \label{carleman estimate for special designed test function}
\end{align}
\end{proposition}

\begin{proof}
Let
\begin{equation}
f(x,t)=e^{\mu|x+Rt(1-t)\textbf{e}_{1}|^{2}-\frac{(1+\varepsilon+N(\varepsilon_{0}+\varepsilon_{1})\mu^{2})R^{2}t(1-t)}{16\mu}}h(x,t).
\end{equation}
Then the Carleman estimate can be rewritten as
\begin{align*}
    &e^{\mu|x+Rt(1-t)\textbf{e}_{1}|^{2}-\frac{(1+\varepsilon+N(\varepsilon_{0}+\varepsilon_{1})\mu^{2})R^{2}t(1-t)}{16\mu}}(\partial_{t}-i\operatorname{div}_{A}(G\nabla_{A}\cdot))h\\
=& e^{\mu|x+Rt(1-t)\textbf{e}_{1}|^{2}-\frac{(1+\varepsilon+N(\varepsilon_{0}+\varepsilon_{1})\mu^{2})R^{2}t(1-t)}{16\mu}}(\partial_{t}-i\operatorname{div}_{A}(G\nabla_{A}\cdot))\big(e^{-\mu|x+Rt(1-t)\textbf{e}_{1}|^{2}+\frac{(1+\varepsilon+N(\varepsilon_{0}+\varepsilon_{1})\mu^{2})R^{2}t(1-t)}{16\mu}}f\big)\\
=&(\partial_{t}-\mathcal{S}-\mathcal{A})f.
\end{align*}
Note that $\mathcal{S}$ and $\mathcal{A}$ are given by \eqref{symmetry part} and \eqref{antisymmetry part}, respectively, with $a=0$ and $b=1$.

Following the standard argument for Carleman estimates, we obtain
\begin{align}
    \big\|(\partial_{t}-\mathcal{S}-\mathcal{A})f\big\|^{2}_{L^{2}(\mathbb{R}^{d+1})}
    =&\|(\partial_{t}-\mathcal{A})f\|^{2}_{L^{2}(\mathbb{R}^{d+1})}+\|\mathcal{S}f\|^{2}_{L^{2}(\mathbb{R}^{d+1})}\nonumber\\
    &-2\operatorname{Re}\int_{\mathbb{R}^{d+1}}\mathcal{S}f\overline{(\partial_{t}-\mathcal{A})f}\,\dd x\dd t\\
    \geq&\int_{\mathbb{R}^{d+1}}\big(\partial_{t}\mathcal{S}+[\mathcal{S},\mathcal{A}]\big)f\bar{f}\,\dd x\dd t.
\end{align}
Using \eqref{S time derivative} and \eqref{formula of commutator} with $a=0$ and $b=1$, and taking
\begin{align}
\varphi(x,t)=\mu|x+Rt(1-t)\textbf{e}_{1}|^{2}-\frac{(1+\varepsilon+N(\varepsilon_{0}+\varepsilon_{1})\mu^{2})R^{2}t(1-t)}{16\mu},
\end{align}
we compute
\begin{align*}
    &\int_{\mathbb{R}^{d+1}}(\partial_{t}\mathcal{S}+[\mathcal{S},\mathcal{A}])f\bar{f}\,\dd x\dd t\\
    =&{8\mu R\int_{\mathbb{R}^{d+1}}\operatorname{Im}G(1-2t)\textbf{e}_{1}\cdot\nabla_{A}f\bar{f}}\,\dd x\dd t+\int_{\mathbb{R}^{d+1}}2G\nabla\varphi\cdot\nabla(G\nabla\varphi\cdot\nabla\varphi)|f|^{2}\,\dd x\dd t\\
    &-\int_{\mathbb{R}^{d+1}}\mathcal{Q}^{2}\varphi|f|^{2}\,\dd x\dd t\\
    &-8\mu\int_{\mathbb{R}^{d+1}}g_{jk}g_{lm}(x_{l}+Rt(1-t)\textbf{e}_{1}\delta_{1l})B_{mk}f\overline{D_{j}f}\,\dd x\dd t+8\mu\int_{\mathbb{R}^{d+1}}|G\nabla_{A}f|^{2}\,\dd x\dd t\\
    &+2\int_{\mathbb{R}^{d+1}} g_{jk}\partial_{k}g_{lm}\partial_{l}\varphi D_{m}f\overline{D_{j}f}\,\dd x\dd t\\
    &-2\int_{\mathbb{R}^{d+1}}g_{lm}\partial_{l}\varphi\partial_{m}g_{jk}D_{k}f\overline{D_{j}f}\,\dd x\dd t
    +\int_{\mathbb{R}^{d+1}}2\partial_{j}g_{lm}\partial_{l}\varphi g_{jk}D_{k}f\overline{D_{m}f}\,\dd x\dd t\\
    &+\int_{\mathbb{R}^{d+1}}\Bigg(\frac{(1+\varepsilon+N\varepsilon_{0}\mu^{2})R^{2}}{8\mu}+2\mu R^{2}|\textbf{e}_{1}|^{2}(1-2t)^{2}-4\mu R\textbf{e}_{1}\cdot(x+Rt(1-t)\textbf{e}_{1})\Bigg)|f|^{2}\,\dd x\dd t\\
    =:&\sum_{\ell=1}^{9}Z_{\ell}.
\end{align*}

We now estimate these terms.

\textbf{Estimate of $Z_1$, $Z_{5}$, $Z_{9}^{2}$.}
\begin{align}
Z_{1}=8\mu\int_{\mathbb{R}^{d+1}}\Big|G\nabla_{A}f+i\frac{\textbf{e}_{1}R(1-2t)}{2}f\Big|^{2}\,\dd x\dd t-Z_{5}-Z_{9}^{2},
\end{align}
where
\[
Z_{9}^{2}= \int_{\mathbb{R}^{d+1}}2\mu R^{2}|\textbf{e}_{1}|^{2}(1-2t)^{2}|f|^{2}\,\dd x\dd t.
\]

\textbf{Estimate of $Z_2$, $Z_{9}^{\textit{part}}$.}

\ 
Since
\[
G=I+\tilde{G},
\qquad
\tilde G=(\tilde g_{jk})_{j,k=1}^d,
\qquad
\tilde{g}_{1k}=0,
\]
a direct calculation yields
\begin{align}
    G\nabla\varphi\cdot\nabla(G\nabla\varphi\cdot\nabla\varphi)
    =&\nabla\varphi\cdot\nabla(|\nabla\varphi|^{2})+\nabla\varphi\cdot\nabla(\tilde{G}\nabla\varphi\cdot\nabla\varphi)\nonumber\\
    &+\tilde{G}\nabla\varphi\cdot\nabla(\tilde{G}\nabla\varphi\cdot\nabla\varphi)+\tilde{G}\nabla\varphi\cdot\nabla(|\nabla\varphi|^{2}).
\end{align}
Accordingly,
\begin{align}
    Z_{2}
    =&2\int_{\mathbb{R}^{d+1}}\Big[\nabla\varphi\cdot\nabla(|\nabla\varphi|^{2})|f|^{2}+\nabla\varphi\cdot\nabla(\tilde{G}\nabla\varphi\cdot\nabla\varphi)|f|^{2}\nonumber\\
    &\hspace{8ex}
    +\tilde{G}\nabla\varphi\cdot\nabla(\tilde{G}\nabla\varphi\cdot\nabla\varphi)|f|^{2}
    +\tilde{G}\nabla\varphi\cdot\nabla(|\nabla\varphi|^{2})|f|^{2}\Big]\,\dd x\dd t\nonumber\\
    :=&Z_{2}^{1}+Z_{2}^{2}+Z_{2}^{3}+Z_{2}^{4}. \label{decomposition of G}
\end{align}

By the definition of $\varphi$, we derive
\begin{equation}
Z_{2}^{1}=32\mu^{3}\int_{\mathbb{R}^{d+1}}|x+Rt(1-t)\textbf{e}_{1}|^{2}|f|^{2}\,\dd x\dd t.\label{spe1}
\end{equation}
Let
\begin{equation}
Z_{9}^{\textit{part}}:=\int_{\mathbb{R}^{d+1}}\frac{R^{2}}{8\mu}|f|^{2}\,\dd x\dd t+\int_{\mathbb{R}^{d+1}}-4\mu R\textbf{e}_{1}\cdot(x+Rt(1-t)\textbf{e}_{1})|f|^{2}\,\dd x\dd t.
\end{equation}

Completing the square, we obtain the nonnegative expression
\begin{align}
Z_{2}^{1}+Z_{9}^{\textit{part}}
=32\mu^{3}\int_{\mathbb{R}^{d+1}}|f|^2\Big|x+Rt(1-t)\textbf{e}_{1}-\frac{R}{16\mu^{2}}\textbf{e}_{1}\Big|^{2}\,\dd x\dd t\geq0.
\end{align}

We estimate $Z_2^2$, $Z_2^3$, and $Z_2^4$ together.
Here, $\tilde G(x')$ denotes the transverse
$(d-1)\times(d-1)$ matrix in Hypothesis \ref{assm2}.
Since
\begin{align}
\nabla_{x^\prime}\varphi=2\mu x^\prime,
\qquad
\nabla^2\varphi=2\mu I_d,
\end{align}
a direct calculation gives
\begin{align}
Z_2^2+Z_2^3+Z_2^4
={}&
32\mu^3\int_0^1\int_{\mathbb R^d}
\Big(
2x'^{\top}\tilde G(x')x'
+|\tilde G(x')x'|^2
\Big)|f|^2\,\dd x\,\dd t
\nonumber\\
&+
16\mu^3\int_0^1\int_{\mathbb R^d}
x'^{\top}
\Big(
\big((I_{d-1}+\tilde G(x'))x'\big)
\cdot\nabla_{x'}\tilde G(x')
\Big)
x'\,|f|^2\,\dd x\,\dd t.
\label{variable1}
\end{align}
Since $\tilde G(x')$ is nonnegative,
\[
|(I_{d-1}+\tilde G(x^\prime))x^\prime|\ge |x^\prime|.
\]
Moreover, \eqref{special metric101} implies
\begin{align*}
\left|
x'^{\top}
\Big(
\big((I_{d-1}+\tilde G(x'))x'\big)
\cdot\nabla_{x'}\tilde G(x')
\Big)x'
\right|
\leq&
c_d|\nabla_{x'}\tilde G(x')|\,
|(I_{d-1}+\tilde G(x'))x'|\,|x'|^2
\\ \leq
&
c_d\varepsilon_0
|(I_{d-1}+\tilde G(x'))x'|\,|x'|
\\\leq
&
c_d\varepsilon_0
|(I_{d-1}+\tilde G(x'))x'|^2.
\end{align*}
Combining \eqref{variable1} with the completed-square
identity above, and using
\[
|x'|^2+2x'^{\top}\tilde G(x')x'
+|\tilde G(x')x'|^2
=
|(I_{d-1}+\tilde G(x'))x'|^2,
\]
we obtain
\begin{align}
Z_2+Z_9^{\mathrm{part}}
\ge{}&
32\mu^3\int_0^1\int_{\mathbb R^d}
\left(
x_1+Rt(1-t)-\frac{R}{16\mu^2}
\right)^2
|f|^2\,\dd x\,\dd t
\nonumber\\
&+
(32-16c_d\varepsilon_0)\mu^3
\int_0^1\int_{\mathbb R^d}
|(I_{d-1}+\tilde G(x'))x'|^2
|f|^2\,\dd x\,\dd t
\ge0,
\label{variable2}
\end{align}
provided that $c_d\varepsilon_0\le2$.
Thus, these terms do not produce any loss in the estimate.

\textbf{Estimate of $Z_3$.}
\begin{align}
Z_{3}
&= - \int_{\mathbb{R}^{d+1}} \mathcal{Q}^{2}\varphi\, |f|^{2}\,\dd x\,\dd t
\nonumber\\
&= -2\mu \int_{\mathbb{R}^{d+1}}
\Bigl(
\partial_{l} g_{lm}\,\partial_{m}\partial_{k} g_{kj}\,
(x_{j}+Rt(1-t)\mathbf e_{1}\delta_{1j})
\nonumber\\
&\qquad
+ g_{lm}\,\partial_{l}\partial_{m}\partial_{k} g_{kj}\,
(x_{j}+Rt(1-t)\mathbf e_{1}\delta_{1j})
\nonumber\\
&\qquad
+ \partial_{l} g_{lm}\,\partial_{k} g_{kj}\,\delta_{jm}
+ g_{lm}\,\partial_{m}\partial_{k} g_{kj}\,\delta_{jl}
\nonumber\\
&\qquad
+ g_{lm}\,\partial_{l}\partial_{k} g_{kj}\,\delta_{jm}
+ \partial_{l} g_{lm}\,\partial_{m} g_{kk}
+ g_{lm}\,\partial_{l}\partial_{m} g_{kk}
\Bigr)
|f|^{2}\,\dd x\,\dd t
\nonumber\\
&\ge
-2 c_{d}\mu
\Bigl(
\sup_{x'\in\mathbb{R}^{d-1}}
\bigl(|x'|\,|\nabla G(x')|\bigr)\,
\|\nabla^{2} G\|_{L^{\infty}}
\nonumber\\
&\qquad
+ \sup_{x'\in\mathbb{R}^{d-1}}
|x'|\,|\nabla^{3} G(x')|
+ 2\|\nabla G\|^{2}_{L^{\infty}}
+ 2\Lambda \|\nabla^{2} G\|_{L^{\infty}}
\Bigr)
\int_{\mathbb{R}^{d+1}} |f|^{2}\,\dd x\,\dd t .
\label{bilaplace estimate}
\end{align}

\vspace{-6pt}
\textbf{Estimate of $Z_4$ (magnetic term).}
Using
$\|(G\mathbf e_{1})^{\top} B\|_{L^\infty_{t,x}} \le \varepsilon_{1},
$
\vspace{-6pt}
we obtain
\vspace{-6pt}
\begin{align}
Z_{4}
&= -8\mu \int_{\mathbb{R}^{d+1}}
g_{jk} g_{lm} (x_{l}+Rt(1-t)\mathbf e_{1}\delta_{1l}) B_{mk}\,
f\,\overline{D_{j} f}
\,\dd x\,\dd t
\nonumber\\
&= -8\mu \int_{\mathbb{R}^{d+1}}
(Gx)^{\top}BG f\, \overline{\nabla_{A} f}
\,\dd x\,\dd t
- 8\mu \int_{\mathbb{R}^{d+1}}
Rt(1-t)(G\mathbf e_{1})^{\top}BGf\,
\overline{\nabla_{A}f}\,\dd x\,\dd t
\nonumber\\
&= -8\mu \int_{\mathbb{R}^{d+1}}
(Gx)^{\top}B f\,
\overline{\Bigl(G\nabla_{A} f
+ i \tfrac{R(1-2t)}{2} G\mathbf e_{1} f\Bigr)}
\,\dd x\,\dd t
+ 8i\mu \int_{\mathbb{R}^{d+1}}
(Gx)^{\top}B \,\tfrac{R(1-2t)}{2}\, G\mathbf e_{1}\, |f|^{2}
\,\dd x\,\dd t
\nonumber\\
&\qquad
- 8\mu \int_{\mathbb{R}^{d+1}}
Rt(1-t)(G\mathbf e_{1})^{\top}BGf\,
\overline{\nabla_{A}f}\,\dd x\,\dd t
\nonumber\\
&\ge
-16\mu \|(Gx)^{\top}B\|_{L^{\infty}}^{2}
\int_{\mathbb{R}^{d+1}} |f|^{2}\,\dd x\,\dd t
- \mu \int_{\mathbb{R}^{d+1}}
\Bigl|G\nabla_{A} f
+ i \tfrac{R(1-2t)}{2} \mathbf e_{1} f\Bigr|^{2}
\,\dd x\,\dd t
\nonumber\\
&\qquad
-\Bigl(
64\mu R^2\|(G\mathbf e_{1})^{\top}B\|_{L^{\infty}_{t,x}}
+4\mu R\|(Gx)^{\top}B\|_{L^{\infty}_{t,x}}
\Bigr)
\int_{\mathbb{R}^{d+1}}|f|^{2}\,\dd x\,\dd t
-\mu\int_{\mathbb{R}^{d+1}}|G\nabla_{A}f|^{2}\,\dd x\,\dd t
\nonumber\\
&\ge
-16\mu \|(Gx)^{\top}B\|_{L^{\infty}}^{2}
\int_{\mathbb{R}^{d+1}} |f|^{2}\,\dd x\,\dd t
- 2\mu \int_{\mathbb{R}^{d+1}}
\Bigl|G\nabla_{A} f
+ i \tfrac{R(1-2t)}{2} \mathbf e_{1} f\Bigr|^{2}
\,\dd x\,\dd t
\nonumber\\
&\qquad
-\Bigl(
68\mu R^2\|(G\mathbf e_{1})^{\top}B\|_{L^{\infty}_{t,x}}
+4\mu R\|(Gx)^{\top}B\|_{L^\infty_{t,x}}
\Bigr)
\int_{\mathbb{R}^{d+1}}|f|^{2}\,\dd x\,\dd t .
\label{magnetic estimate}
\end{align}

\textbf{Estimate of $Z_{6}$, $Z_{7}$, and $Z_{8}$.}
Hypothesis \ref{assm2} implies that only the indices $j,k,l,m\ge 2$ contribute. Hence
\begin{align}
&Z_{6}+Z_{7}+Z_{8}
\\
=&
2\int_{\mathbb{R}^{d+1}}
g_{jk}\,\partial_{k} g_{lm}\,\partial_{l}\varphi\,
D_{m}f\,\overline{D_{j}f}\,\dd x\dd t
-2\int_{\mathbb{R}^{d+1}}
g_{lm}\,\partial_{l}\varphi\,\partial_{m} g_{jk}\,
D_{k}f\,\overline{D_{j}f}\,\dd x\dd t
\nonumber\\
&\quad
+2\int_{\mathbb{R}^{d+1}}
\partial_{j} g_{lm}\,\partial_{l}\varphi\, g_{jk}\,
D_{k}f\,\overline{D_{m}f}\,\dd x\dd t
\nonumber\\
=&
\sum_{j,k,l,m\ge 2}
\Bigg(
2\int_{\mathbb{R}^{d+1}}
g_{jk}\,\partial_{k} g_{lm}\,\partial_{l}\varphi\,
D_{m}f\,\overline{D_{j}f}\,\dd x\dd t
\nonumber\\
&\qquad
-2\int_{\mathbb{R}^{d+1}}
g_{lm}\,\partial_{l}\varphi\,\partial_{m} g_{jk}\,
D_{k}f\,\overline{D_{j}f}\,\dd x\dd t
+2\int_{\mathbb{R}^{d+1}}
\partial_{j} g_{lm}\,\partial_{l}\varphi\, g_{jk}\,
D_{k}f\,\overline{D_{m}f}\,\dd x\dd t
\Bigg)
\end{align}
\begin{align}
\ge&
-6\mu c_{d}
\sup_{x'\in\mathbb{R}^{d-1}}
\bigl(|x'|\,|\nabla \tilde{G}(x')|\bigr)
\int_{\mathbb{R}^{d+1}}
|\nabla_{A}f|^{2}\,\dd x\dd t
\\
\ge&
-12\mu c_{d}
\sup_{x'\in\mathbb{R}^{d-1}}
\bigl(|x'|\,|\nabla \tilde{G}(x')|\bigr)
\int_{\mathbb{R}^{d+1}}
\Bigl(
\bigl|G\nabla_{A}f
+ i \tfrac{R(1-2t)}{2} \textbf{e}_{1} f\bigr|^{2}
+ \bigl|\tfrac{R(1-2t)}{2} \textbf{e}_{1} f\bigr|^{2}
\Bigr)
\,\dd x\dd t
\nonumber\\
\ge&
-12\mu c_{d}
\sup_{x'\in\mathbb{R}^{d-1}}
\bigl(|x'|\,|\nabla \tilde{G}(x')|\bigr)
\int_{\mathbb{R}^{d+1}}
\bigl|G\nabla_{A}f
+ i \tfrac{R(1-2t)}{2} \textbf{e}_{1} f\bigr|^{2}
\,\dd x\dd t
\nonumber\\
&\quad
-48\mu c_{d}
\sup_{x'\in\mathbb{R}^{d-1}}
\bigl(|x'|\,|\nabla \tilde{G}(x')|\bigr)
R^{2}
\int_{\mathbb{R}^{d+1}}
|f|^{2}\,\dd x\,\dd t .
\label{strange term}
\end{align}

Collecting all the bounds for the terms $Z_{\ell}$, we arrive at
\begin{align}
    &\int_{\mathbb{R}^{d+1}}(\partial_{t}\mathcal{S}+[\mathcal{S},\mathcal{A}])f\bar{f}\,\dd x\dd t\\
    \geq&
    \Bigg(\frac{(\varepsilon+N(\varepsilon_{0}+\varepsilon_{1})\mu^{2})R^{2}}{8\mu}
    -48R^{2}\mu c_{d}\sup_{x'\in\mathbb{R}^{d-1}}|x'||\nabla\tilde{G}(x')|
    -68\mu R^{2}\|(G\textbf{e}_{1})^{\top}B\|_{L^{\infty}_{t,x}}\nonumber\\
    &
    -4\mu R\|(Gx)^{\top}B\|_{L^\infty}-16\mu\|(Gx)^{\top}B\|^{2}_{L^{\infty}}\\&
    -2c_{d}\mu\Big(\sup_{x'\in\mathbb{R}^{d-1}}(|x'||\nabla G(x')|)\|\nabla^{2}G\|_{L^{\infty}}+\sup_{x'\in\mathbb{R}^{d-1}}|x'||\nabla^{3}G(x')|+2\|\nabla G\|^{2}_{L^{\infty}}\nonumber\\
    &\hspace{48ex}+2\Lambda\|\nabla^{2}G\|_{L^{\infty}}\Big)\Bigg)\int_{\mathbb{R}^{d+1}}|f|^{2}\,\dd x\dd t\nonumber\\
    &+(6\mu-12\mu c_{d}\sup_{x'\in\mathbb{R}^{d-1}}|x^{\prime}||\nabla\tilde{G}(x^{\prime})|)\int_{\mathbb{R}^{d+1}}\Big|G\nabla_{A}f+i\frac{\textbf{e}_{1}(1-2t)R}{2}f\Big|^{2}\,\dd x\dd t.
\end{align}

Now observe that
\begin{align}
   & \frac{N(\varepsilon_{0}+\varepsilon_{1})\mu^{2}R^{2}}{8\mu}-68\mu R^2\|(G\textbf{e}_{1})^{\top}B\|_{L^{\infty}_{t,x}}-48c_{d}\mu R^{2}\sup_{x'\in\mathbb{R}^{d-1}}|x'||\nabla\tilde{G}(x')|\\
\geq&
\Big(\frac{N(\varepsilon_{0}+\varepsilon_{1})}{8}-48c_{d}\varepsilon_{0}-68\varepsilon_{1}\Big)\mu R^{2}\geq0,
\end{align}
provided
\[
N=\max(384 c_{d},544).
\]
Thus, in order to preserve positivity, it remains to require
\begin{align}
    &\frac{\varepsilon R^{2}}{8\mu}-4\mu R\|(Gx)^{\top}B\|_{L^{\infty}_{t,x}}-16\mu\|(Gx)^{\top}B\|^{2}_{L^{\infty}}\\
    &\hspace{8ex}
    -2c_{d}\mu\Big(\sup_{x'\in\mathbb{R}^{d-1}}(|x'||\nabla G(x')|)\|\nabla^{2}G\|_{L^{\infty}}+\sup_{x'\in\mathbb{R}^{d-1}}|x'||\nabla^{3}G(x')|+2\|\nabla G\|^{2}_{L^{\infty}}+2\Lambda\|\nabla^{2}G\|_{L^{\infty}}\Big)\nonumber\\
    \geq&\frac{\varepsilon R^{2}}{16\mu},
\end{align}
which is equivalent to
\begin{align}
    &\frac{\varepsilon R^{2}}{16\mu}\geq16\mu\|(Gx)^{\top}B\|^{2}_{L^{\infty}}+4\mu R\|(Gx)^{\top}B\|_{L^\infty}\\
    &\hspace{8ex}
    +2c_{d}\mu\Big(\sup_{x'\in\mathbb{R}^{d-1}}(|x'||\nabla G(x')|)\|\nabla^{2}G\|_{L^{\infty}}+\sup_{x'\in\mathbb{R}^{d-1}}|x'||\nabla^{3}G(x')|+2\|\nabla G\|^{2}_{L^{\infty}}+2\Lambda\|\nabla^{2}G\|_{L^{\infty}}\Big).
\end{align}
Rearranging terms gives a quadratic condition in $R$:
\begin{align}
&R^{2}-\Big(\frac{64\mu^{2}}{\varepsilon}\|(Gx)^{\top}B\|_{L^{\infty}_{t,x}}\Big)R\\
\geq&
\frac{256\mu^{2}}{\varepsilon}\|(Gx)^{\top}B\|_{L^{\infty}}^{2}
+\frac{32c_{d}\mu^{2}}{\varepsilon}\Big(\sup_{x'\in\mathbb{R}^{d-1}}(|x'||\nabla G(x')|)\|\nabla^{2}G\|_{L^{\infty}}+\sup_{x'\in\mathbb{R}^{d-1}}|x'||\nabla^{3}G(x')|\nonumber\\
&\hspace{20ex}+2\|\nabla G\|^{2}_{L^{\infty}}+2\Lambda\|\nabla^{2}G\|_{L^{\infty}}\Big)
:=\mathcal{C}_{3}\frac{\mu^{2}}{\varepsilon}.
\end{align}
Hence, for
\[
R>\max\Big(\frac{128\mu^{2}}{\varepsilon}\|(Gx)^{\top}B\|_{L^\infty},(2\mathcal{C}_{3}\frac{\mu^{2}}{\varepsilon})^{2}\Big),
\]
we obtain the desired estimate,
\begin{align}
    &\sqrt{\frac{\varepsilon R^{2}}{16\mu}}\Big\|e^{\mu|x+Rt(1-t)\textbf{e}_{1}|^{2}-\frac{(1+\varepsilon+N(\varepsilon_{0}+\varepsilon_{1})\mu^{2})R^{2}t(1-t)}{16\mu}}h\Big\|_{L^{2}([0,1]_{t}\times\mathbb{R}^{d})}\\
    \leq&
    \Big\|e^{\mu|x+Rt(1-t)\textbf{e}_{1}|^{2}-\frac{(1+\varepsilon+N(\varepsilon_{0}+\varepsilon_{1})\mu^{2})R^{2}t(1-t)}{16\mu}}(\partial_{t}-i\operatorname{div}_{A}(G\nabla_{A}))h\Big\|_{L^{2}([0,1]_{t}\times\mathbb{R}^{d})},
\end{align}
where
\begin{align}
\mathcal{C}_{3}
=&256\|(Gx)^{\top}B\|_{L^{\infty}}^{2}
+32c_{d}\Big(\sup_{x'\in\mathbb{R}^{d-1}}(|x'||\nabla G(x')|)\|\nabla^{2}G\|_{L^{\infty}}+\sup_{x'\in\mathbb{R}^{d-1}}|x'||\nabla^{3}G(x')|\nonumber\\
&\hspace{12ex}+2\|\nabla G\|^{2}_{L^{\infty}}+2\Lambda\|\nabla^{2}G\|_{L^{\infty}}\Big).
\end{align}
This completes the proof.
\end{proof}

\begin{proof}[Proof of Theorem \ref{schrodinger}]
First, by logarithmic convexity (Lemma \ref{Log-convex-special}), we have
\begin{align}
N_{\gamma}
:=
\sup_{0\le t\le 1}
\|e^{\gamma |x|^{2}}u(t)\|_{L^{2}(\mathbb{R}^{d})}
+
\|\sqrt{t(1-t)}\,e^{\gamma |x|^{2}}\nabla_{A}u\|_{L^{2}([0,1]\times\mathbb{R}^{d})}
< \infty .
\label{N gamma}
\end{align}

We choose
\[
h(x,t)=\theta_M(x)\,\zeta_R(t)\,u(x,t),
\]
where $\theta_M \in C_0^\infty(\mathbb{R}^d)$ satisfies
\[
\theta_M(x)=1 \quad \text{for } |x|\le M,\qquad \theta_M(x)=0 \quad \text{for } |x|\ge 2M,
\]
and
\[
|\partial^k \theta_M(x)| \le \frac{C_k}{M^k},
\]
with $M>0$ to be chosen later. Moreover, let $\zeta_R \in C_0^\infty(0,1)$ satisfy
\[
0\le \zeta_R \le 1,
\qquad
\zeta_R(t)=1
\quad \text{for } t\in \Big[\tfrac{1}{R},\,1-\tfrac{1}{R}\Big],
\]
and
\[
\zeta_R(t)=0
\quad \text{for } t\in
\Big[0,\,\tfrac{1}{2R}\Big]\cup
\Big[1-\tfrac{1}{2R},\,1\Big].
\]
In addition,
\[
|\partial^k \zeta_R(t)| \le C_k R^k .
\]

A direct computation yields
\begin{align}
&\bigl(\partial_{t}
- i\,\operatorname{div}_{A}(G\nabla_{A}\cdot)\bigr) h
\\
={}&
\theta_M(x)\,(\partial_t \zeta_R)\,u
+ \theta_M \zeta_R\,\partial_t u
- i \zeta_R\,\partial_j g_{jk}\,\partial_k \theta_M\,u
- i \zeta_R\, g_{jk}\,\partial_j \partial_k \theta_M\,u
\nonumber\\
&- i \zeta_R\, g_{jk}\,\partial_k \theta_M\, D_j u
- i \zeta_R\, g_{jk}\,\partial_j \theta_M\, D_k u
- i \theta_M \zeta_R\, D_j (g_{jk} D_k u)
\nonumber\\
={}&
i V h
+ \theta_M(x)\,(\partial_t \zeta_R)\,u
- i \zeta_R\,\partial_j g_{jk}\,\partial_k \theta_M\,u
- i \zeta_R\, g_{jk}\,\partial_j \partial_k \theta_M\,u
\nonumber\\
&- i \zeta_R\, g_{jk}\,\partial_k \theta_M\, D_j u
- i \zeta_R\, g_{jk}\,\partial_j \theta_M\, D_k u .
\label{ugly terms}
\end{align}

We now estimate the Carleman weight on the supports of the cutoff terms. Assume that
\[
\mu \le \frac{\gamma}{1+\varepsilon}
\]
for some $\varepsilon>0$. On the support of $\theta_M(x)\partial_t\zeta_R\cdot u$, we have
\[
t(1-t)\sim \frac{1}{R}\Bigl(1-\frac{1}{R}\Bigr).
\]
By Young's inequality with parameter $\varepsilon$,
\begin{align}
&\mu\bigl|x+Rt(1-t)\textbf{e}_{1}\bigr|^{2}
-\frac{R^{2}\bigl(1+\varepsilon+N(\varepsilon_{0}+\varepsilon_{1})\mu^{2}\bigr)t(1-t)}{16\mu}
\nonumber\\
\le{}&
\mu|x|^{2}
+2\mu R\, x\cdot \textbf{e}_{1}\, t(1-t)
+\mu R^{2} t^{2}(1-t)^{2}
\nonumber\\
\le{}&
\mu|x|^{2}
+\mu\varepsilon |x|^{2}
+\mu\frac{1}{\varepsilon} R^{2} t^{2}(1-t)^{2}
+\mu R^{2} t^{2}(1-t)^{2}
\nonumber\\
\le{}&
\mu(1+\varepsilon)|x|^{2}
+\mu\Bigl(\frac{1}{\varepsilon}+1\Bigr)
R^{2}\frac{1}{R^{2}}
\Bigl(1-\frac{1}{R}\Bigr)^{2}
\nonumber\\
\le{}&
\mu(1+\varepsilon)|x|^{2}
+\mu\Bigl(\frac{1}{\varepsilon}+1\Bigr)
\nonumber\\
\le{}&
\gamma |x|^{2}
+\frac{\gamma}{\varepsilon}.
\label{support1}
\end{align}

Now denote
\[
\begin{aligned}
\mathcal{I} := {} &
- i \zeta_R(t)\,\partial_j g_{jk}\,\partial_k \theta_M\,u
- i \zeta_R(t)\, g_{jk}\,\partial_j \partial_k \theta_M\,u \\
& - i \zeta_R(t)\, g_{jk}\,\partial_k \theta_M\, D_j u
- i \zeta_R(t)\, g_{jk}\,\partial_j \theta_M\, D_k u .
\end{aligned}
\]
Then
\[
\operatorname{supp} \mathcal{I}
\subset
\Big[\tfrac{1}{2R},\,1-\tfrac{1}{2R}\Big]_t
\times \bigl( B_{2M} \setminus B_M \bigr).
\]
On this support, we have
\begin{align}
    &\mu|x+Rt(1-t)\textbf{e}_{1}|^{2}-\frac{R^{2}(1+\varepsilon+N(\varepsilon_{0}+\varepsilon_{1})\mu^{2})t(1-t)}{16\mu}\nonumber\\
\leq&
\mu|x|^{2}+\mu\varepsilon|x|^{2}+\mu\frac{1}{\varepsilon}R^{2}t^{2}(1-t)^{2}+\mu R^{2}t^{2}(1-t)^{2}\nonumber\\
\leq&
\mu(1+\varepsilon)|x|^{2}+\mu\Big(1+\frac{1}{\varepsilon}\Big)R^{2}\nonumber\\
\leq&
\gamma|x|^{2}+\frac{\gamma}{\varepsilon}R^{2}. \label{support2}
\end{align}

We now apply the Carleman estimate \eqref{carleman estimate for special designed test function} to
\[
h(x,t)=\theta_M(x)\,\zeta_R(t)\,u(x,t).
\]
Using \eqref{ugly terms}, together with the bounds \eqref{support1} and \eqref{support2} on the support of the corresponding error terms, we obtain
\begin{align}
    &R\sqrt{\frac{\varepsilon}{\mu}}\Big\|e^{\mu|x+Rt(1-t)\textbf{e}_{1}|^{2}-\frac{(1+\varepsilon+N\mu^{2}(\varepsilon_{0}+\varepsilon_{1}))t(1-t)}{16\mu}}h(x,t)\Big\|_{L^{2}([0,1]\times\mathbb{R}^{d})}\nonumber\\
 \leq&
 \Big\|e^{\mu|x+Rt(1-t)\textbf{e}_{1}|^{2}-\frac{(1+\varepsilon+N\mu^{2}(\varepsilon_{0}+\varepsilon_{1}))t(1-t)}{16\mu}}(\partial_{t}-i\operatorname{div}_{A}(G\nabla_{A}\cdot))h\Big\|_{L^{2}([0,1]\times\mathbb{R}^{d})}\nonumber\\
    \leq&
    \|V\|_{L^{\infty}([0,1]\times\mathbb{R}^{d})}\Big\|e^{\mu|x+Rt(1-t)\mathbf{e}_{1}|^{2}-\frac{(1+\varepsilon+N\mu^{2}(\varepsilon_{0}+\varepsilon_{1}))t(1-t)}{16\mu}}h\Big\|_{L^{2}([0,1]\times\mathbb{R}^{d})}\nonumber\\
    &+Re^{\frac{\gamma}{\varepsilon}}\sup_{0\leq t\leq1}\|e^{\gamma|x|^{2}}u(t)\|_{L^{2}_{x}}\nonumber\\
    &+M^{-1}e^{\frac{\gamma R^{2}}{\varepsilon}}\Big\|e^{\gamma|x|^{2}}(|u|+|\nabla_{A}u|)\Big\|_{L^{2}(\mathbb{R}^{d}\times[\frac{1}{2R},1-\frac{1}{2R}])}. \label{reduction to 0}
\end{align}

Choose $R$ sufficiently large so that
\[
R\sqrt{\frac{\varepsilon}{\mu}}
\ge
2\|V\|_{L^{\infty}([0,1]\times\mathbb{R}^{d})}.
\]
Under this condition, the first term on the right-hand side of \eqref{reduction to 0} can be absorbed into the left-hand side. Moreover, since the quantity in \eqref{N gamma} is finite, the third term on the right-hand side of \eqref{reduction to 0} tends to zero as $M\to\infty$, for every fixed $R$.

Now choose a positive function $\mathfrak{f}=\mathfrak{f}(\varepsilon)$ such that
\[
\mathfrak{f}<(1-\varepsilon^{2})^2-\frac{N(\varepsilon_0+\varepsilon_1)}{4},
\qquad
\mathfrak{f}(\varepsilon)\to0^+
\quad\text{as }\varepsilon\to0.
\]
Since $h\equiv u$ on
\[
B_{\frac{\mathfrak{f}(\varepsilon)R}{8}}\times\Big[\frac{1-\varepsilon}{2},\frac{1+\varepsilon}{2}\Big],
\]
we have
\begin{align}
    &\mu|x+Rt(1-t)\textbf{e}_{1}|^{2}-\frac{(1+\varepsilon+N(\varepsilon_{0}+\varepsilon_{1})\mu^{2})R^{2}t(1-t)}{16\mu}\nonumber\\
    \geq&
    \frac{R^{2}}{16\mu}\Big\{16\mu^{2}\big[t^{2}(1-t)^{2}+2\frac{x}{R}\textbf{e}_{1}t(1-t)+\frac{|x|^{2}}{R^{2}}\big]-(1+\varepsilon+N(\varepsilon_{0}+\varepsilon_{1})\mu^{2})\frac{1}{4}\Big\}\nonumber\\
    \geq&
    \frac{R^{2}}{16\mu}\Big\{16\mu^{2}\Big[\frac{1}{16}(1-\varepsilon^{2})^{2}-2\frac{\mathfrak{f}(\varepsilon)}{8}\frac{1}{4}\Big]-(1+\varepsilon+N(\varepsilon_{0}+\varepsilon_{1})\mu^{2})\frac{1}{4}\Big\}\nonumber\\
    \geq&
    \frac{R^{2}}{16\mu}\Big\{\mu^{2}\Big[(1-\varepsilon^{2})^{2}-\mathfrak{f}(\varepsilon)-\frac{N(\varepsilon_{0}+\varepsilon_{1})}{4}\Big]-\frac{1+\varepsilon}{4}\Big\}\nonumber\\
    =&
    \frac{R^{2}}{16\mu}\Big\{\mu^{2}\Big[1-2\varepsilon^{2}+\varepsilon^{4}-\mathfrak{f}(\varepsilon)-\frac{N(\varepsilon_{0}+\varepsilon_{1})}{4}\Big]-\frac{1+\varepsilon}{4}\Big\}.
\end{align}

We now choose $\mu$ so that
\begin{align}
    \mu^{2}>
    \frac{1}{4}\frac{1+\varepsilon}{1-\frac{N(\varepsilon_{0}+\varepsilon_{1})}{4}-2\varepsilon^{2}+\varepsilon^{4}-\mathfrak{f}(\varepsilon)}
 \sim
    \frac{1}{4}\frac{1}{1-\frac{N(\varepsilon_{0}+\varepsilon_{1})}{4}}.
\end{align}
Since $\varepsilon_{0}$ and $\varepsilon_{1}$ are small enough so that
\[
\frac{N(\varepsilon_{0}+\varepsilon_{1})}{4}\ll1,
\]
it follows that
\begin{align}
    \mu>\frac{1}{2}\sqrt{\frac{1}{1-\frac{N(\varepsilon_{0}+\varepsilon_{1})}{4}}}
    =\frac{1}{2}\Big(\sum_{k=0}^\infty\big(\frac{N(\varepsilon_{0}+\varepsilon_{1})}{4}\big)^{k}\Big)^{\frac{1}{2}}
    =\frac{1}{2}+\frac{N(\varepsilon_0+\varepsilon_1)}{16}+O\Big(\frac{1}{2}(\frac{N(\varepsilon_0+\varepsilon_1)}{4})^{2}\Big). \label{formula of mu}
\end{align}
Thus, if
\[
\gamma>\frac{1}{2}+\frac{N(\varepsilon_{0}+\varepsilon_{1})}{8},
\]
then for
\[
R\geq2\sqrt{\frac{\mu}{\varepsilon}}\|V\|_{L^{\infty}([0,1]\times\mathbb{R}^{d})},
\]
we obtain
\begin{align}
    R\sqrt{\frac{\varepsilon}{\mu}}e^{R^{2}\mathcal{C}_{\varepsilon,\mu}}\|u\|_{L^{2}(B_{\frac{\mathfrak{f}(\varepsilon)R}{8}}\times[\frac{1-\varepsilon}{2},\frac{1+\varepsilon}{2}])}
    \leq
    Re^{\frac{\gamma}{\varepsilon}}\sup_{0\leq t\leq1}\|e^{\gamma|x|^{2}}u(t)\|_{L_{x}^{2}}, \label{reduction step at last}
\end{align}
where
\[
\mathcal{C}_{\varepsilon,\mu}=\frac{1}{16\mu}\Big\{\mu^{2}\Big[1-2\varepsilon^{2}+\varepsilon^{4}-\mathfrak{f}(\varepsilon)-\frac{N(\varepsilon_{0}+\varepsilon_{1})}{4}\Big]-\frac{1+\varepsilon}{4}\Big\}>0.
\]
Letting $R\to\infty$, we conclude that
\[
u\equiv0
\qquad
\text{on }
\mathbb{R}^{d}\times\Big[\frac{1-\varepsilon}{2},\frac{1+\varepsilon}{2}\Big].
\]

It remains to prove that $u(0)\equiv0$. Once this is known, the uniqueness statement $u\equiv0$ on $[0,1]$ follows from the Duhamel formula and Gronwall's inequality.

Using the condition \eqref{syl} in Hypothesis \ref{assum3} and following closely to Lemma 3.2 in \cite{BFGRV-JFA}, we obtain
\begin{equation}
    \|u(0)\|_{L^{2}}\sim_{\mathcal{C}_{V}}\|u(t)\|_{L^{2}}, \label{key relationship for last step}
\end{equation}
where
\[
\mathcal{C}_{V}=e^{\sup_{0\leq t\leq1}\|\operatorname{Im}V_2\|_{L^{\infty}}}.
\]
Moreover, \eqref{N gamma} yields
\begin{equation}
    N_{\gamma}\geq \|e^{\gamma|x|^{2}}u(t)\|_{L^{2}(\mathbb{R}^{d})}\geq e^{\gamma\frac{R^{2}}{64}}\|u\|_{L^{2}(\mathbb{R}^{d}\backslash B_{\frac{R}{8}})}
    \quad\Rightarrow\quad
    e^{-\gamma\frac{R^{2}}{64}}N_{\gamma}\geq\|u\|_{L^{2}(\mathbb{R}^{d}\backslash B_{\frac{R}{8}})}.
\end{equation}
Hence
\begin{align}
    \|u(t)\|_{L^{2}(\mathbb{R}^{d})}
    \leq
    \|u(t)\|_{L^{2}(B_{\frac{R}{8}})}+\|u(t)\|_{L^{2}(\mathbb{R}^{d}\backslash B_{\frac{R}{8}})}
    \leq
    \|u(t)\|_{L^{2}(B_{\frac{R}{8}})}+e^{-\gamma\frac{R^{2}}{64}}N_{\gamma}. \label{small domain control large domain}
\end{align}

Integrating over the time interval $\big[\frac{1-\varepsilon}{2},\frac{1+\varepsilon}{2}\big]$, and using \eqref{key relationship for last step}, \eqref{small domain control large domain}, and \eqref{reduction step at last}, we obtain
\begin{align}
   \varepsilon R\sqrt{\frac{\varepsilon}{\mu}}e^{R^{2}\mathcal{C}_{\varepsilon,\mu}}\|u(0)\|_{L^{2}(\mathbb{R}^{d})}
   \sim_{\mathcal{C}_{V}}
   &R\sqrt{\frac{\varepsilon}{\mu}}e^{R^{2}\mathcal{C}_{\varepsilon,\mu}}\|u(t)\|_{L^{2}([\frac{1-\varepsilon}{2},\frac{1+\varepsilon}{2}]\times\mathbb{R}^{d})}\nonumber\\
    \leq&
    R\sqrt{\frac{\varepsilon}{\mu}}e^{R^{2}\mathcal{C}_{\varepsilon,\mu}}\|u\|_{L^{2}([\frac{1-\varepsilon}{2},\frac{1+\varepsilon}{2}]\times B_{\frac{R}{8}})}\nonumber\\
    &+\varepsilon e^{-\gamma\frac{R^{2}}{64}}N_{\gamma}R\sqrt{\frac{\varepsilon}{\mu}}e^{R^{2}\mathcal{C}_{\varepsilon,\mu}}\nonumber\\
    \leq&
    Re^{\frac{\gamma}{\varepsilon}}\sup_{0\leq t\leq1}\|e^{\gamma|x|^{2}}u(t)\|_{L^{2}(\mathbb{R}^{d})}\nonumber\\
    &+\varepsilon e^{-\gamma\frac{R^{2}}{64}}N_{\gamma}R\sqrt{\frac{\varepsilon}{\mu}}e^{R^{2}\mathcal{C}_{\varepsilon,\mu}}.
\end{align}
Therefore,
\begin{align}
    \|u(0)\|_{L^{2}(\mathbb{R}^{d})}
    \leq
    \mathfrak{C}_{\mu,\varepsilon}e^{-R^{2}\mathcal{C}_{\mu,\varepsilon}}+ e^{-\gamma\frac{R^{2}}{64}}N_{\gamma},
    \qquad \forall R>0,
\end{align}
and letting $R\to\infty$ gives $u(0)\equiv0$.
\end{proof}

\section{The proof of Theorem \ref{parabolic}: the parabolic case}\label{sec:proof-parabolic}
In this section, we prove the unique continuation property for the variable-coefficient heat equation.  

To achieve this, we first establish a Carleman estimate for the case $a=1, b=0$. 
\begin{proposition}[Carleman estimate]
   Assume that $d\geq2$, the metric $G$ satisfies Hypothesis \ref{assm2}, and the magnetic field $B$ satisfies
 \begin{equation}
     \|(G\textbf{e}_1)^{\top}B\|_{L^{\infty}_{x}}\leq\varepsilon_1,\, (Gx)^{\top}B\in L_{x}^{\infty}.
 \end{equation}

Then, for any $\varepsilon>0$, $\mu>0$, $h=h(x,t)\in C_{0}^{\infty}(\mathbb{R}_{x}^{d}\times[0,1]_{t})$, and $R>\mathcal{C}_{3}\mu\varepsilon^{-\frac{1}{2}}$, the Carleman inequality holds:
    \begin{align}
&\hspace{5ex}R\sqrt{\frac{\varepsilon}{16\mu}}\Big\|e^{\mu|x+Rt(1-t)\textbf{e}_{1}|^{2}+\frac{R^{2}t(1-t)(1-2t)}{6}-\frac{(1+\varepsilon+N\mu^{2}\varepsilon_{1})R^{2}t(1-t)}{16\mu}}h\Big\|_{L^{2}([0,1]_{t}\times\mathbb{R}^{d})}\\
&\leq\Big\|e^{\mu|x+Rt(1-t)\textbf{e}_{1}|^{2}+\frac{R^{2}t(1-t)(1-2t)}{6}-\frac{(1+\varepsilon+N\mu^{2}\varepsilon_{1})R^{2}t(1-t)}{16\mu}}(\partial_{t}-\operatorname{div}_{A}(G\nabla_{A}))h\Big\|_{L^{2}([0,1]_{t}\times\mathbb{R}^{d})}\label{prop-Carleman for parabolic}
\end{align}

\end{proposition}

\begin{proof}
Substituting $\varphi=\mu|x+Rt(1-t)\textbf{e}_{1}|^{2}+\frac{R^{2}t(1-t)(1-2t)}{6}-\frac{(1+\varepsilon+N\mu^{2}\varepsilon_{1})R^{2}t(1-t)}{16\mu}$ into \eqref{S time derivative} and the commutator identity \eqref{formula of commutator}, we have
\begin{align}
    &\int_{\mathbb{R}^{d+1}}(\partial_{t}\mathcal{S}+[\mathcal{S},\mathcal{A}])f\bar{f}\,\dd x\dd t\\=&\int_{\mathbb{R}^{d+1}}2\operatorname{Im}(G\partial_{t}A\nabla_{A}f)\bar{f}\,\dd x\dd t+\int_{\mathbb{R}^{d+1}}8R(1-2t)\textbf{e}_{1}\cdot G(2\mu(x+Rt(1-t)\textbf{e}_{1}))|f|^{2}\,\dd x\dd t\\&+\int_{\mathbb{R}^{d+1}}2G\nabla\varphi\cdot\nabla(G\nabla\varphi\cdot\nabla\varphi)|f|^{2}\,\dd x\dd t\\
    &-\int_{\mathbb{R}^{d+1}}\mathcal{Q}^{2}\varphi|f|^{2}\,\dd x\dd t-4\operatorname{Im}\int_{\mathbb{R}^{d+1}}g_{jk}g_{lm}\partial_{l}\varphi(\partial_{k}A_{m}-\partial_{m}A_{k})f\overline{D_{j}f}\,\dd x\dd t\\&+4\int_{\mathbb{R}^{d+1}} g_{jk}g_{lm}\partial_{k}\partial_{l}\varphi D_{m}f\overline{D_{j}f}\,\dd x\dd t\nonumber\\
    &+\int_{\mathbb{R}^{d+1}}2g_{jk}\partial_{k}g_{lm}\partial_{l}\varphi D_{m}f\overline{D_{j}f}\,\dd x\dd t-\int_{\mathbb{R}^{d+1}}2g_{lm}\partial_{l}\varphi\partial_{m}g_{jk}D_{k}f\overline{D_{j}f}\,\dd x\dd t \\&+\int_{\mathbb{R}^{d+1}}2\partial_{j}g_{lm}\partial_{l}\varphi g_{jk}D_{k}f\overline{D_{m}f}\,\dd x\dd t+\int_{\mathbb{R}^{d+1}}\partial_{t}^{2}\varphi|f|^{2}\,\dd x\dd t\\
    =:&\sum_{\ell=1}^{10}Z_{j}.
\end{align}


Since $A$ is independent of $t$, 
\begin{align}
    Z_{1}=0.
\end{align}
Observing that 
\begin{equation}
    (1,0,\dots,0)\begin{pmatrix}
    0&0&\dots&0\\0&\tilde{g}_{22}(x')&\dots&\tilde{g}_{2n}(x')\\\dots&\dots&\dots&\dots\\0&\tilde{g}_{n2}(x')&\dots&\tilde{g}_{nn}(x')
\end{pmatrix}\Bigg[\begin{pmatrix}
    x_{1}\\x_{2}\\ \dots\\x_{n}
\end{pmatrix}+Rt(1-t)\begin{pmatrix}
    1\\0\\\dots\\0
\end{pmatrix}\Bigg]=0,
\end{equation}
we deduce that 
\begin{align}
    Z_{2}=&\int_{\mathbb{R}^{d+1}}8\mu R(1-2t)\textbf{e}_{1}G(2\mu(x+Rt(1-t)\textbf{e}_{1}))|f|^{2}\,\dd x\dd t
\\=&\int_{\mathbb{R}^{d+1}}8\mu R(1-2t)\textbf{e}_{1}(I+\widetilde{G})(2\mu(x+Rt(1-t)\textbf{e}_{1}))|f|^{2}\,\dd x\dd t
\\=&\int_{\mathbb{R}^{d+1}}8\mu R(1-2t)\textbf{e}_{1}I(2\mu(x+Rt(1-t)\textbf{e}_{1}))|f|^{2}\,\dd x\dd t
\\=&\int_{\mathbb{R}^{d+1}}16\mu^2 R(1-2t)\textbf{e}_{1}\cdot(x+Rt(1-t)\textbf{e}_{1})|f|^{2}\,\dd x\dd t.\label{combi0}
\end{align}

For $Z_3$, invoking 
 \eqref{decomposition of G} into $Z_3$, we have
 \begin{align}
     2G\nabla\varphi\cdot\nabla(G\nabla\varphi\cdot\nabla\varphi)=&2\nabla\varphi\cdot\nabla(|\nabla\varphi|^{2})+2\nabla\varphi\cdot\nabla(\tilde{G}\nabla\varphi\cdot\nabla\varphi)\nonumber\\
     &+2\tilde{G}\nabla\varphi\cdot\nabla(\tilde{G}\nabla\varphi\cdot\nabla\varphi)+2\tilde{G}\nabla\varphi\cdot\nabla(|\nabla\varphi|^{2}).
 \end{align}
 Thus
 \begin{align}
     Z_{3}=&\int_{\mathbb{R}^{d+1}}\Big[2\nabla\varphi\cdot\nabla(|\nabla\varphi|^{2})|f|^{2}+2\nabla\varphi\cdot\nabla(\tilde{G}\nabla\varphi\cdot\nabla\varphi)|f|^{2}+2\tilde{G}\nabla\varphi\cdot\nabla(\tilde{G}\nabla\varphi\cdot\nabla\varphi)|f|^{2}\nonumber\\
     &+2\tilde{G}\nabla\varphi\cdot\nabla(|\nabla\varphi|^{2})|f|^{2}\Big]\,\dd x\dd t\\
     :=&Z_{3}^{1}+Z_{3}^{2}+Z_{3}^{3}+Z_{3}^{4}
 \end{align}
Since we have treated $Z_{3}^{2},Z_{3}^{3},Z_{3}^{4}$ in Schr\"odinger case and estimates are also valid in parabolic case, it remains to handle $Z_{3}^{1}$. By definition, we have 
\begin{align}
Z_{3}^{1}=2\int_{\mathbb{R}^{d+1}}\nabla\varphi\cdot\nabla(|\nabla\varphi|^{2})|f|^{2}\,\dd x\dd t=&4\int_{\mathbb{R}^{d+1}}\nabla\varphi\cdot\operatorname{Hess}(\varphi)\nabla\varphi|f|^{2}\,\dd x\dd t\\
=&32\mu^{3}\int_{\mathbb{R}^{d+1}}|x+Rt(1-t)\textbf{e}_{1}|^{2}|f|^{2}\,\dd x\dd t.\label{combi1}
\end{align}
By direct computation, 
\begin{align*}
    \partial_{t}\varphi&=2\mu(x+Rt(1-t)\textbf{e}_{1})\cdot R(1-2t)\textbf{e}_{1}+\frac{R^{2}}{6}(1-6t+6t^{2})-\frac{(1+\varepsilon+N\mu^{2}\varepsilon_{1})R^{2}(1-2t)}{16\mu},\\
    \partial_{tt}\varphi&=2\mu|R(1-2t)\textbf{e}_{1}|^{2}-4\mu R(x+Rt(1-t)\textbf{e}_{1})\cdot\textbf{e}_{1}+R^{2}(-1+2t)+\frac{R^{2}(1+\varepsilon+N\mu^{2}\varepsilon_{1})}{8\mu}.
\end{align*}
    Combining the analysis above, we compute $Z_{10}$ as follows:
\begin{align}
Z_{10}=&\int_{\mathbb{R}^{d+1}}\Big(2\mu|R(1-2t)\textbf{e}_{1}|^{2}-4\mu R(x+Rt(1-t)\textbf{e}_{1})\cdot\textbf{e}_{1}+R^{2}(-1+2t)\\&+\frac{R^{2}(1+\varepsilon+N\mu^{2}\varepsilon_{1})}{8\mu}\Big)|f|^{2}\,\dd x\dd t.\label{combi2}
\end{align}
Hence, we have
\begin{align}
&Z_{2}+Z_{3}^{1}+Z_{10}\\=&
32\mu^{3}\int_{\mathbb{R}^{d+1}}\Big|x+Rt(1-t)\textbf{e}_{1}+\frac{(4\mu(1-2t)-1)R}{16\mu^{2}}\Big|^{2}|f|^{2}\,\dd x\dd t+\int_{\mathbb{R}^{d+1}}\frac{R^{2}(\varepsilon+N\mu^{2}\varepsilon_{1})}{8\mu}|f|^{2}\,\dd x\dd t.
\end{align}
Indeed, by completing the square:
\begin{align}
    &32\mu^{3}|x+Rt(1-t)\textbf{e}_{1}|^{2}+16\mu^2 R(1-2t)\textbf{e}_{1}\cdot(x+Rt(1-t)\textbf{e}_{1})+2\mu|R(1-2t)\textbf{e}_{1}|^{2}\\&-4\mu R(x+Rt(1-t)\textbf{e}_{1})\cdot\textbf{e}_{1}+R^{2}(2t-1)+\frac{R^{2}}{8\mu}\\
    =&32\mu^{3}\Big[|x+Rt(1-t)\textbf{e}_{1}|^{2}+\frac{4R\mu}{8\mu^{2}}(1-2t)\textbf{e}_{1}\cdot(x+Rt(1-t)\textbf{e}_{1})+\frac{16\mu^{2}|R(1-2t)|^{2}}{256\mu^{4}}\\
    &-\frac{R}{8\mu^{2}}(x+Rt(1-t)\textbf{e}_{1})\textbf{e}_{1}+\frac{8\mu}{256\mu^{4}}R^{2}(2t-1)+\frac{R^{2}}{256\mu^{4}}\Big]
    \end{align}
    \begin{align}
    =&32\mu^{3}\Big[|x+Rt(1-t)\textbf{e}_{1}|^{2}+\frac{4R\mu}{8\mu^{2}}(1-2t)\textbf{e}_{1}\cdot(x+Rt(1-t)\textbf{e}_{1})-\frac{R}{8\mu^{2}}(x+Rt(1-t)\textbf{e}_{1})\cdot\textbf{e}_{1}\\
    &+\big(\frac{4R\mu(1-2t)-R}{16\mu^{2}}\big)^{2}\Big]\\
    =&32\mu^{3}\Big[|x+Rt(1-t)\textbf{e}_{1}|^{2}+\frac{4R\mu(1-2t)-R}{8\mu^{2}}\textbf{e}_{1}\cdot(x+Rt(1-t)\textbf{e}_{1})+\big|\frac{4R\mu(1-2t)-R}{16\mu^{2}}\textbf{e}_{1}\big|^{2}\Big]\\
    =&32\mu^{3}\Big|x+Rt(1-t)\textbf{e}_{1}+\frac{\big(4\mu(1-2t)-1\big)R}{16\mu^{2}}\textbf{e}_{1}\Big|^{2}.
\end{align}
The estimate of $\int_{\mathbb{R}^{d+1}}\mathcal{Q}^{2}\varphi|f|^{2}\,\dd x\dd ts$ is similar to \eqref{bilaplace estimate}.
\begin{align}
    Z_{4}=&-\int_{\mathbb{R}^{d+1}}\mathcal{Q}^{2}\varphi|f|^{2}\,\dd x\dd t\\
    =&-2\mu\int_{\mathbb{R}^{d}}\Big(\partial_{l}g_{lm}\partial_{m}\partial_{k}g_{kj}(x_{j}+Rt(1-t)\delta_{1j}\textbf{e}_{j})+g_{lm}\partial_{l}\partial_{m}\partial_{k}g_{kj}(x_{j}+Rt(1-t)\delta_{1j}\textbf{e}_{j})\\&+\partial_{l}g_{lm}\partial_{k}g_{kj}\delta_{jm}\nonumber\\&+g_{lm}\partial_{m}\partial_{k}g_{kj}\delta_{jl}
    +g_{lm}\partial_{l}\partial_{k}g_{kj}\delta_{jm}
    +\partial_{l}g_{lm}\partial_{m}g_{kk}+g_{lm}\partial_{l}\partial_{m}g_{kk}  \Big)|f|^{2}\,\dd x\dd t\nonumber\\ =&-2\mu\int_{\mathbb{R}^{d}}\Big(\sum_{j,k,l,m\geq2}\Big[\partial_{l}g_{lm}\partial_{m}\partial_{k}g_{kj}x_{j}+g_{lm}\partial_{l}\partial_{m}\partial_{k}g_{kj}x_{j}\Big]+\partial_{l}g_{lm}\partial_{k}g_{kj}\delta_{jm}\nonumber\\&+g_{lm}\partial_{m}\partial_{k}g_{kj}\delta_{jl}
    +g_{lm}\partial_{l}\partial_{k}g_{kj}\delta_{jm}
    +\partial_{l}g_{lm}\partial_{m}g_{kk}+g_{lm}\partial_{l}\partial_{m}g_{kk}  \Big)|f|^{2}\,\dd x\dd t\nonumber \\
\geq&-2\Big(\sup_{x^\prime\in\mathbb{R}^{d-1}}(|x'||\nabla G(x')|)\|\nabla^{2}G\|_{L^{\infty}}+\sup_{x^\prime\in\mathbb{R}^{d-1}}|x'||\nabla^{3}G(x')|+2\|\nabla G\|^{2}_{L^{\infty}}\\&+2\Lambda\|\nabla^{2}G\|_{L^{\infty}}\Big)\mu\int_{\mathbb{R}^{d+1}}|f|^{2}\,\dd x\,\dd t.
\end{align}

For terms involving the magnetic field,  we apply the Cauchy–Schwarz inequality directly,
\begin{align}
    Z_{5}=&-8\mu\int_{\mathbb{R}^{d+1}}g_{jk}g_{lm}(x_{l}+Rt(1-t)\textbf{e}_{1})B_{mk}f\overline{D_{j}f}\,\dd x\dd t\\
    =&-8\mu\int_{\mathbb{R}^{d+1}}(Gx)^{\top} BGf\overline{\nabla_{A}f}\,\dd x\dd t-8\mu R\int_{\mathbb{R}^{d+1}}t(1-t)(G\textbf{e}_{1})^{\top}BGf\overline{\nabla_{A}f}\,\dd x\dd t\\
    \geq&-16\mu\|(Gx)^{\top} B\|^{2}_{L^{\infty}}\int_{\mathbb{R}^{d+1}}|f|^{2}\,\dd x\dd t-64\mu R^2\|(G\textbf{e}_{1})^{\top}B\|_{L^{\infty}}\int_{\mathbb{R}^{d+1}}|f|^{2}\,\dd x\dd t\\
    &-2\mu\int_{\mathbb{R}^{d+1}}|G\nabla_{A}f|^{2}\,\dd x\dd t.
\end{align}

Unlike \eqref{strange term}, in the parabolic case we obtain a simpler estimate for $Z_{7}+Z_{8}+Z_{9}$
\begin{align}
    &Z_{7}+Z_{8}+Z_{9}\\=&2\int_{\mathbb{R}^{d+1}} g_{jk}\partial_{k}g_{lm}\partial_{l}\varphi D_{m}f\overline{D_{j}f}\,\dd x\dd t-2\int_{\mathbb{R}^{d+1}}g_{lm}\partial_{l}\varphi\partial_{m}g_{jk}D_{k}f\overline{D_{j}f}\,\dd x\dd t\\&+
\int_{\mathbb{R}^{d+1}}2\partial_{j}g_{lm}\partial_{l}\varphi g_{jk}D_{k}f\overline{D_{m}f}\,\dd x\dd t\end{align}\begin{align} =& \sum_{j,k,l,m\geq2}\Bigg(2\int_{\mathbb{R}^{d+1}} g_{jk}\partial_{k}g_{lm}\partial_{l}\varphi D_{m}f\overline{D_{j}f}\,\dd x\dd t-2\int_{\mathbb{R}^{d+1}}g_{lm}\partial_{l}\varphi\partial_{m}g_{jk}D_{k}f\overline{D_{j}f}\,\dd x\dd t\\&+
\int_{\mathbb{R}^{d+1}}2\partial_{j}g_{lm}\partial_{l}\varphi g_{jk}D_{k}f\overline{D_{m}f}\,\dd x\dd t \Bigg)\\
\geq&-6\mu c_{d}\sup_{x'\in\mathbb{R}^{d-1}}|x'||\nabla \tilde{G}(x')|\int_{\mathbb{R}^{d+1}}|G\nabla_{A}f|^{2}\,\dd x\dd t.
\end{align}

Then, together all the estimates for $Z_{\ell}$, we have 
\begin{align}
     &\int_{\mathbb{R}^{d+1}}(\partial_{t}\mathcal{S}+[\mathcal{S},\mathcal{A}])f\bar{f}\,\dd x\dd t\\
     \geq&\Bigg(\frac{R^{2}\varepsilon+R^{2}N\mu^{2}\varepsilon_{1}}{8\mu}-64\mu R^{2}\|(G\textbf{e}_{1})^{\top}B\|_{L^{\infty}_{t,x}}-16\mu\|(Gx)^{\top} B\|_{L^{\infty}}^{2}\\&-2c_{d}\mu\Big(\sup_{x'\in\mathbb{R}^{d-1}}(|x'||\nabla G(x')|)\|\nabla^{2}G\|_{L^{\infty}}+\sup_{x'\in\mathbb{R}^{d-1}}|x'||\nabla^{3}G(x')|+2\|\nabla G\|^{2}_{L^{\infty}}\\&+2\Lambda\|\nabla^{2}G\|_{L^{\infty}}\Big)\Bigg)\int_{\mathbb{R}^{d+1}}|f|^{2}\,\dd x\dd t\\
     &+\Big(6\mu-6\mu c_{d}\sup_{x'\in\mathbb{R}^{d-1}}|x'||\nabla\tilde{G}(x')|\Big)\int_{\mathbb{R}^{d+1}}|G\nabla_{A}f|^{2}\,\dd x\dd t.
\end{align}
To obtain the desired  estimate, we need
\begin{align}
    \frac{R^{2}N\mu^{2}\varepsilon_{1}}{8\mu}\geq64\mu R^{2}\varepsilon_{1}\Rightarrow N\geq512,
\end{align}
\begin{align}
    &\frac{R^{2}\varepsilon}{16\mu}\geq16\mu\|(Gx)^{\top} B\|_{L^{\infty}}^{2}+2c_{d}\mu\Big(\sup_{x'\in\mathbb{R}^{d-1}}(|x'||\nabla G(x')|)\|\nabla^{2}G\|_{L^{\infty}}\\&+\sup_{x'\in\mathbb{R}^{d-1}}|x'||\nabla^{3}G(x')|+2\|\nabla G\|^{2}_{L^{\infty}}+2\Lambda\|\nabla^{2}G\|_{L^{\infty}}\Big).
\end{align}
Then the Carleman estimate reads as
\begin{align}
R\sqrt{\frac{\varepsilon}{16\mu}}\big\|e^{\varphi}h\big\|_{L^{2}([0,1]_{t}\times\mathbb{R}^{d})}\leq\big\|e^{\varphi}(\partial_{t}-\operatorname{div}_{A}(G\nabla_{A}))h\big\|_{L^{2}([0,1]_{t}\times\mathbb{R}^{d})}\label{Carleman for parabolic}
\end{align}
where 
\begin{equation}
    \varphi=\mu|x+Rt(1-t)\textbf{e}_{1}|^{2}+\frac{R^{2}t(1-t)(1-2t)}{6}-\frac{(1+\varepsilon+N\mu^{2}\varepsilon_{1})R^{2}t(1-t)}{16\mu}.
\end{equation}
Therefore, we get the required Carleman estimate and complete the proof.
\end{proof}
Now, we are in position to prove Theorem \ref{parabolic} via the Carleman estimate established above. 
\begin{proof}[Proof of Theorem \ref{parabolic}]
    We choose the test function
\begin{align}
h(x,t) = \theta_M(x)\,\zeta_R(t)\,u(x,t),\label{test-function}    
\end{align}
where $M, R>0$ are parameters to be chosen later. Here, $\theta_M \in C_0^\infty(\mathbb{R}^d)$ satisfies $\theta_M(x)=1$ for $|x|\le M$, $\theta_M(x)=0$ for $|x|\ge 2M$, and
\[
|\partial^k \theta_M(x)| \le \frac{C_k}{M^k},
\]
with $M>0$ to be chosen later. Meanwhile, $\zeta_R \in C_0^\infty(0,1)$ satisfies $0\le \zeta_R \le 1$, with
\[
\zeta_R(t)=1 \quad \text{for } t\in \Big[\tfrac{1}{R},\,1-\tfrac{1}{R}\Big],
\]
and
\[
\zeta_R(t)=0 \quad \text{for } t\in \Big[0,\tfrac{1}{2R}\Big]\cup\Big[1-\tfrac{1}{2R},1\Big],
\]
and enjoys the  bound
\[
|\partial^k \zeta_R(t)| \le C_k R^k.
\]

By calculation, we have 
\begin{align}
    \big(\partial_{t}-\operatorname{div}_{A}(G\nabla_{A}\cdot)\big)h=&\theta_{M}(x)(\partial_{t}\zeta_{R})u+\theta_{M}\zeta_{R}\partial_{t}u-\zeta_{R}(t)\partial_{j}g_{jk}\partial_{k}\theta_{M}u-\zeta_{R}(t)g_{jk}\partial_{j}\partial_{k}\theta_{M}u\nonumber\\&-\zeta_{R}(t)g_{jk}\partial_{k}\theta_{M}D_{j}u-\zeta_{R}(t)g_{jk}\partial_{j}\theta_{M}D_{k}u-\theta_{M}(x)\zeta_{R}(t)D_{j}(g_{jk}D_{k}u)\nonumber\\=&Vh+\theta_{M}(x)(\partial_{t}\zeta_{R})u-\zeta_{R}(t)\partial_{j}g_{jk}\partial_{k}\theta_{M}u-\zeta_{R}(t)g_{jk}\partial_{j}\partial_{k}\theta_{M}u\nonumber\\&-\zeta_{R}(t)g_{jk}\partial_{k}\theta_{M}D_{j}u-\zeta_{R}(t)g_{jk}\partial_{j}\theta_{M}D_{k}u.\label{parabolic ugly terms}
\end{align}
Taking $\mu\leq\frac{\gamma}{1+\varepsilon}$ for some $\varepsilon>0$, On the support of $\theta_M(x)\partial_t\zeta_R\cdot u$, by Young's inequality with $\varepsilon$ and $t(1-t)\sim\frac{1}{R}(1-\frac{1}{R})$, we obtain
\begin{align}
    &\mu|x+Rt(1-t)\textbf{e}_{1}|^{2}+\frac{R^{2}t(1-t)(1-2t)}{6}-\frac{(1+\varepsilon+N\mu^{2}\varepsilon_{1})R^{2}t(1-t)}{16\mu}\\
    \leq&\mu|x|^{2}+2\mu Rx\cdot \textbf{e}_{1}t(1-t)+\mu R^{2}t^{2}(1-t)^{2}+\frac{R}{6}+O(\frac{1}{R})\\ \leq &\mu|x|^{2}+\mu\varepsilon|x|^{2}+\mu\frac{1}{\varepsilon} R^{2}t^{2}(1-t)^{2}+\mu R^{2}t^{2}(1-t)^{2}+\frac{R}{6}+O(\frac{1}{R})\\ \sim&\mu(1+\varepsilon)|x|^{2}+\mu(\frac{1}{\varepsilon}+1)R^{2}\frac{1}{R^{2}}(1-\frac{1}{R})^{2}+\frac{R}{6}+O(\frac{1}{R})\\\sim&\mu(1+\varepsilon)|x|^{2}+\mu(\frac{1}{\varepsilon}+1)+\frac{R}{6}+O(\frac{1}{R})\\ \leq&\gamma|x|^{2}+\frac{\gamma}{\varepsilon}+\frac{R}{6}+O(\frac{1}{R}).\label{support3}
\end{align}
Further, we denote
\[
\begin{aligned}
\mathcal{I}_1 := {} & - \zeta_R(t)\,\partial_j g_{jk}\,\partial_k \theta_M\,u
- \zeta_R(t)\, g_{jk}\,\partial_j \partial_k \theta_M\,u  \\
& - \zeta_R(t)\, g_{jk}\,\partial_k \theta_M\, D_j u
- \zeta_R(t)\, g_{jk}\,\partial_j \theta_M\, D_k u,
\end{aligned}
\]
whose support satisfies
\[
\operatorname{supp} \mathcal{I}_1
\subset \Big[\tfrac{1}{2R},\,1-\tfrac{1}{2R}\Big]_t
\times \big( B_{2M} \setminus B_M \big).
\]
On the support of $\mathcal{I}_{1}$, it holds
\begin{align}
    &\mu|x+Rt(1-t)\textbf{e}_{1}|^{2}+\frac{R^{2}t(1-t)(1-2t)}{6}-\frac{R^{2}(1+\varepsilon+N\varepsilon_{1}\mu^{2})t(1-t)}{16\mu}\\
    \leq&\mu|x|^{2}+\mu\varepsilon|x|^{2}+\mu\frac{1}{\varepsilon}R^{2}t^{2}(1-t)^{2}+\mu R^{2}t^{2}(1-t)^{2}+\frac{R^{2}t(1-t)(1-2t)}{6}\\
    \leq&\mu(1+\varepsilon)|x|^{2}+\mu(1+\frac{1}{\varepsilon})R^{2}+R^{2}\\ 
    \leq&\gamma|x|^{2}+\frac{\gamma}{\varepsilon}R^{2}+R^{2}.\label{support4}
\end{align}

We now apply the Carleman estimate \eqref{Carleman for parabolic} to the test function \eqref{test-function} and it follows from \eqref{parabolic ugly terms},
\eqref{support3} and \eqref{support4}
\begin{align}
    &R\sqrt{\frac{\varepsilon}{16\mu}}\Big\|e^{\mu|x+Rt(1-t)\textbf{e}_{1}|^{2}+\frac{R^{2}t(1-t)(1-2t)}{6}-\frac{(1+\varepsilon+N\mu^{2}\varepsilon_{1})t(1-t)}{16\mu}}h(x,t)\Big\|_{L^{2}([0,1]\times\mathbb{R}^{d})}\\ \leq&\Big\|e^{\mu|x+Rt(1-t)\textbf{e}_{1}|^{2}+\frac{R^{2}t(1-t)(1-2t)}{6}-\frac{(1+\varepsilon+N\mu^{2}\varepsilon_{1})t(1-t)}{16\mu}}(\partial_{t}-\operatorname{div}_{A}(G\nabla_{A}\cdot))h\Big\|_{L^{2}([0,1]\times\mathbb{R}^{d})}\\
    \leq& \|V\|_{L^{\infty}([0,1]\times\mathbb{R}^{d})}\Big\|e^{\mu|x+Rt(1-t)\textbf{e}_{1}|^{2}+\frac{R^{2}t(1-t)(1-2t)}{6}-\frac{(1+\varepsilon+N\mu^{2}\varepsilon_{1})t(1-t)}{16\mu}}h\Big\|_{L^{2}([0,1]\times\mathbb{R}^{d})}\\&+Re^{\frac{\gamma}{\varepsilon}+\frac{R}{6}+1}\sup_{0\leq t\leq1}\|e^{\gamma|x|^{2}}u(t)\|_{L^{2}_{x}}\\&
    +M^{-1}e^{\frac{\gamma R^{2}}{\varepsilon}+R^{2}}\Big\|e^{\gamma|x|^{2}}(|u|+|\nabla_{A}u|)\Big\|_{L^{2}(\mathbb{R}^{d}\times[\frac{1}{2R},1-\frac{1}{2R}])}\label{parabolic last carleman}.
\end{align}

Choosing $R$ sufficiently large so that
\[
R\sqrt{\frac{\varepsilon}{16\mu}}
\ge 2\|V\|_{L^{\infty}([0,1]\times\mathbb{R}^d)},
\]
then the first term on the right-hand side of \eqref{parabolic last carleman}
can be absorbed into the left-hand side. Meanwhile, for any fixed $R$, the parabolic analogue of \eqref{N gamma} implies that the third term in \eqref{parabolic last carleman}  tends to 
 $0$ as $M \to \infty$.

Let $\mathfrak{f}(\varepsilon)<(1-\varepsilon^{2})^2-\frac{N\varepsilon_1}{4} $ be a real-valued function such that $\mathfrak{f}(\varepsilon)\to 0^{+}$ as $\varepsilon \to 0$.
Notice that $h \equiv u$ in
\[
\Big[\tfrac{1-\varepsilon}{2},\,\tfrac{1+\varepsilon}{2}\Big]
\times B_{\frac{\mathfrak{f}(\varepsilon)R}{8}},
\]
we obtain
\begin{align}
    &\mu|x+Rt(1-t)\textbf{e}_{1}|^{2}+\frac{R^{2}t(1-t)(1-2t)}{6}-\frac{(1+\varepsilon+N\varepsilon_{1}\mu^{2})R^{2}t(1-t)}{16\mu}\\
    \geq& \frac{R^{2}}{16\mu}\Big\{16\mu^{2}\big[t^{2}(1-t)^{2}+2\frac{x}{R}\textbf{e}_{1}t(1-t)+\frac{|x|^{2}}{R^{2}}\big]-(1+\varepsilon+N\varepsilon_{1}\mu^{2})\frac{1}{4}+\frac{8t(1-t)(1-2t)}{3}\mu\Big\}\\
    \geq&\frac{R^{2}}{16\mu}\Big\{16\mu^{2}[\frac{1}{16}(1-\varepsilon^{2})^{2}-2\frac{\mathfrak{f}(\varepsilon)}{8}\frac{1}{4}]-(1+\varepsilon+N\varepsilon_{1}\mu^{2})\frac{1}{4}-\frac{2(\varepsilon-\varepsilon^{3})}{3}\mu\Big\}\\
    \geq&\frac{R^{2}}{16\mu}\Big\{\mu^{2}\big[(1-\varepsilon^{2})^{2}-\mathfrak{f}(\varepsilon)-\frac{1}{4}N\varepsilon_{1}\big]-\frac{1+\varepsilon}{4}-\frac{2(\varepsilon-\varepsilon^{3})}{3}\mu\Big\}\\
    =&\frac{R^{2}}{16\mu}\Big\{\mu^{2}\big[1-2\varepsilon^{2}+\varepsilon^{4}-\mathfrak{f}(\varepsilon)-\varepsilon-\frac{1}{4}N\varepsilon_{1}\big]-\frac{1+\varepsilon}{4}-\frac{\varepsilon}{4}(\frac{2(\varepsilon-\varepsilon^{3})}{3})^{2}\Big\}.
\end{align}

Let $\mu$ satisfy
\[
\mu^{2}\bigl[1 - 2\varepsilon^{2} + \varepsilon^{4}-\varepsilon - \mathfrak{f}(\varepsilon)-\frac{1}{4}N\varepsilon_{1}\bigr]
- \frac{1+\varepsilon+\varepsilon(\frac{2(\varepsilon-\varepsilon^3)}{3})^{2}}{4} > 0,
\]
which implies 
\begin{align}
    \mu^{2}>\frac{1+\varepsilon+\varepsilon(\frac{2(\varepsilon-\varepsilon^3)}{3})^2}{4(1-\frac{N\varepsilon_{1}}{4}-\varepsilon-2\varepsilon^2+\varepsilon^4-\mathfrak{f}(\varepsilon))}>\frac{1}{4(1-\frac{N\varepsilon_{1}}{4})}.
\end{align}
Thus, $\mu$ satisfies
\begin{align}
    \mu>\frac{1}{2}\sqrt{\frac{1}{1-\frac{N\varepsilon_{1}}{4}}}=\frac{1}{2}\Big(\sum_{k=0}^\infty\big(\frac{N\varepsilon_1}{4}\big)^k\Big)^{\frac{1}{2}}=\frac{1}{2}+\frac{N\varepsilon_{1}}{16}+O\big((\frac{1}{2}(\frac{N\varepsilon_1}{4})^{2}\big)
\end{align}


As a consequence, we obtain
\begin{align}
\frac{R}{8}\sqrt{\frac{\varepsilon}{\mu}}
\,e^{\mathcal{C}_{\mu,\varepsilon}R^{2}}
\|u\|_{L^{2}\big([\tfrac{1-\varepsilon}{2},\,\tfrac{1+\varepsilon}{2}]
\times B_{\mathfrak{f}(\varepsilon)R}\big)}
\le
R e^{\frac{\gamma}{\varepsilon}+R+1}
\sup_{t\in[0,1]}\|e^{\gamma |x|^{2}}u(t)\|_{L^{2}} .
\end{align}
By taking $R$ sufficiently large, we obtain
\[
u \equiv 0
\quad \text{in } \mathbb{R}^{d}\times
\Big[\tfrac{1-\varepsilon}{2},\,\tfrac{1+\varepsilon}{2}\Big].
\]
It remains to show that $u(0)=0$.  The basic energy estimates for parabolic equation gives
\begin{equation}
\frac{1}{2}\frac{\dd}{\dd t}\|u(t)\|^2=-\int_{\mathbb{R}^d}G\nabla_{A}u\cdot\overline{\nabla_{A}u}\,\dd{x}+\int_{\mathbb{R}^{d}}\operatorname{Re}V(x,t)|u(x,t)|^2\,\dd{x}\leq \|\operatorname{Re}V(\cdot,t)\|_{L_x^\infty}\|u(t)\|^2_{L_{x}^{2}}.
\end{equation}
Since $u(\frac12)=0$, by Gronwall inequality
\begin{equation}
\|u(t)\|_{L^2}^{2}\leq \exp{2\int_{\frac{1}{2}}^{t}\|\operatorname{Re}V(x,s)\|_{L_x^\infty}\,\dd s}\|u(\frac{1}{2})\|_{L^2}^2,   
\end{equation}
we get $u\equiv0$ in $\mathbb{R}^{d}\times[\frac{1-\varepsilon}{2},1]$.
Thus, we obtain 
\begin{equation}
    H(1)=\|e^{\gamma|x|^2}u(1)\|_{L^2}^2=0.
\end{equation} To get the vanishing of $u$ in $(0,1)$, 
we define $H_\delta(t)=H(t)+\delta$. Repeating the Lemma \ref{abstract lemma1} in Section 2, we can get the logarithmic convexity inequality
\begin{equation}
    H_\delta(t)\leq e^{CM}H_\delta(0)^{1-t}H_\delta(1)^{t} ,\qquad t\in(0,1).
\end{equation}
Let $\delta$ tend  to
we get $H(t)=0$ for $(0,\frac{1-\varepsilon}{2})$, which implies that $u\equiv0$ in this interval. Since $u\in C([0,1]; L^2(\mathbb{R}^d))$ and $u=0$ in $(0,1)$, it holds that $u(0)=0$.  
Hence, we finish the proof of Theorem \ref{parabolic}.

\end{proof}

\section*{Acknowledgements}
The authors would like to thank Ricardo Machado and Luz Roncal for their helpful comments and suggestions on the proof of the main theorem. L. Fanelli and Y. Wang are partially supported by the Basque Government through the BERC 2022--2025 program and by the research project PID2024-155550NB-100 funded by MICIU/AEI/10.13039/501100011033 and FEDER/EU.\,\,
L. Fanelli is also supported by the project IT1875-26 funded by the Basque Government.\,\,
Y. Wang is also supported by a Juan de la Cierva fellowship funded by MICIU/AEI/10.13039/501100011033, under Grant JDC2024-053285-I.\,\,
J. Zheng was supported by  National key R\&D program of China: 2021YFA1002500 and  NSFC Grant 12271051.

\end{document}